\documentclass[11pt,reqno]{amsart}
\pdfoutput=1
\usepackage[active]{srcltx}

\usepackage[titletoc,title]{appendix}

\usepackage{amssymb,amsmath,tikz}
\usepackage{bbm}
\usepackage{enumerate}
\usepackage{graphicx}
\usepackage{xcolor} 
\usepackage{geometry}
\usepackage{amsfonts,mathtools}
\usepackage{slashed}
\usepackage{caption}
\usepackage{subcaption}
\usepackage{physics}
\usepackage{txfonts}
\usepackage{hyperref}
\usepackage{wrapfig}

\usepackage{amsthm,stackengine}

\newcommand{\mathsym}[1]{{}}
\newcommand{\unicode}[1]{{}}

\newtheorem{thm}{Theorem}[section]
\newtheorem{lem}[thm]{Lemma}
\newtheorem{cor}[thm]{Corollary}
\newtheorem{prop}[thm]{Proposition}
\newtheorem{defn}[thm]{Definition}
\newtheorem{remark}[thm]{Remark}

\def\Dio{\mathrm{Dioph}}

\def\la{\langle}
\def\ra{\rangle}
\def\Re{{\,\rm Re\,}}
\def\Im{{\,\rm Im\,}}

\def\eps{{\varepsilon}}
\def\les{\lesssim}
\def\calS{\mathcal S}
\def\beeq{\begin{equation}}
\def\eneq{\end{equation}}
\def\Q{\mathbb{Q}}
\def\R{{\mathbb R}}
\def\C{{\mathbb C}}
\def\Pbb{\mathbb{P}}
\def\Exp{\mathbb{E}}
\def\tilde{\widetilde}

\def\calD{\mathcal{D}}

\def\nn{\nonumber}
\def\trace{\mathrm{trace}}
\def\Z{\mathbb{Z}}
\def\Exp{\mathbb{E}}

\def\rank{\mathrm{rank}}

\def\calE{\mathcal{E}}

\def\calN{\mathcal{N}}
\def\calG{\mathcal{G}}
\def\calF{\mathcal{F}} 
\newcommand{\EQ}[1]{\begin{equation}  \begin{split} #1 \end{split} \end{equation} }
\def\dist{\mathrm{dist}}
\def\tor{\mathbb{T}}
\def\spec{\mathrm{spec}}
\def\one{\mathbbm{1}}
\def\calB{\mathcal{B}}
\def\cB{\calB}
\def\const{\mathrm{const}}
\def\lspan{\mathrm{span}}
\def\bad{\calB}
\def\Car{{\rm Car}}
\def\calA{\mathcal{A}}

\newcommand{\bbC}{\mathbb{C}}
\newcommand{\bbD}{\mathbb{D}}

\def\Compl{\bbC}

\newcommand{\bbZ}{\mathbb{Z}}
\newcommand{\bbR}{\mathbb{R}}

\newcommand\Span[1]{\left\langle #1 \right\rangle}

\newcommand{\BMO}{{\rm{BMO}}}

\renewcommand{\emptyset}{\O}

\providecommand{\abs}[1]{\lvert #1 \rvert}

\providecommand{\,}{\thinspace}
\providecommand{\oP}[1]{\lvert\lvert #1 \rvert\rvert}

\def\diam{\mathrm{diam}}

\numberwithin{equation}{section}

\begin{document}

\title[Anderson localization, multiscale techniques:  I]{An introduction to multiscale techniques in the theory of Anderson localization. Part I. }

\author[W. Schlag]{Wilhelm Schlag}
\address{Department of Mathematics \\ Yale University \\ New Haven, CT 06511, USA}
\email{wilhelm.schlag@yale.edu}

\thanks{
The author was partially supported by NSF grant DMS-1902691. He thanks Adam Black, Yakir Forman, and Tom VandenBoom for their assistance during the preparation of these notes, and Gong Chen and David Damanik for their comments on a preliminary version. The author is indebted to Massimilano Berti at SISSA, Trieste, Italy, for inviting him to teach a mini-course at SISSA in March of 2020 on Anderson localization. Due to the COVID-19 pandemic 
these lectures did not materialize, but they still provided the impetus for the writing of these notes. 
}

\begin{abstract}
        These lectures present some basic ideas and techniques in the spectral analysis of lattice Schr\"odinger operators with disordered potentials. In contrast to the classical Anderson tight binding model, the randomness is also allowed to possess only finitely many degrees of freedom. This refers to dynamically defined potentials, i.e., those given by evaluating a function along an orbit of some ergodic transformation (or of several commuting such transformations on higher-dimensional lattices).  Classical localization theorems by Fr\"ohlich--Spencer for large disorders are presented, both for random potentials in all dimensions, as well as even quasi-periodic ones on the line. After providing the needed background on subharmonic functions, we then discuss the Bourgain-Goldstein theorem on localization for quasiperiodic Schr\"odinger cocycles assuming positive Lyapunov exponents.  
\end{abstract}

\maketitle 

\tableofcontents

\section{Introduction}

In the 1950s Phil Anderson studied random operators of the form 
\[
 H = \Delta_{\Z^d} + \lambda V
\]
where $\Delta_{\Z^d}$ is the discrete Laplacian on the $d$-dimensional lattice and $V:\Z^d \to\R$ a random field with i.i.d.\ components, and a real parameter $\lambda$. His pioneering work suggested by physical arguments that for large $\lambda$, with probability~$1$, a typical realization of the random operator $H$ exhibits exponentially decaying eigenfunctions which form a basis of $\ell^2(\Z^d)$. This is referred to as {\em Anderson localizaton} (AL). It is in stark contrast to periodic $V$ for which the spectrum is absolutely continuous (a.c.) with a distorted Fourier basis of Bloch-Floquet waves, see~\cite{Kuch, MW}. Furthermore, and most importantly, Anderson found a phase transition in dimensions three and higher, leading to the a.s.\ presence of a.c.\ spectrum for small $\lambda$. This famous {\em extended states problem} is still not understood. 

On the other hand, a large mathematical literature now exists dealing with Anderson localization and its ramifications (density of states, Poisson behavior of eigenvalues). This introduction is not meant as a broad introduction to this field, for which we refer the reader to the recent textbook~\cite{AW}, as well as the more classical treaties~\cite{BouLa, FigP, CarLac} and the forthcoming  texts~\cite{DF1, DF2}.  
Our focus here is with the body of techniques commonly referred to as {\em multiscale}. They are all based on some form of induction on scales, and are reminiscent of KAM arguments. 

This approach is effective  both in random models, as well as those  with {\em deterministic} potentials, which refers to $V(n)$ being fixed by a finite number of parameters. For example, Harper's model on $\Z$ is given by $V(n)= \cos(2\pi(n\omega+ x))$ with irrational $\omega$ and $x\in\R/\Z$. The only stochastic parameter is this choice of~$x$. The Harper operator, which is also known as almost Mathieu  operator, as well as more general quasi-periodic operators, exhibit a rich and subtle spectral theory, see for example the survey~\cite{JM}. 

Bourgain's book~\cite{Bou} contains a wealth of material on a wide class of stochastic Schr\"odinger operators with deterministic potentials. An important basic assumption in that book is the analyticity of the generating function, i.e., if $V(n)=F(T^n x)$ for some ergodic transformation $T$ on a torus, then $F$ is assumed to be analytic or a trigonometric polynomial. The analyticity allows for the use of subharmonic functions. These are relevant for {\em large deviation theorems}, which in turn hinge on some Cartan type lower bound for subharmonic functions. This first part of the notes can be seen as a companion to Bourgain's book~\cite{Bou} but only up to Chapter~12. The plan for the second part of this introduction is to  focus  on the matrix-valued Cartan theorem of~\cite{BouGS2}, and the higher-dimensional theory as in~\cite{Bou2}, with applications. This will then hopefully serve to make   Chapters~14 through~19 of~\cite{Bou} more accessible.

\section{Polynomially bounded Fourier basis} 

In this section we establish the following widely known fact concerning the Fourier transform associated with a Schr\"odinger operator. It is a particular case of a more general theory, see the text~\cite{Ber} and the survey~\cite{Sim}. Results of this type go by the name of {\em Shnol theorem}. We follow the argument in~\cite{Far}. Throughout, the discrete Laplacian on $\Z^d$ is defined as the sum over nearest neighbors, i.e., 
\EQ{\label{eq:Lap}
(\Delta f)(x) = \sum_{\pm} \sum_{j=1}^d  f(x\pm e_j)\quad\forall\; x\in \Z^d
}
where $e_j$ are the standard coordinate vectors. If $\calF:\ell^2(\Z^d)\to L^2(\tor^d)$ denotes the Fourier transform, then 
\EQ{\label{eq:LapF}
(\calF\circ \Delta\circ \calF^{-1} f)(\theta) =  m(\theta)f(\theta)=2\sum_{j=1}^d \cos(2\pi \theta_j) f(\theta)
}
and the spectrum satisfies $\spec(\Delta)=[-2d,2d]$. The Laplacian~\eqref{eq:Lap} differs from the more customary $-\tilde \Delta = \nabla^*\nabla$ where $(\nabla f)(x)=\{f(x+e_j)-f(x)\}_{j=1}^d$, by a diagonal term: $ -\tilde \Delta = -\Delta +2d$. 

\begin{thm}\label{thm:Ber}
Consider $H=\Delta+V$ as a bounded operator on $\ell^2(\Z^d)$, with $V\in\ell^\infty(\Z^d)$ real-valued and acting by multiplication. Fix $\sigma>\frac{d}{2}$. Then for almost every $E\in\R$ with respect to the spectral measure\footnote{I.e., up to a set of measure $0$ relative to the spectral measure of $H$.} of~$H$ there exists $\psi:\Z^d\to\R$ not identically vanishing with $H\psi= E\psi$ and $|\psi(n)|\le C(d,\sigma,E) \langle n\rangle^\sigma$ for all $n\in\Z^d$. 
\end{thm}
\begin{proof}
Take any $z\in \R\setminus \Sigma$, where $\Sigma$ is the spectrum of $H$ (we take $z\in\R$ for simplicity). 
By the  Combes-Thomas estimate the kernel of the Green function $(H-z)^{-1}(x,y)$ has exponential decay: there exist positive constants $C,\beta$ so that 
\EQ{\label{eq:CT}
|(H-z)^{-1}(x,y)|\le C \exp(-\beta|x-y|)\qquad\forall\; x,y\in\Z^d
}
To see this, let $(M_{\beta,j} f)(n)= e^{\beta n_j} f(n)$ and compute $$M_{\beta,j}^{-1} \circ (H-z) \circ M_{\beta,j} = H-z +  S_{\beta,j}$$ with $\|S_{\beta,j}\|_{2\to2} \le C|\beta|$ uniformly in $|\beta|\le 1$. Hence, 
$$M_{\beta,j}^{-1} \circ (H-z)^{-1} \circ M_{\beta,j} = (H-z)^{-1}\circ(I +  S_{\beta,j}(H-z)^{-1})^{-1}$$
where the inverse on the right exists by a Neumann series as long as in the operator norm $$\|S_{\beta,j}(H-z)^{-1}\|<1$$ which holds if $|\beta|\dist(z,\Sigma)\le c$, some small constant. In particular, 
\[
|\langle \delta_y, (H-z)^{-1} \delta_x\rangle|= |(H-z)^{-1}(x,y)|\le Ce^{-\beta(x_j-y_j)}
\]
Since the sign of $\beta$ and the choice of $j$ are arbitrary, \eqref{eq:CT} follows. 

Let $w_\sigma(x):=\la x\ra^{-\sigma}$ on $\Z^d$. Fixing any $\sigma>\frac{d}{2}$ so that $\langle x\rangle^{-\sigma}\in\ell^2(\Z^d)$, the Combes-Thomas bound~\eqref{eq:CT} implies that  $(H-z)^{-1}(x,y) w_\sigma(y) \in \ell^2(\Z^d\times\Z^d)$ whence
\[
(H-z)^{-1}: \ell^2_\sigma(\Z^d) \to \ell^2(\Z^d)
\]
is a Hilbert-Schmidt operator. Here $ \ell^2_\sigma(\Z^d):= w_{-\sigma} \ell^2(\Z^d)$. By the spectral theorem there exists a unitary $U:\ell^2(\Z^d)\to L^2(X,\mu)$ where $\mu$ is a $\sigma$-finite measure, and $\phi\in L^\infty(X)$ real-valued, with $UHf=\phi\, Uf$ for all $f\in\ell^2(\Z^d)$. The $\mu$-essential range of $\phi$ equals~$\Sigma$. The composition 
\[
T:= U(H-z)^{-1}=(\phi-z)^{-1}U : \ell^2_\sigma(\Z^d) \to  L^2(X,\mu)
\]
is Hilbert-Schmidt, whence by the standard kernel representation of such operators, for every $n\in\Z^d$ there exists $K(\cdot,n)\in L^2(X,\mu)$ with 
\[
\int\limits_X \sum_{n\in\Z^d} |K(x,n)|^2 \, w_\sigma^2(n)\, \mu(dx) <\infty
\]
and 
\[
Tf(x) = \sum_{n\in\Z^d} K(x,n) f(n) \qquad\forall \; f\in \ell^2_\sigma(\Z^d) 
\]
The series converges in $ L^2(X,\mu)$.  Define $\psi_x(n):=(\phi(x)-z)\overline{K(x,n)}$.  Then  for all $f\in \ell^2_\sigma(\Z^d)$,  and $ \mu-\text{a.e.}$ $x$, 
\EQ{\label{eq:U psix}
(Uf)(x) &= (\phi(x)-z)(Tf)(x) =  \sum_{n\in\Z^d} (\phi(x)-z) K(x,n) f(n) = \langle \psi_x,f\rangle_{\ell^2(\Z^d)} 
}
By the preceding $\psi_x\in (\ell^2_\sigma(\Z^d))^* = \ell^2_{-\sigma}(\Z^d)$ for $\mu$-a.e.~$x$.  Next, we claim that a.e.\ in~$x$ and in the point-wise sense on $\Z^d$
\EQ{\label{eq:claim} 
H\psi_x = \phi(x)\psi_x
}
as well as $\psi_x\not\equiv0$. 
Take $f$ on the lattice with finite support. Then $Hf$  has finite support and by~\eqref{eq:U psix}
\EQ{\nn
\langle H \psi_x,f\rangle=  \langle  \psi_x, Hf\rangle = (UHf)(x) & = \phi(x) (Uf)(x)  
= \langle \phi(x) \psi_x,f\rangle
}
all scalar products in $\ell^2(\Z^d)$, and for $\mu$-a.e.~$x$. It follows that $H \psi_x = \phi(x)\psi_x$ whence~\eqref{eq:claim}.  Now suppose $\psi_x\equiv0$ for all $x\in\calS\subset X$, $\mu(\calS)>0$. Then for all $f\in\ell^2(\Z^d)$, this implies that  $Uf=0$ $\mu$-a.e.\ on~$\calS$, and 
\[
0=\langle Uf,\chi_{\calS}\rangle_{L^2(\mu)} = \langle f, U^*\chi_{\calS}\rangle_{\ell^2(\Z^d)} 
\]
But this means that $U^*\chi_\calS=0$ which contradicts $\|U^*\chi_\calS\|_2^2 =\mu(\calS)>0$.  To summarize, there is $\calN\subset X$ with $\mu(\calN)=0$ so that for all $E\in \calG:=\phi(X\setminus \calN)$ the equation $H\psi=E\psi$ has a nonzero solution $\psi\in\ell^2_{-\sigma}(\Z^d)$. We claim that $\calE(\R\setminus\calG)=0$ where $\calE$ is the spectral resolution of~$H$, i.e., a projection-valued Borel measure with $H=\int\lambda\, \calE(d\lambda)$. But for any Borel set $B\subset\R$, $U\calE(B)\, U^* g = \chi_{\phi^{-1}(B)}g$ on all $g\in L^2(X,\mu)$. Since $\phi^{-1} (\R\setminus \calG) \subset \calN$ is a $\mu$-nullset, we conclude that $U\calE(\R\setminus \calG)\, U^*=0$ whence $\calE(\R\setminus \calG)=0$. So $H\psi=E\psi$ has nonzero solution $\psi\in\ell^2_{-\sigma}(\Z^d)$ {\em spectrally a.e.}
\end{proof}

This proof applies much more generally and the discrete Laplacian is used only sparingly. For example, it can be replaced with a self-adjoint T\"oplitz operator with exponential off-diagonal decay. 

\section{The Anderson model and localization}
\label{sec:AL FS}

Let 
\EQ{\label{eq:Anderson}
H_\omega=\Delta+V_\omega
} on $\ell^2(\Z^d)$ where $V_\omega$ is a diagonal operator given by i.i.d.~random variables at each lattice site $n\in\Z^d$. The single site distribution refers to the law of 
$V_\omega(0)$ which we assume to be a.s.~bounded. Then $H_\omega$ is a.s.\ a bounded operator. 
The notation is based on an underlying probability space $(\Omega,\Pbb)$, with $\omega\in\Omega$. 
If $d=1$ we may consider a more general model with a potential generated by any ergodic dynamical system. Thus, let  $T:X\to X$ be measure preserving, invertible,  and ergodic on~$X$ relative to the probability  measure~$\nu$. Then set $V_x(n):=f(T^n x)$ where $f:X\to \R$ is measurable and $\nu$-essentially bounded, and define $H_x:=\Delta + V_x$ in $\ell^2(\Z)$. For this model we shall now demonstrate that the spectrum of $H_x$ is deterministic by ergodicity. 

\begin{prop}
\label{prop:spec determ} 
There exists a compact set $\Sigma\subset\R$ so that $\spec(H_x)=\Sigma$ for $\nu$-almost every $x\in X$. 
\end{prop}
\begin{proof}
One has the conjugation 
\EQ{\label{eq:conj}
H_{Tx}= U^{-1}\circ H_x \circ U
} with the right-shift $U:\ell^2(\Z)\to \ell^2(\Z)$ and so also $\calE_{Tx}= U^{-1}\circ \calE_x \circ U$ with the spectral resolution $\calE_x$ of $H_x$. Recall that $H_x = \int\lambda\, \calE_x(d\lambda)$ in a suitable sense.  For any Borel set $B\subset\R$, 
\[
E(B) := \{ x\in X\:|\: \rank(\calE_x(B))>0\} 
\]
is an invariant set, i.e., $E(B)=T^{-1}(E(B))$ whence $\nu(E(B))=0$ or $\nu(E(B))=1$. Define 
\[
\Sigma=\R\setminus \bigcup_{a<b } \{ (a,b) \:|\: \nu(E((a,b))=0\}
\]
where the union is over rational $a,b$. If any interval $I\subset \R$ intersects $\Sigma$, then $\rank(\calE_x(I))>0$ for $\nu$-a.e.\ $x\in X$. Hence also $I\cap \spec(H_x)\ne \emptyset$ for a.e.~$x$. 
\end{proof}

A finer description of the set $\Sigma$ can be obtained by a similar argument. The subscripts $ac, sc, pp$ stand for, respectively, absolutely continuous, singular continuous, and pure point. 

\begin{prop}
\label{prop:spec parts determ} 
There exist  compact subsets $\Sigma_{ac}$, $\Sigma_{sc}$, and $\Sigma_{pp}$ of $\Sigma$ such that $\Sigma= \Sigma_{ac}\cup \Sigma_{sc}\cup\Sigma_{pp}$ (not necessarily disjoint) such that 
for any Borel set $B$ with $B\cap \Sigma_{ac}\ne\emptyset$ the following holds: for a.e.~$x\in X$ there exists $f\in\ell^2(\Z)$ so that $\mu(A):=\langle \calE_x( A\cap B)f,f\rangle$ defined on Borel sets $A$ is an absolutely continuous probability measure.  Analogous statements hold for the singular continuous, and pure point (atomic) parts. 
\end{prop}
\begin{proof}
We define, with the union being over rationals, 
\EQ{\nn 
\Sigma_{ac} &=\R\setminus \bigcup_{a<b } \{ (a,b) \:|\: \forall \; f\in\ell^2(\Z), \; A\mapsto \langle \calE_x((a,b)\cap A)f,f\rangle  \\
& \qquad \nu-a.s. \text{\ \ has no absolutely continuous component}\}
}
where the latter property refers to the Lebesgue decomposition. We adopt the convention that the $0$ measure has no absolutely continuous component  (as well as no singular component). 
By ergodicity and the conjugacy of $H_x$ and $H_{Tx}$, respectively, by the shift, the set 
\EQ{\nn 
Y(a,b)&:= \{x\in X\:|\:  \exists \; f\in\ell^2(\Z), \;  A\mapsto \langle \calE_x((a,b)\cap A)f,f\rangle  \\
& \qquad \text{\ \ has an absolutely continuous component}\}
}
is $T$-invariant and thus $\nu(Y(a,b))=0$ or~$\nu(Y(a,b))=1$. Hence
\EQ{\nn 
\Sigma_{ac} &=\R\setminus \bigcup_{a<b } \{ (a,b) \:|\:  a,b\in\Q,\; \nu(Y(a,b))=0  \}
}
Now suppose $B\cap \Sigma_{ac}\ne\emptyset$. Without loss of generality, $B\subset \Sigma_{ac}$. If $B\cap (a,b)\ne\emptyset$ with $a,b\in\Q$, then $\nu(Y(a,b))=1$. 
Thus $\nu$-a.s., $A\mapsto \langle \calE_x((a,b)\cap A)f,f\rangle$ is absolutely continuous for some $f$. We used here that we may pass from the existence of an absolutely continuous component to
purely absolutely continuous by projecting $f$ on the a.c.~subspace of $H_x$. The claim of having a probability measure is obtained by normalization. The proofs for the singular parts is identical. 
\end{proof}

These arguments make no use of the Laplacian and therefore apply to the diagonal operator given by multiplication by the potential $V_x$. In that case the eigenvalues are $\{ V_x(n)=f(T^n x)\:|\: n\in\Z\}$ and the closure of this set is deterministic and equals $\Sigma_{pp}$. Moreover,  $\Sigma_{ac}=\Sigma_{sc}=\emptyset$. 

Propositions~\ref{prop:spec determ}  and~\ref{prop:spec parts determ} apply as stated to the random model $H_\omega$ from above, as the reader is invited to explore. In fact,  on $\ell^2(\Z^d)$ we may consider $d$ measure preserving, invertible, commuting  transformations  $T_j:X\to X$ with the following ergodicity property: if $A\subset X$ is invariant under {\em all} $T_j$, then $\nu(A)=0$ or $\nu(A)=1$.  Then the previous two propositions apply to the operator $H_x := \Delta_{\Z^d} + V_x$ with $V_x(\bar n)=f(T_1^{n_1}\circ T_2^{n_2} \circ \cdots \circ T_d^{n_d} x)$ for any $\bar n=(n_1,\ldots,n_d)\in\Z^d$ with essentially the same proofs. See~\cite{Kir,FigP} for a systematic development of the spectral theory of ergodic families of operators. 

For the random model, which is the original Anderson model, we can now explicitly compute the almost sure spectrum $\Sigma$ in Proposition~\ref{prop:spec determ}. Recall that we are assuming bounded support of the single site distribution. 

\begin{prop}
\label{prop:Sigma}
For $H_\omega$ as defined in \eqref{eq:Anderson} satisfies 
\[
\Sigma = [-2d,2d]+ K
\]
where $K$ is the essential support of the single site distribution $V_\omega(0)$. 
\end{prop}
\begin{proof}
By definition, $K=\R\setminus \bigcup \{I \:|\: \mu(I)=0\}$ where $I$ is an interval with rational endpoints.  If $\lambda_0\in [-2d,2d]$, then by~\eqref{eq:LapF} there exists $\alpha\in \tor^d$ with $m(\alpha)=\lambda_0$. Thus, $\Delta e_\alpha = \lambda_0 e_\alpha$ where $e_\alpha(n)=e^{2\pi i\alpha\cdot n}$ for all $n\in\Z^d$. The following holds almost surely: given $L\ge1$, $\eps>0$, and $\lambda_1\in K$, there exists a cube $\Lambda\subset \Z^d$ of side length $L$ such that $\| V - \lambda_1\|_{\ell^\infty(\Lambda)}\le\epsilon$. Then with $\lambda=\lambda_0+\lambda_1$, 
\EQ{\nn 
(H-\lambda) \chi_{\Lambda} e_\alpha = ( V-\lambda_1)\chi_{\Lambda} e_\alpha + g 
}
with $\|g\|_2^2\lesssim |\partial\Lambda| \lesssim L^{d-1}$. Here $\partial\Lambda$ is defined as those $x\in\Lambda$ which have a nearest neighbor in~$\Z^d\setminus\Lambda$, and $|\cdot|$ is the cardinality (or volume).  Hence, with the normalized function $\varphi = \chi_{\Lambda} e_\alpha |\Lambda|^{-\frac12}$ 
\[
\| (H-\lambda) \varphi \|_2 \lesssim \eps   + L^{-\frac12} 
\]
which implies that almost surely, 
\[
\sup_{\eps>0} \| (H-\lambda-i\eps)^{-1} \| =\infty
\]
and thus $\lambda\in \spec(H)$. This shows that $[-2d,2d]+ K\subset \Sigma$.  

Conversely, suppose $\lambda \in \R\setminus ([-2d,2d]+K)$. By compactness of the sum set there exists $\delta>0$ so that almost surely 
\[
\inf_{n\in\Z^d} |V_\omega(n) - \lambda| \ge 2d+\delta
\]
Thus, a.s.\ the resolvent 
\[
(H-\lambda)^{-1}= \big(I + (V_\omega-\lambda)^{-1}\Delta\big)^{-1}(V_\omega-\lambda)^{-1}
\]
exists as a bounded operator on $\ell^2(\Z^d)$. 
\end{proof}

For any cube  $\Lambda\subset\Z^d$ we denote by $P_\Lambda$ the projection onto all {\em states}, i.e., $f\in \ell^2(\Z^d)$ supported in~$\Lambda$. Thus, $P_\Lambda f = \one_{\Lambda} f$ for  any $f\in \ell^2(\Z^d)$. By $H_\Lambda := P_\Lambda H P_\Lambda$ we denote the restriction of~$H$ as in~\eqref{eq:Anderson} to the cube~$\Lambda$ with Dirichlet boundary conditions. Note that the randomness of $H$ is understood and not indicated in the notation, say by an index~$\omega$. 

It is natural to ask about the probability that any given number $E\in\R$ comes close to the spectrum of~$H_\Lambda$. In other words, what is 
\EQ{
\label{eq:Wegner}
\Pbb(\{ \dist(E,\spec(H_\Lambda))<\eps\}) = \Pbb(\{ \| (H_\Lambda -E)^{-1}\|>\eps^{-1}\})  \,?
}
The diagonal operator given by the random potential $V$ alone satisfies
\EQ{\label{eq:Weg0} 
\Pbb(\{ \dist(E,\spec(P_\Lambda VP_\Lambda))<\eps\}) &\le  \Exp \; \#\{ n\in\Lambda \:|\: V(n) \in (E-\eps,E+\eps) \} \\
&\le |\Lambda| \mu((E-\eps,E+\eps)) \le 2 \eps  |\Lambda| \big \| \frac{d\mu}{dx} \big\|_\infty
}
where $\mu$ is the law of $V(0)$. 
A classical fact concerning the random Schr\"odinger operator is that~\eqref{eq:Wegner} permits essentially the same bound as~\eqref{eq:Weg0}. 
This is known as  {\em Wegner's estimate}, see~\cite{Wegner}. 

\begin{prop}
\label{prop:wegner}
Assume the single site distribution of the random operator \eqref{eq:Anderson} satisfies $\mu'=\frac{d\mu}{dy}\in L^\infty(\R)$. Then for all $E\in\R$, 
\EQ{\label{eq:wegner}
\Pbb(\{ \dist(E,\spec(H_\Lambda))<\eps\}) \le 4\eps \big\| \mu' \big\|_{\infty} |\Lambda|
}
for all cubes $\Lambda\subset\Z^d$ and $\eps>0$. 
\end{prop}
\begin{proof}
We will present two proofs. For the first we follow Wegner's original argument~\cite{Wegner}. Denote by $N_\Lambda(x)$ the integrated density of states for the random operator $H_\Lambda$. To wit,  if $E_\Lambda^1 \le E_\Lambda^2 \le \ldots \le E_\Lambda^m$, $m=|\Lambda|$,  denote the eigenvalues of $H_\Lambda$ with multiplicity, then 
\[
N_\Lambda(x) = \#\, \{ 1\le j\le m\:|\: E_\Lambda^j \le x\} 
\]
Let $\varphi\ge0$ be a smooth bump function on $\R$ supported in $[-1,1]$, and set $\varphi_\eps(x)=\eps^{-1} \varphi(x/\eps)$. Normalize so that $\int_\R\varphi(x)\, dx=1$. Then   
with $F_{\Lambda,\eps} = N_\Lambda\ast \varphi_\eps$ one has 
\[
N_\Lambda(E+\eps)-N_\Lambda(E-\eps)\le F_{\Lambda,\eps} (E+2\eps)- F_{\Lambda,\eps} (E- 2\eps) = \int_{E-2\eps}^{E+2\eps} F_{\Lambda,\eps}'(x)\, dx
\]
Since $N_\Lambda$ is a monotone increasing step-function, we have $F_{\Lambda,\eps}'\ge0$. We may interpret $N_\Lambda(x) = N_\Lambda(V_\Lambda,x)$, indicating the dependence of $N_\Lambda$ on all the potential values in~$\Lambda$. Then $N_\Lambda(x+h) = N_\Lambda(V_\Lambda-h,x)$ whence 
\[
N_\Lambda'(x) = -\sum_{j\in\Lambda}  \frac{\partial N_\Lambda}{\partial v_j}(x),\qquad  F_{\Lambda,\eps}'(x)= -\sum_{j\in\Lambda}  \frac{\partial F_{\Lambda,\eps}}{\partial v_j}(x)
\]
as identities between distributional derivatives, respectively smooth functions. Note that $\frac{\partial F_{\Lambda,\eps}}{\partial v_j}\le0$ for each $j$. Indeed, $N_\Lambda$ is decreasing in each $v_j$ separately by the min-max characterization of the eigenvalues of a symmetric matrix. More generally, min-max shows that if $A\ge B$ for any two symmetric matrices, then the eigenvalues $\lambda_1\ge\lambda_2\ge\ldots$ of $A$ dominate those of $B$, denoted by $\mu_1\ge\mu_2\ge\ldots$ which  means that $\lambda_k \ge \mu_k$ for all $k$.  

Thus, with $[-L,L]$ containing the support of $\mu$, 
\EQ{\label{eq:P E}
\Pbb(\{ \dist(E,\spec(H_\Lambda))<\eps\})  &\le  - \int_{E-2\eps}^{E+2\eps}  \sum_{j\in\Lambda} \Exp  \frac{\partial F_{\Lambda,\eps}}{\partial v_j}(x)\, dx \\
&\le   \int_{E-2\eps}^{E+2\eps}  \sum_{j\in\Lambda} \Exp'_j \int_{-L}^L   -\frac{\partial F_{\Lambda,\eps}}{\partial v_j}(x)  \mu'(v_j) \, dv_j \, dx 
}
where $\Exp_j'$ refers to the expectation relative to $\{v_k\}_{k\in\Lambda\setminus\{j\}}$. Further, using the positivity of the integrand,  
\EQ{\label{eq:Weg 1}
\int_{-L}^L   -\frac{\partial F_{\Lambda,\eps}}{\partial v_j}(x)  \mu'(v_j) \, dv_j &\le \|\mu'\|_\infty \int_{-L}^L   -\frac{\partial F_{\Lambda,\eps}}{\partial v_j}(x)   \, dv_j \\
& = \|\mu'\|_\infty\big( F_{\Lambda,\eps}(v_j=-L) - F_{\Lambda,\eps}(v_j=L))\\
& = \|\mu'\|_\infty \int_\R (N_\Lambda (v_j=-L,x) - N_{\Lambda}(v_j=L,x))\varphi_\eps(x)\, dx  \le \|\mu'\|_\infty
}
For the final estimate we use that passing from $v_j=-L$ to $v_j=L$ in $H_\Lambda$ constitutes a rank-$1$ perturbation which implies by min-max that the eigenvalues of $H_\Lambda(v_j=-L)$ and $H_\Lambda(v_j=L)$ interlace. This in turn guarantees that 
\[
N_\Lambda (v_j=-L,x) - N_{\Lambda}(v_j=L,x) \le 1 \qquad \forall\; x\in\R
\]
and thus~\eqref{eq:Weg 1}. Combining~\eqref{eq:P E} with \eqref{eq:Weg 1} implies~\eqref{eq:wegner}. 

For the second proof, we estimate 
\EQ{\label{eq:Weg 3} 
\Pbb(\{ \dist(E,\spec(H_\Lambda))<\eps\})  &\le \Exp \trace \one_{[E-\eps,E+\eps]}(H_\Lambda) \le 2\eps \Exp\: \trace \Im (H_\Lambda - (E+i \eps))^{-1} \\
&\le 2\eps \, \Exp \sum_{n\in \Lambda} \Im\,\langle  (H_\Lambda - (E+i \eps))^{-1}\delta_n,\delta_n\rangle 
}
where we used that 
\[
\one_{[E-\eps,E+\eps]}(x)  \le \frac{2\eps^2}{\eps^2+(x-E)^2}  = \Im \frac{2\eps}{x-(E+i\eps)}
\]
Next, we establish a fundamental relation on the resolvents of rank-$1$ perturbations. Let $A$ be any self-adjoint operator on a Hilbert space and $\varphi$ a unit vector, $\lambda$ a real scalar. From the resolvent identity, for any complex $z$ with $\Im z>0$, 
\EQ{\nn 
(A + \lambda \varphi\otimes\varphi -z)^{-1} - (A  -z)^{-1} &= -\lambda (A  -z)^{-1} (\varphi\otimes \varphi)  (A + \lambda \varphi\otimes\varphi -z)^{-1} \\
\langle (A + \lambda \varphi\otimes\varphi -z)^{-1}\varphi,\varphi\rangle - \langle (A  -z)^{-1}\varphi,\varphi\rangle &= - \lambda \langle (A + \lambda \varphi\otimes\varphi -z)^{-1}\varphi,\varphi\rangle \langle (A  -z)^{-1}\varphi,\varphi\rangle\\
\langle (A + \lambda \varphi\otimes\varphi -z)^{-1}\varphi,\varphi\rangle &=  \big[\lambda + \langle (A  -z)^{-1}\varphi,\varphi\rangle^{-1}\big]^{-1}
}
Note that $\Im \langle (A  -z)^{-1}\varphi,\varphi\rangle^{-1} \ne 0$ by $\Im z\ne0$. Applying this to $$H_\Lambda = H_\Lambda^{(n)} + V_\omega(n) \delta_n\otimes \delta_n, \qquad n\in\Lambda$$ where $H_\Lambda^{(n)}$ is the operator with the potential at lattice site $n$ set equal to~$0$, yields 
\[
\langle  (H_\Lambda - (E+i \eps))^{-1}\delta_n,\delta_n\rangle = \big[   V_\omega(n) + \langle (H_\Lambda^{(n)} - (E+i \eps))^{-1}\delta_n,\delta_n\rangle^{-1}\big]^{-1}
\]
Writing 
\[
\langle (H_\Lambda^{(n)} - (E+i \eps))^{-1}\delta_n,\delta_n\rangle^{-1} =- t - is,\qquad s>0
\]
the random variables $t,s$ only depend on the random lattice sites in $\Lambda\setminus\{n\}$. Consequently, the inner product in the final expression of~\eqref{eq:Weg 3} is  bounded by 
\EQ{\nn 
\Exp \Im \langle  (H_\Lambda - (E+i \eps))^{-1}\delta_n,\delta_n\rangle   &\le \Exp_n'\int \frac{ s \, \mu(dx)}{(x-t)^2 + s^2} \le \pi \|\mu'\|_\infty
}
which in combination with~\eqref{eq:Weg 3} yields 
\[
\Pbb(\{ \dist(E,\spec(H_\Lambda))<\eps\})  \le 2\pi \eps \|\mu'\|_\infty |\Lambda|
\]
This is slightly worse than the previous proof but the precise constant is irrelevant. 
\end{proof}

The assumption of bounded density $\mu'$ can be relaxed, but some amount of regularity of the single-site  distribution is needed. Indeed, the {\em mobility of the eigenvalues}
under the randomness expressed by Wegner's estimate is reduced to the mobility of the potential at each site. The heuristic notion of ``mobility" refers to the movement of the eigenvalues as a result of the movement of the potential. Both arguments presented above hinge on this step. See, however, an alternative approach by Stollmann~\cite{Stoll}.  

{\em Anderson localization} refers to the following statement. 

\begin{thm}
\label{thm:AL}
Let $H = \eps^2\Delta_{\Z^d} + V_\omega$ where $V_\omega$ is random i.i.d.\ potential with single site distribution $\mu$ of compact support and of bounded density with $\|\mu'\|_\infty\le1$. Then there exists $\eps_0=\eps_0(d)>0$ so that for all $0<\eps\le \eps_0$, almost surely $\ell^2(\Z^d)$ has  an orthonormal basis of exponentially decaying eigenfunctions of the random operator $H$. In particular, the spectrum is a.s.~pure point and thus $\Sigma_{sc}=\Sigma_{ac}=\emptyset$. 
\end{thm} 

In stark contrast to this result, periodic potentials exhibit a Fourier basis of Bloch-Floquet solutions with $\Sigma_{pp}=\Sigma_{sc}=\emptyset$. 
 Thus their spectral measures are absolutely continuous. This, as well as  $V_\omega=\const.$ shows that Theorem~\ref{thm:AL} requires the removal of a zero probability event. A wide open problem is to prove $\Sigma_{ac}\ne\emptyset$ for large $\eps$ in dimensions $d\ge3$. This is known as Anderson's {\em extended states} conjecture.

There are two main techniques known to prove Theorem~\ref{thm:AL}: Fr\"ohlich-Spencer~\cite{FS} {\em multiscale analysis} on the one hand, and the Aizenman-Molchanov~\cite{AiMo} {\em fractional moment method} on the other hand.  We will sketch the former and refer to~\cite{Dirk} for an introduction to the latter. A streamlined rendition of the induction-on-scales method of~\cite{FS} can be found in~\cite{vDK}, which also does not require the use of the Simon-Wolff criterion~\cite{STom}, as earlier multi-scale proofs of Theorem~\ref{thm:AL} had done. Germinet and Klein have obtained significant refinements of the multi-scale argument in a series of papers, see for example~\cite{GerK}. 

Returning briefly to the Wegner estimate, we remark that the physically important example of Bernoulli potentials taking discrete values completely falls outside the range of Proposition~\ref{prop:wegner}. 
See~\cite{DSm} for a recent advance on this case in two dimensions and on localization for Anderson Bernoulli. The mobility of the eigenvalues of $H_\Lambda$ if $V=\pm1$ derives from the interaction between eigenfunctions and is  more delicate. On the other hand, localization in the one-dimensional Bernoulli model is a classical result by Carmona, Klein and Martinelli~\cite{CKM}. While these authors rely on the original multi-scale methods of Fr\"ohlich and Spencer, this is avoided in the recent papers \cite{B7}, \cite{GorK}, and~\cite{JitZ}. The arguments there use the large deviation theorems and the   methods of Bourgain, Goldstein~\cite{BouG}, see the final section of these notes. 

Before getting in to the details, some basic ideas and motivation. Suppose $H$ has an $\ell^2$-complete sequence of exponentially decaying normalized eigenfunctions $\{\phi_j\}_{j\in\Z}$ with eigenvalues~$E_j$, both random. Restrict $H$ to a large cube $\Lambda$ and write (heuristically) 
\[
(H_\Lambda - (E+i\eps))^{-1} \approx \sum_j \frac{\phi_j\otimes\phi_j}{E_j-(E+i\eps)}
\]
where the sum extends over all eigenfunctions ``supported" in the box $\Lambda$. It should be intuitively clear what this means.  Then $|(\phi_j\otimes\phi_j)(x,y)|\les \exp(-\gamma|x-y|)$ with $\gamma>0$ for those $j$, for which either $x$ or~$y$ are in the support of~$\phi_j$. All the others make much smaller contributions which we can essentially ignore. In conclusion, if 
\[
\|(H_\Lambda - (E+i0))^{-1}\|\le K\text{\ \ then\ \ } |(H_\Lambda - (E+i0))^{-1}(x,y)|\le K  \exp(-\gamma|x-y|)\quad\forall\; x,y\in\Lambda
\]
The condition here is precisely what Wegner's estimate controls, and a cube which exhibits both the separation from the spectrum and the exponential off-diagonal decay will be called {\em regular} for energy~$E$. A substantial effort below is to show that cubes are regular for a given energy with high probability. However, this is insufficient to prove localization and one needs to consider two disjoint cubes and understand the probability that they are both singular for {\em any energy}. The essential feature of this idea is to control the probability of an event uniformly in all energies, rather than for a fixed energy. The latter can never imply an a.s.~statement about the spectrum since we cannot take the union of a bad event over an uncountable family of energies. More importantly, excluding the event that two boxes are in resonance simultaneously (which refers that they are both singular at the same $E$) will precisely allow us to show that a.s.\ tunneling cannot occur over long distances leading to exponentially localized eigenfunctions. 

We shall now prove Theorem~\ref{thm:AL} by induction on scales. We will need to allow rectangles as regions of finite volume rather than just cubes. Thus, define 
a {\em box} centered at $x$ of scale $L$ to be any rectangle of the form
\EQ{\label{eq:box}
\Lambda_L(x) =\big\{ y\in\Z^d \:|\:  - m_j \le  y_j-x_j \le M_j \quad \forall\; 1\le j \le d \big\} 
}
where  $m_j\ge0$, $M_j\ge0$ and $\max(m_j,M_j)=L$ for each $j$. If $m_j=M_j=L$ for each $j$, then we have standard cube which we denote by~$Q_L(x)$. 
These rectangles arise as intersections of  cubes $Q_L(x)\cap Q_{\tilde L}(y)$ if $x\in Q_{\tilde L}(y)$ with $\tilde L\ge L$, see Figure~\ref{fig:Qcap}. Wegner's estimate applies unchanged to boxes. Note that for a given integer $L\ge1$ and $x\in\Z^d$ there are $B(L)=(2L+1)^d$ boxes~$\Lambda_L(x)$.  The following {\em deterministic lemma} allows us to bound the Green function at an initial scale which will be specified later. 

\begin{lem}
\label{lem:step0}
Suppose $4d\eps\le \delta$ and $0<\eps\le\frac12$. Let $\Lambda$ be any box as in~\eqref{eq:box} and assume 
 $$\dist(\spec(H_{\Lambda}),E)\ge\delta.$$ Then 
 \EQ{\label{eq:G step0}
|G_\Lambda(E)(x,y)| \le 4\delta^{-1} \eps^{|x-y|} \quad\forall \; x,y\in \Lambda. 
}
Here $|x|=\max_j |x_j|$ and $G_\Lambda(E)=(H_\Lambda-E)^{-1}$ is the Green function on $\Lambda$ with energy~$E$. 
\end{lem}
\begin{proof}
By min-max, $|V(x)-E|\ge \delta-2d\eps^2\ge \delta/2$ for all $x\in\Lambda$. Then 
\EQ{\nn 
G_\Lambda(E)(x,y) &= (I + \eps^2(V_\Lambda-E)^{-1}\Delta)^{-1} (V_\Lambda-E)^{-1}(x,y) \\
& = \sum_{\ell=0}^\infty (-1)^\ell\eps^{2\ell}[(V_\Lambda-E)^{-1}\Delta]^{\ell}(x,y) (V_\Lambda-E)^{-1}(x)
}
Recall that $(V_\Lambda-E)^{-1}$ is a diagonal operator. Using that $\| (V_\Lambda-E)^{-1}\Delta\|\le 2\delta^{-1}\|\Delta\|\le 4d \delta^{-1}$, 
we have $|\eps^{2\ell}[(V_\Lambda-E)^{-1}\Delta]^{\ell}(x,y)|\le (4d \eps^2\delta^{-1})^{\ell}\le \eps^\ell$ where it suffices to consider $\ell\ge |x-y|_1\ge |x-y|$ with 
 $|x|_1=\sum_{j=1}^d |x_j|$ (otherwise this term vanishes). Summing up the geometric series using that $\eps\le\frac12$ proves~\eqref{eq:G step0}. 
\end{proof}

In terms of random operators one has \eqref{eq:G step0} with high probability. 

\begin{cor}
\label{cor:step0}
Suppose $4d\eps\le \delta$ and $0<\eps\le\frac12$. Then \eqref{eq:G step0} holds up to probability at most $4\delta|\Lambda|$. 
\end{cor}
\begin{proof}
Apply Wegner. 
\end{proof}

\begin{figure}[ht]
\tikzset{every picture/.style={line width=0.75pt}} 
\begin{tikzpicture}[x=0.48pt,y=0.48pt,yscale=-1,xscale=1]
\draw  [line width=1.5]  (48,18.2) -- (272.51,18.2) -- (272.51,242.71) -- (48,242.71) -- cycle ;
\draw  [fill={rgb, 255:red, 74; green, 144; blue, 226 }  ,fill opacity=1 ] (155.38,130.46) .. controls (155.38,127.76) and (157.56,125.58) .. (160.26,125.58) .. controls (162.95,125.58) and (165.14,127.76) .. (165.14,130.46) .. controls (165.14,133.15) and (162.95,135.34) .. (160.26,135.34) .. controls (157.56,135.34) and (155.38,133.15) .. (155.38,130.46) -- cycle ;
\draw  [line width=1.5]  (118,68.2) -- (342.51,68.2) -- (342.51,292.71) -- (118,292.71) -- cycle ;
\draw  [fill={rgb, 255:red, 74; green, 144; blue, 226 }  ,fill opacity=1 ] (225.38,180.46) .. controls (225.38,177.76) and (227.56,175.58) .. (230.26,175.58) .. controls (232.95,175.58) and (235.14,177.76) .. (235.14,180.46) .. controls (235.14,183.15) and (232.95,185.34) .. (230.26,185.34) .. controls (227.56,185.34) and (225.38,183.15) .. (225.38,180.46) -- cycle ;
\draw  [color={rgb, 255:red, 208; green, 2; blue, 27 }  ,draw opacity=1 ][line width=2.25]  (118,68.2) -- (272.51,68.2) -- (272.51,242.71) -- (118,242.71) -- cycle ;

\draw (181,127) node  [font=\large]  {$y$};
\draw (252,177) node  [font=\large]  {$x$};
\draw (85,41) node  [font=\large]  {$Q_{L}( y)$};
\draw (309,87) node  [font=\large]  {$Q_{L}( x)$};
\draw (214,86) node  [font=\large]  {$\Lambda _{L}( x)$};

\end{tikzpicture}
\caption{A box arising as intersection of cubes}
\label{fig:Qcap}
\end{figure}
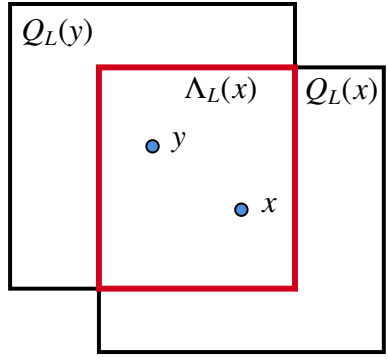
The following lemma demonstrates how to obtain exponential decay of the Green function at a large scale box if all boxes contained in it of a much smaller scale have this property, {\em with possibly one exception}. The latter is needed in order to be able to {\it square the probabilities} of a having a bad small box inside a bigger one  as we pass to the next scale. We will use a resolvent expansion, obtained by iterating the resolvent identity: let $\Lambda'\subset\Lambda$ be boxes, and let $A=H_\Lambda$, $B=H_{\Lambda'}\oplus H_{\Lambda\setminus\Lambda'}$ viewed as operators on~$\ell^2(\Lambda)$. Then, 
\EQ{\label{eq:resid}
(H_\Lambda-E)^{-1} &= (H_{\Lambda'}\oplus H_{\Lambda\setminus\Lambda'} -E)^{-1} - \eps^2 (H_\Lambda-E)^{-1}\Gamma_{\Lambda,\Lambda'} (H_{\Lambda'}\oplus H_{\Lambda\setminus\Lambda'} -E)^{-1}\\
G_\Lambda(E)(x,y) &= G_{\Lambda'}(E)(x,y)\one_{[y\in\Lambda']} - \eps^2 \!\!\!\!\sum_{(w',w)\in\partial\Lambda'} G_{\Lambda}(E)(w,y) G_{\Lambda'}(x,w')
}
for all  $x\in\Lambda'$ and $y\in\Lambda$. Here $\partial\Lambda'=\{(w',w)\:|\: w'\in\Lambda',\,w\in\Lambda\setminus\Lambda',\,|w-w'|=1\}$ is the {\em relative boundary} of $\Lambda'$ inside of~$\Lambda$,  
and $\Gamma_{\Lambda,\Lambda'}=\one_{\partial\Lambda'}$. 

\begin{lem}
\label{lem:ind step}
Let $\Lambda$ be a box at scale $L_1\ge 100L_0$ and assume $\dist(\spec(H_\Lambda),E)\ge\delta_1$ with $0<\delta_1\le 1$. Let $\Lambda'_*\subset\Lambda$ be some box at scale $L_0\ge1$ and assume that all boxes $\Lambda'\subset \Lambda\setminus\Lambda'_*$ at scale $L_0$ satisfy 
\EQ{\label{eq:G L0}
|G_{\Lambda'}(E)(x,y)|\le 4\delta_0^{-1} \eps^{\gamma_0|x-y|}\quad\forall\; x,y\in \Lambda',\; |x-y|\ge L_0/2
}
Suppose $8d (2L_0+1)^{d-1}  \eps^{2+\gamma_0 L_0}\le\delta_0\le1$.  Then
\EQ{\label{eq:G L1}
|G_{\Lambda}(E)(x,y)|\le \delta_1^{-1} \eps^{\gamma_1|x-y|}\quad\forall\; x,y\in \Lambda,\; |x-y|\ge L_1/2
}
provided
\EQ{\label{eq:gamma1}
\Big[ \gamma_0-\gamma_1\Big(1-\frac{1}{L_0+1}-\frac{8L_0}{L_1}\Big)^{-1}+\frac{2}{L_0}\Big] \log\frac1\eps \ge L_0^{-1}\log( 8d  (2L_0+1)^{d-1} \delta_0^{-1} )
} 
\end{lem}
\begin{proof}
Pick $x,y\in \Lambda$ with $|x-y|\ge L_1/2$ and set $\Lambda_x = Q_{L_0}(x)\cap\Lambda$. If $\Lambda_x\cap \Lambda'_*\ne\emptyset$, then we do not expand around $x$ and instead expand around $y$ since $L_1\ge 100L_0$ implies that $\Lambda_y\cap \Lambda'_* =\emptyset$.  

\begin{figure}[ht]
\tikzset{every picture/.style={line width=0.75pt}} 

\begin{tikzpicture}[x=0.5pt,y=0.5pt,yscale=-1,xscale=1]

\draw  [line width=1.5]  (60,11.48) -- (480.51,11.48) -- (480.51,529.48) -- (60,529.48) -- cycle ;
\draw  [fill={rgb, 255:red, 65; green, 117; blue, 5 }  ,fill opacity=1 ] (86,60.98) .. controls (86,58.49) and (88.01,56.48) .. (90.5,56.48) .. controls (92.98,56.48) and (95,58.49) .. (95,60.98) .. controls (95,63.46) and (92.98,65.48) .. (90.5,65.48) .. controls (88.01,65.48) and (86,63.46) .. (86,60.98) -- cycle ;
\draw   (60.46,11.52) -- (139.95,11.44) -- (140.05,109.34) -- (60.57,109.42) -- cycle ;
\draw   (111,41.48) -- (208.57,41.48) -- (208.57,139.05) -- (111,139.05) -- cycle ;
\draw   (361,472.48) -- (458.9,472.48) -- (458.9,528.42) -- (361,528.42) -- cycle ;
\draw   (407,347.42) -- (479.9,347.42) -- (479.9,444.42) -- (407,444.42) -- cycle ;

\draw  [fill={rgb, 255:red, 74; green, 144; blue, 226 }  ,fill opacity=1 ] (419,472.98) .. controls (419,470.49) and (421.01,468.48) .. (423.5,468.48) .. controls (425.98,468.48) and (428,470.49) .. (428,472.98) .. controls (428,475.46) and (425.98,477.48) .. (423.5,477.48) .. controls (421.01,477.48) and (419,475.46) .. (419,472.98) -- cycle ;
\draw  [fill={rgb, 255:red, 74; green, 144; blue, 226 }  ,fill opacity=1 ] (155.28,90.26) .. controls (155.28,87.77) and (157.3,85.76) .. (159.78,85.76) .. controls (162.27,85.76) and (164.29,87.77) .. (164.29,90.26) .. controls (164.29,92.75) and (162.27,94.76) .. (159.78,94.76) .. controls (157.3,94.76) and (155.28,92.75) .. (155.28,90.26) -- cycle ;
\draw  [fill={rgb, 255:red, 74; green, 144; blue, 226 }  ,fill opacity=1 ] (136.28,90.26) .. controls (136.28,87.77) and (138.3,85.76) .. (140.78,85.76) .. controls (143.27,85.76) and (145.28,87.77) .. (145.28,90.26) .. controls (145.28,92.75) and (143.27,94.76) .. (140.78,94.76) .. controls (138.3,94.76) and (136.28,92.75) .. (136.28,90.26) -- cycle ;
\draw  [fill={rgb, 255:red, 65; green, 117; blue, 5 }  ,fill opacity=1 ] (405.28,521.26) .. controls (405.28,518.77) and (407.3,516.76) .. (409.78,516.76) .. controls (412.27,516.76) and (414.29,518.77) .. (414.29,521.26) .. controls (414.29,523.75) and (412.27,525.76) .. (409.78,525.76) .. controls (407.3,525.76) and (405.28,523.75) .. (405.28,521.26) -- cycle ;

\draw  [fill={rgb, 255:red, 208; green, 2; blue, 27 }  ,fill opacity=1 ] (248,162.48) -- (345.57,162.48) -- (345.57,260.05) -- (248,260.05) -- cycle ;
\draw   (374,410.92) -- (471.57,410.92) -- (471.57,508.49) -- (374,508.49) -- cycle ;
\draw   (367.5,281.91) -- (465.07,281.91) -- (465.07,379.48) -- (367.5,379.48) -- cycle ;
\draw  [fill={rgb, 255:red, 74; green, 144; blue, 226 }  ,fill opacity=1 ] (418.28,459.7) .. controls (418.28,457.22) and (420.3,455.2) .. (422.78,455.2) .. controls (425.27,455.2) and (427.29,457.22) .. (427.29,459.7) .. controls (427.29,462.19) and (425.27,464.2) .. (422.78,464.2) .. controls (420.3,464.2) and (418.28,462.19) .. (418.28,459.7) -- cycle ;
\draw  [fill={rgb, 255:red, 74; green, 144; blue, 226 }  ,fill opacity=1 ] (451.28,396.2) .. controls (451.28,393.71) and (453.3,391.7) .. (455.78,391.7) .. controls (458.27,391.7) and (460.29,393.71) .. (460.29,396.2) .. controls (460.29,398.69) and (458.27,400.7) .. (455.78,400.7) .. controls (453.3,400.7) and (451.28,398.69) .. (451.28,396.2) -- cycle ;
\draw  [fill={rgb, 255:red, 74; green, 144; blue, 226 }  ,fill opacity=1 ] (412.28,346.48) .. controls (412.28,344) and (414.3,341.98) .. (416.78,341.98) .. controls (419.27,341.98) and (421.29,344) .. (421.29,346.48) .. controls (421.29,348.97) and (419.27,350.98) .. (416.78,350.98) .. controls (414.3,350.98) and (412.28,348.97) .. (412.28,346.48) -- cycle ;
\draw  [fill={rgb, 255:red, 74; green, 144; blue, 226 }  ,fill opacity=1 ] (411.78,330.7) .. controls (411.78,328.21) and (413.8,326.2) .. (416.29,326.2) .. controls (418.77,326.2) and (420.79,328.21) .. (420.79,330.7) .. controls (420.79,333.19) and (418.77,335.2) .. (416.29,335.2) .. controls (413.8,335.2) and (411.78,333.19) .. (411.78,330.7) -- cycle ;
\draw  [fill={rgb, 255:red, 74; green, 144; blue, 226 }  ,fill opacity=1 ] (452.28,411.48) .. controls (452.28,409) and (454.3,406.98) .. (456.78,406.98) .. controls (459.27,406.98) and (461.29,409) .. (461.29,411.48) .. controls (461.29,413.97) and (459.27,415.98) .. (456.78,415.98) .. controls (454.3,415.98) and (452.28,413.97) .. (452.28,411.48) -- cycle ;
\draw  [fill={rgb, 255:red, 74; green, 144; blue, 226 }  ,fill opacity=1 ] (363,281.91) .. controls (363,279.43) and (365.02,277.41) .. (367.5,277.41) .. controls (369.99,277.41) and (372,279.43) .. (372,281.91) .. controls (372,284.4) and (369.99,286.42) .. (367.5,286.42) .. controls (365.02,286.42) and (363,284.4) .. (363,281.91) -- cycle ;
\draw  [fill={rgb, 255:red, 74; green, 144; blue, 226 }  ,fill opacity=1 ] (347.28,282.26) .. controls (347.28,279.77) and (349.3,277.76) .. (351.78,277.76) .. controls (354.27,277.76) and (356.29,279.77) .. (356.29,282.26) .. controls (356.29,284.75) and (354.27,286.76) .. (351.78,286.76) .. controls (349.3,286.76) and (347.28,284.75) .. (347.28,282.26) -- cycle ;
\draw  [fill={rgb, 255:red, 74; green, 144; blue, 226 }  ,fill opacity=1 ] (203.28,119.26) .. controls (203.28,116.77) and (205.3,114.76) .. (207.78,114.76) .. controls (210.27,114.76) and (212.29,116.77) .. (212.29,119.26) .. controls (212.29,121.75) and (210.27,123.76) .. (207.78,123.76) .. controls (205.3,123.76) and (203.28,121.75) .. (203.28,119.26) -- cycle ;
\draw  [fill={rgb, 255:red, 74; green, 144; blue, 226 }  ,fill opacity=1 ] (219.28,119.26) .. controls (219.28,116.77) and (221.3,114.76) .. (223.78,114.76) .. controls (226.27,114.76) and (228.29,116.77) .. (228.29,119.26) .. controls (228.29,121.75) and (226.27,123.76) .. (223.78,123.76) .. controls (221.3,123.76) and (219.28,121.75) .. (219.28,119.26) -- cycle ;

\draw (509,28.48) node  [font=\Large]  {$\Lambda $};
\draw (75,55.28) node  [font=\large]  {$x$};
\draw (90, 26) node   {$\Lambda'(x)$};
\draw (396,519) node  [font=\large]  {$y$};
\draw (330, 515) node   {$\Lambda'(y)$};
\draw (367,179.64) node  [font=\Large]  {$\Lambda '_{*}$};
\draw (240,63.64) node    {$\Lambda '( w_{1})$};
\draw (178,90.64) node    {$w_{1}$};
\draw (124,90) node    {$w'_{1}$};
\draw (422,492) node    {$\tilde w_1'$};
\draw (405,458.03) node    {$\tilde w_1$};
\draw (440,426.03) node    {$\tilde w_2'$};
\draw (436.45,395.92) node    {$\tilde w_2$};
\draw (421,365.03) node    {$\tilde w_3'$};
\draw (416,315.03) node    {$\tilde w_3$};
\draw (340,435.31) node    {$\Lambda '(\tilde{w}_{1})$};
\draw (335,309.31) node    {$\Lambda '(\tilde{w}_3)$};
\draw (375,394.31) node    {$\Lambda '(\tilde{w}_{2})$};
\draw (190,117.25) node    {$w'_{2}$};
\draw (244,119) node    {$w_{2}$};
\draw (383,296.25) node    {$\tilde w_4'$};
\draw (334,280.25) node    {$\tilde w_4$};
\end{tikzpicture}
\caption{One term in the expansion \eqref{eq:G iter} with $s=2, t=4$} 
\end{figure}
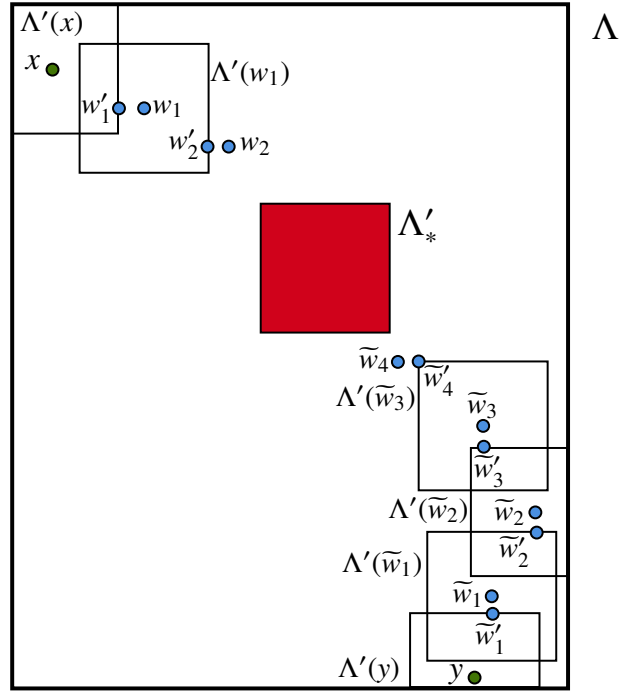

Iterating~\eqref{eq:resid} leads to an expression of the form, with $w_0=x$, $\tilde w_0=y$, 
\EQ{\label{eq:G iter}
& G_{\Lambda}(E)(x,y) = (-\eps^2)^{s+t} \!\!\! \!\!\!\sum_{(w'_1,w_1)\in\partial\Lambda'(x)}\sum_{(w'_2,w_2)\in\partial\Lambda'(w_1)}...\!\!\!\sum_{(w'_s,w_s)\in\partial\Lambda'(w_{s-1})} \prod_{j=1}^s G_{\Lambda'(w_{j-1})}(w_{j-1},w'_{j}) 
\\
& \sum_{(\tilde w'_1,\tilde w_1)\in\partial\Lambda'(y)}\sum_{(\tilde w'_2,\tilde w_2)\in\partial \Lambda'(\tilde w_1)}...\!\!\!\sum_{(\tilde w'_t,\tilde w_t)\in\partial\Lambda'(\tilde w_{t-1})} G_{\Lambda}(w_s,\tilde w_t)  \prod_{k=1}^t G_{  \Lambda' (\tilde w_{k-1}) }(\tilde w_{k-1},\tilde w'_k) 
}
with all Green function on the right-hand side being at energy $E$. 	Here $s\ge0$ and $t\ge0$ are the maximal number of steps we can take from $x$, respectively, $y$ with any boxes of size $L_0$ centered at points distance~$1$ from the boundary of a previous box, before they might intersect~$ \Lambda'_*$.  All boxes here are of the form $Q_{L_0}(w_j)\cap \Lambda=\Lambda'(w_j)$. In particular, if $(y',y)\in \partial \Lambda'(w_j)$, then $|y'-w_j|=L_0$.  We claim  that $s+t\ge1$ is the {\em minimal positive integer} with 
\EQ{\nn 
|x-y| < (t+s)(L_0+1) + 4L_0+1
}
Indeed, if $|x-y| \ge (t+s)(L_0+1) + 4L_0+1$, then $$|x-y|+1 - [(t+s)(L_0+1) + 2L_0+1] \ge 2L_0+1$$
which implies that we could go either one more step in the $x$, resp.~$y$,  expansion without intersecting~$\Lambda'_*$.  Thus, 
\EQ{\label{eq:st}
\xi-1< s+t\le \xi,\qquad \xi= \frac{|x-y|-3L_0}{L_0+1} 
}
To estimate~\eqref{eq:G iter}, use $|G_{\Lambda}(w_s,\tilde w_t)|\le \|G_\Lambda\|\le  \delta_1^{-1}$ and $|G_{\Lambda'(w_{j-1})}(w_{j-1},w'_{j})|\le 4\delta_0^{-1} \eps^{\gamma_0 L_0}$ and the same for all of the Green functions over the smaller boxes.  The number of pairs in the boundary satisfy $|\partial \Lambda'| \le 2d(2L_0+1)^{d-1}$ whence 
\EQ{
\label{eq:GEbd0}
| G_{\Lambda}(E)(x,y) | &\le \delta_1^{-1}\, (8d\eps^2 (2L_0+1)^{d-1} \delta_0^{-1} \eps^{\gamma_0 L_0}\big)^{s+t} \\
&\le \delta_1^{-1}\, (8d\eps^2 (2L_0+1)^{d-1} \delta_0^{-1} \eps^{\gamma_0 L_0}\big)^{\frac{|x-y|}{L_0+1}-4}
}
using that the parenthesis is a number in $(0,1]$. Note that $\frac{|x-y|}{L_0+1}-4\ge \frac{46L_0-4}{L_0+1}\ge 21$. 
We need to ensure that for all $x,y\in\Lambda$, $|x-y|\ge L_1/2$ we have 
\[
(8d\eps^2 (2L_0+1)^{d-1} \delta_0^{-1} \eps^{\gamma_0 L_0}\big)^{\frac{|x-y|}{L_0+1}-4} \le \eps^{\gamma_1|x-y|}
\]
which then implies \eqref{eq:G L1} via~\eqref{eq:GEbd0}. Taking logarithms, this reduces to 
\[
\Big[\gamma_0-\gamma_1\Big(1-\frac{1}{L_0+1}-\frac{4L_0}{|x-y|}\Big)^{-1}+\frac{2}{L_0}\Big] \log\frac1\eps \ge L_0^{-1}\log( 8d  (2L_0+1)^{d-1} \delta_0^{-1} )
\]
The worst case here is $|x-y|= L_1/2$ which gives~\eqref{eq:gamma1}. 
\end{proof}

\begin{defn}\label{def:Reg}
Fix any $x_0\in \Z^d$. Then we define an $L$-box $\Lambda_{L}(x_0)$ to be $(\gamma,E)$-regular if  it exhibits
\begin{itemize}
\item {\em non-resonance at energy}~$E$: $\dist(\spec(H_{\Lambda_L(x_0)}),E)\ge \delta(L)=\exp(-L^\beta)$
\item {\em exponential Green function decay}: $|G_{\Lambda_L(x_0)}(E)(x,y)|\le 4\delta(L)^{-1} \eps^{\gamma|x-y|}\;$ $\forall\; x,y\in \Lambda_L(x_0)$ with $|x-y|\ge L/2$
\end{itemize}
Here  $E\in\R$ is arbitrary, $\gamma>0$ will be specified below, depending on the scale, and $\beta\in(0,1)$ will be a  fixed constant.  A box is $(\gamma,E)$-singular if it is not $(\gamma,E)$-regular. 
\end{defn}

At the initial scale of the induction, by Corollary~\ref{cor:step0}
\EQ{\label{eq:P0}
\sup_{E\in\R} \Pbb(\{ \exists\; \Lambda_{L_0}(x_0) \text{\ \ which is\ \ }(1,E)\text{-singular}\}) &\le 4B(L_0) |Q_{L_0}(x_0)| \delta(L_0) =:p_0
}
The existence inside the set refers to all possible boxes of the initial scale~$L_0=L_0(d,\beta)\ge 100$ centered at $x_0$ of which there are   $B(L_0)=(2(L_0+1))^d$, while $|Q_{L_0}(x_0)|=(2L_0+1)^d$ is the volume of the largest $L_0$-box.  Thus we have 
\[
p_0 = 4(2(L_0+1))^d(2L_0+1)^d \exp(-L_0^\beta)
\]
where $\beta$ is just chosen here to so that $\exp(-L_0^\beta)=\delta$ and will in fact be in $(0,1)$. 
Corollary~\ref{cor:step0} requires that, where $\delta_0:=\delta(L_0)$, 
\EQ{\label{eq:eps small}
4d\eps\le \delta_0. 
}
Set $L_1=\lceil L_0^\alpha\rceil$ where $\alpha>1$ will also be specified later. By Lemma~\ref{lem:ind step}, 
\begin{equation}
\label{eq:P1}
\left\{ \begin{array}{rl}
\sup_{E\in\R}\!\!\!\!\!\! &\;\Pbb(\{ \exists\; \Lambda_{L_1}(x_0) \text{\ \ which is\ \ }(\gamma_1,E)\text{-singular}\}) \le p_1 \\
p_1 &:= 4B(L_1) |Q_{L_1}(x_0)|  \delta_1 + |Q_{L_1}(x_0)|^2 p_0^2 \\
\gamma_1 &:= \big(1-\frac{1}{L_0+1}-\frac{8L_0}{L_1}\big)(1-L_0^{\beta-1})
\end{array}
\right.
\end{equation}
In fact, $p_1$ is the sum of two contributions. On the one hand, Wegner's estimate gives, with $\delta_1=\exp(-L_1^\beta)$, 
\[
\Pbb(\{ \exists\; \Lambda_{L_1}(x_0) \text{\ \ with\ \ } \dist(\spec(H_{\Lambda_{L_1}(x_0)}),E)\le \delta_1\})  \le4B(L_1) |Q_{L_1}(x_0)|  \delta_1 
\]
which is the first term on the right-hand side of $p_1$. It controls the probability that one of the boxes $\Lambda_{L_1}(x_0) $ is resonant at energy $E$ with resonance width~$\delta_1$. The other term bounds
\[
\Pbb(\{ \exists\,\text{\ two disjoint\ }(1,E)\text{-singular\ } L_0\text{\ boxes in\ } Q_{L_1}(x_0)\})\le  |Q_{L_1}(x_0)|^2 p_0^2
\]
where the factor $|Q_{L_1}(x_0)|^2=(2L_1+1)^{2d}$ is a result of selecting the centers of the $L_0$ boxes in~$Q_{L_1}(x_0)$. 
Assuming $L_0^{\beta(\alpha-1)}\ge 2$, we have $\delta_1\le p_0^2$ and thus $p_1\le 5 B(L_1)^2 p_0^2$. 
Finally, setting $\gamma_0=1$ and $\gamma_1$ as above in~\eqref{eq:gamma1} yields
\[
(1 + {2}{L_0}^{-\beta}) \log\frac1\eps \ge 1+L_0^{-\beta} ( \log( 8d)  +(d-1)\log(2L_0+1) )
\]
In view of \eqref{eq:eps small} this holds for $L_0(d,\beta)$ sufficiently large, proving \eqref{eq:P1}. 

Inductively, define $L_{k+1}=\lceil L_k^\alpha\rceil$. In analogy with~\eqref{eq:P1} one has with $\delta_k=\exp(-L_k^\beta)$, 
\begin{equation}
\label{eq:Pk}
\left\{ \begin{array}{rl}
\sup_{E\in\R}\!\!\!\!\!\! &\;\Pbb(\{ \exists\; \Lambda_{L_k}(x_0) \text{\ \ which is\ \ }(\gamma_k,E)\text{-singular}\}) \le p_k \\
p_k &:= 4B(L_k) |Q_{L_k}(x_0)|  \delta_k + |Q_{L_k}(x_0)|^2 p_{k-1}^2 \\
\gamma_k &:= \gamma_{k-1}\big(1-\frac{1}{L_{k-1}+1}-\frac{8L_{k-1}}{L_k}\big)(1-L_{k-1}^{\beta-1})
\end{array}
\right.
\end{equation}
One has $p_k\le B(L_k)^2(4\delta_k + p_{k-1}^2)\le (2(L_k+1))^{2d} (4\delta_k + p_{k-1}^2)$. On the one hand,  of $L_0$ is large enough, then
\[
\prod_{k=1}^\infty \Big(1-\frac{1}{L_{k-1}+1}-\frac{8L_{k-1}}{L_k}\Big)(1-L_{k-1}^{\beta-1})\ge \prod_{k=1}^\infty  (1-{9}{L_{k-1}^{1-\alpha}} )(1-L_{k-1}^{\beta-1}) >0 \]
since $L_k\ge (L_0)^{\alpha^k}$ and $\sum_{j\ge0} \big(L_{j}^{1-\alpha} + L_{j}^{\beta-1}\big)=o(1)$ as $L_0\to\infty$, 
whence $\inf_{k\ge1} \gamma_k>0$ (approaches~$1$ for large $L_0$). 
On the other hand, we claim that
$
\sum_{k=0}^\infty p_k <\infty. 
$
Indeed, from \eqref{eq:Pk}, 
\EQ{\label{eq:pkbd}
L_{k+1}^m p_k &\le 4L_{k+1}^m\, B(L_k)^2  \delta_k + [(L_k^\alpha+1)^{\frac{m}{2}} (2L_k+1)^d p_{k-1}]^2 \\
&\le 4L_{k+1}^m\, (2(L_k+1))^{2d}  \delta_k+ (L_k^m p_{k-1})^2
}
where the second line holds provided $\alpha m/2 + d < m$ which requires $\alpha <2$, and for $L_0$ large enough.  We conclude from \eqref{eq:pkbd} that $L_{k+1}^m p_k\le 1$ if $L_0$ is large. Moreover, due to 
\[
\sum_k 4L_{k+1}^m\, (2(L_k+1))^{2d}  \delta_k<\infty \text{\ \ also\ \ } \sum_k L_{k+1}^m p_k <\infty
\]
which is stronger than the claim.  From the preceding analysis, the parameters need to be in the ranges  $1<\alpha<2$ and $0<\beta<1$.
To summarize, we have obtained this result. 

\begin{prop}
\label{prop:fix E}
Fix $1<\alpha<2$ and $0<\beta<1$. For $L_0=L_0(d,\alpha,\beta)$ large enough, define scales $L_{k+1}=\lceil L_k^\alpha\rceil$ for $k\ge0$. Then for arbitrary $x_0\in \Z^d$ and $E\in\R$, 
\EQ{\label{eq:pk}
\Pbb(\{ \text{all boxes\  } \Lambda_{L_k}(x_0) \text{\  are \ }(\gamma_k,E)\text{-regular}\}) \ge 1- p_k
}
with $0<p_k\le L_{k+1}^{-m}\le L_{0}^{-m \alpha^{k+1} }$ for all $k\ge0$. Here $m>\frac{2d}{2-\alpha} $ and $L_0(d,\alpha,\beta,m)$ is sufficiently large. 
The $p_k$  depend neither on $x_0$ nor on $E$, and $\gamma_k\ge\frac12$ for all $k$. 
\end{prop}

\begin{remark}\label{rem:4}
We shall use below that \eqref{eq:pk} holds as stated for $k\ge1$ if we weaken the non-resonance condition in Definition~\ref{def:Reg} to the following one: 
$\dist(\spec(H_{\Lambda_{L_k}(x_0)}),E)\ge \delta(L_k)/4$. This is due to some room built into Lemma~\ref{lem:ind step}, cf. the factor $4\delta_0^{-1}$ in~\eqref{eq:G L0} 
which improves to $\delta_1^{-1}$ in~\eqref{eq:G L1}. This allows us to replace $\delta(L_k)$ in the resonance width with $\delta(L_k)/4$. 
\end{remark}

An essential feature in the derivation of this result  is {\em stability in the energy}. This means that we can obtain~\eqref{eq:pk} {\em uniformly in an energy interval}
of length half of the resonance width. 

\begin{cor}
\label{cor:fix E}
Under the assumptions of the previous proposition the following holds:  for arbitrary $x_0\in \Z^d$ and $E_*\in\R$, 
\EQ{\label{eq:pk*}
\Pbb(\{ \text{all boxes\  } \Lambda_{L_k}(x_0) \text{\  are \ }(\gamma_k,E)\text{-regular for all\ } E\in [E_*-\delta_k/2, E_*+\delta_k/2] \}) \ge 1- p_k
}
for all $k\ge1$ and the same $p_k$ as above. 
\end{cor}
\begin{proof}
We leave the base case $k=1$ to the reader. 
The inductive step $k-1\to k$ with $k\ge2$, consists of the inequality  (dropping $x_0$ for simplicity)
\EQ{\nn 
& \Pbb(\{ \exists\,  \Lambda_{L_k} \text{\  which is \ }(\gamma_k,E)\text{-singular for some\ } E\in [E_*-\delta_k/2, E_*+\delta_k/2] \}) \\ 
&\le  \Pbb(\{ \exists\,  \Lambda_{L_k} \text{\  with \ } \dist(H_{\Lambda_{L_k}},E)\le \delta_k/2\text{\ for some\ } E\in [E_*-\delta_k/2, E_*+\delta_k/2]\}) \\
&\quad +  \Pbb(\{ \exists\,  \Lambda_{L_k} \text{\  with \ } \dist(H_{\Lambda_{L_k}},E)\ge  \delta_k/2\text{\ for some\ } E\in [E_*-\delta_k/2, E_*+\delta_k/2]\}) \\
&\qquad  \text{\  which is \ }(\gamma_k,E)\text{-singular for the same\ } E  \})\}) \\
&\le  \Pbb(\{ \exists\,  \Lambda_{L_k} \text{\  with \ } \dist(H_{\Lambda_{L_k}},E_*)\le \delta_k \}) \\
&\quad +  \Pbb(\{ \exists\,  \Lambda_{L_k} \text{\  which contains two disjoint \ }L_{k-1}\text{-boxes which are both}\\
&\qquad\qquad (\gamma_{k-1},E)\text{-singular for the same\ } E\in [E_*-\delta_{k-1}/2, E_*+\delta_{k-1}/2] \})\})
}
The final two lines here follow from Lemma~\ref{lem:ind step}, see also Remark~\ref{rem:4}. Note how we widened the $E$-interval in the last line, which makes it clear how
to use the inductive assumption. The proof proceeds exactly as before. 
\end{proof}

This result cannot by itself establish localization, since it only controls the resonance of $H_\Lambda$ with a given energy~$E$ on a {\em single box}~$\Lambda$. Localization requires excluding simultaneous resonances on several disjoint boxes. This in turn allows us to {\em eliminate the energy}~$E$ from these events, and thus estimate them uniformly over all energies.   It suffices to carry out this process on two disjoint boxes, in other words, to show that {\em double resonances} are highly unlikely. The following natural result contains the elimination of energies and absence of double resonances in its proof, but not in the statement. Note, however, that the event of low probability described in the following proposition is uniform in all energies. 

\begin{prop}
\label{prop:NR reg}
Under the assumptions of the previous proposition,  for all $k\ge1$, 
\EQ{\label{eq:double res}
\Pbb(\{ \text{for some\ }E\text{\ a box\ }\Lambda_{L_k}(x_0) \text{\ is nonresonant at\ }E\text{\ but\ }(\gamma_k,E)\text{-singular}\})\le q_k
}
where for any $b>1$ and all $k$,  $q_k\le L_{k+1}^{-b}$ provided $L_0$ is large (and thus $\eps$ is small) enough.  Here {\em nonresonant} is as in Definition~\ref{def:Reg} but with $\delta_k/2$. 
\end{prop}
\begin{proof}
 $E$-{\em nres} stands for nonresonant at energy $E$,  $E$-{\em res} for resonant at~$E$, and {\em sing} for singular,
Let $\Lambda_1=\Lambda_{L_1}(x_0)$ be $E$-nres, i.e., $\dist(\spec(H_{\Lambda_1}),E)\ge \delta_1/2$ but $(\gamma_1,E)$-sing. 
By Lemmas~\ref{lem:ind step} and~\ref{lem:step0} there can be at most one resonant $L_0$-box inside of $\Lambda_{L_1}(x_0)$ (here but only here we measure resonance with $\delta_0$ and not $\delta_0/2$). 
Hence
\EQ{\nn 
&  \Pbb(\{ \Lambda_{L_1}(x_0) \text{\ is \ }E\text{-nres  but\ }(\gamma_1,E)\text{-sing}\}) \\
&\le \Pbb(\{ \Lambda_{L_1}(x_0) \text{\ contains two disjoint\ }L_0\text{-boxes, both\ }E\text{-res}\}) \\
&\le \sum_{\Lambda_{L_1}(x_0)} \;\;\sum_{\Lambda_{L_0}'\subset \Lambda_{L_1}(x_0)} \;\;\sum_{\tilde \Lambda_{L_0}'\subset \Lambda_{L_1}(x_0)\setminus \Lambda_{L_0}'}  
 \!\!\!\!\!\Pbb(\{ \dist(\spec(H_{\tilde \Lambda_{L_0}'}),E_j)\le 2\delta_0 \text{\ for some\ }E_j\in\spec(H_{\Lambda_{L_0}'})\})  \\
&\le 8B(L_1) B(L_0)^2 |Q_{L_1}|^2 |Q_{L_0}|^2 \delta_0 =: q_1
}
In the third line the energy is eliminated by $\delta_0$-closeness of $E$ to  some eigenvalue $E_j$ of $H_{\Lambda_{L_0}'}$, and the fourth line is Wegner's estimate. The sum over $\Lambda_{L_1}(x_0)$ expresses the existence of some $L_1$-box with the stated property. 
At scale $L_k$, $k\ge2$,   and suppressing $x_0$ for simplicity, 
\EQ{\nn 
&  \Pbb(\{ \text{for some\ }E\text{\ a box\ }\Lambda_{L_k} \text{\ is\ }E\text{-nres  but\ }(\gamma_k,E)\text{-sing}\}) \\
&\le \Pbb(\{ \text{for some\ }E\text{\ a box\ }\Lambda_{L_k}  \text{\ is\ }E\text{-nres, contains two disjoint\ }L_{k-1}\text{-boxes, both\ } (\gamma_{k-1},E)\text{-sing}\}) \\
&\le \Pbb(\{ \text{for some\ }E\text{\ a box\ }\Lambda_{L_k} \text{\  contains two disjoint\ }L_{k-1}\text{-boxes, one\ }E\text{-res,  the other\ }(\gamma_{k-1},E)\text{-sing}\}) \\ 
&+ \Pbb(\{ \text{for some\ }E\text{\ a box\ }\Lambda_{L_k} \text{\  contains two disjoint\ }L_{k-1}\text{-boxes,  both\ }E\text{-nres,  but\ } (\gamma_{k-1},E)\text{-sing}\}) \\
}
In analogy with $k=1$ we bound the third line by 
\EQ{\nn 
&\le \sum_{\Lambda_{L_k}} \;\;\sum_{\Lambda_{L_{k-1}}'\subset \Lambda_{L_k}} \; \sum_{ y_0\in  \Lambda_{L_k} }
  \Pbb(\{\text{some box\ }\tilde \Lambda_{L_{k-1}}'(y_0) \subset \Lambda_{L_k}\setminus \Lambda'_{L_{k-1}}\text{\ is\ }(\gamma_{k-1},E)\text{-sing} \\
 &\hspace{5cm} \text{\ with}\,|E-E_j|\le \delta_{k-1}/2, E_j\in\spec(H_{\Lambda_{L_{k-1}}'})\})  \\
 &\le B(L_k)B(L_{k-1})|Q_{L_k}|^2 |Q_{L_{k-1}}|p_{k-1}
}
where the final estimate is given by Corollary~\ref{cor:fix E} with $E_*=E_j$. Note that while $E_j$ is random, these variables are independent from $H_{\tilde \Lambda_{L_{k-1}}'(y_0)}$. Hence we may first condition on the random variables in~$H_{\Lambda_{L_{k-1}}'}$. 
The fourth line above is bounded by the inductive assumption and independence, and so it is 
$
\le B(L_k)|Q_{L_k}|^2 q_{k-1}^2
$.
In summary, by Proposition~\ref{prop:fix E}, 
\EQ{\label{eq:qk}
q_k &\le B(L_k)B(L_{k-1})|Q_{L_k}|^2 |Q_{L_{k-1}}|p_{k-1} + B(L_k)|Q_{L_k}|^2 q_{k-1}^2 \\
&\le (2(L_k+1))^{5d} L_{k}^{-m} + (2(L_k+1))^{3d} q_{k-1}^2
}
and we conclude as for~\eqref{eq:pkbd} that $q_k\le L_{k+1}^{-b}$ for all $b$ provided $L_0$ is taken large enough depending on~$b$. 
\end{proof}

\begin{proof}[Proof of Theorem~\ref{thm:AL}]
Lets $\calB_k(x_0)$ be the event in \eqref{eq:double res}. We remove the $0$-probability event $$\calB=\cup_{x_0\in\Z^d}\limsup_{k\to\infty} \calB_k(x_0).$$ Considering a realization of the random operator $H$ off of this event, for spectrally almost every energy $E$ relative to this operator we can find a nontrivial generalized eigenfunction $H\psi=E\psi$ which is at most polynomially growing, say $|\psi(n)|\le C(\sigma, \psi) |n|^{\sigma}$ with $\sigma>\frac{d}{2}$ and all $n\in\Z^d$, $n\ne0$.   Let $\psi(x_0)\ne0$. 
 Suppose $\Lambda_{L_k}(x_0)$ is $E$-nonresonant for infinitely many $k$. Then by Proposition~\ref{prop:NR reg},  for those $k$, 
 \[
\max_{(y',y)\in\partial\Lambda_{L_k}(x_0)}  |G_{\Lambda_{L_k}(x_0)}(E)(x_0,y')| \le \eps^{\frac12 L_{k}}
 \]
 Then $[(H_{\Lambda_{L_k}(x_0)}-E)\psi](y') = -\sum_{(y',y)\in  \partial\Lambda_{L_k}(x_0)} \psi(y)$, 
 \EQ{\label{eq:Poisson}
 |\psi(x_0)| &\le \sum_{(y',y)\in \partial\Lambda_{L_k}(x_0)} |G_{\Lambda_{L_k}(x_0)}(E)(x_0,y')| |\psi(y)| \le C(\sigma, \psi) \sum_{(y',y)\in \partial\Lambda_{L_k}(x_0)} \eps^{\frac12 L_{k}} L_k^\sigma  
 }
 which is impossible for infinitely many $k$. Hence for all $k\ge k_0(\psi)$, $\Lambda_{L_k}(x_0)$ is $E$-resonant. We now remove another $0$-probability event, namely {\em double resonances } between disjoint boxes which are not too far from each other. To be specific, as above we conclude that, a.s.\ for every~$x_0$ and all but finitely many $k$, 
 \EQ{
 \label{eq:zeroP}
 \forall\, E\in\R\text{\ if \ }\Lambda_{L_k}(x_0)\text{\ is\  }E\text{-res then\ }\forall\;2L_k< |y_0-x_0| \le 100 L_{k+1}, \; 
 \Lambda_{L_k}(y_0)\text{\ is\ }(\gamma_k,E)\text{-reg}
 }
 Indeed, the resonance condition ensures that $E$ is $\delta_k/2$-close to one of the (random) eigenvalues of $H_{\Lambda_{L_k}(x_0)}$, and Corollary~\ref{cor:fix E} bounds the probability that one of the  boxes $\Lambda_{L_k}(y_0)$  is $(\gamma_k,E)$-regular by $L_{k+1}^{-m}$ where $m>2d$, say. Hence we can sum this up over all $y_0$ in a $100L_{k+1}$-box and apply Borel-Cantelli as before. Consequently, {\em all boxes} $\Lambda_{L_k}(y_0)$ are regular as stated in \eqref{eq:zeroP}. By a resolvent expansion as in the proof of Lemma~\ref{lem:ind step}, the reader will easily verify that all Green functions $G_{\Lambda_L(y_0)}(E)(x,y)$ have exponential decay if $\Lambda_L(y_0)\subset \Lambda_{100 L_{k+1}}(x_0)\setminus  \Lambda_{2 L_{k}}(x_0)$ where we take $L_{k+1}\le L\le 50L_{k+1}$.  By an estimate as in \eqref{eq:Poisson} one now concludes exponential decay of~$\psi$. 
 \end{proof}

\section{The one-dimensional quasi-periodic model}
\label{sec:FSW}

\subsection{The Fr\"ohlich-Spencer-Wittwer theorem: even potentials}

In this section we will provide a fairly complete proof sketch of the following result due to Fr\"ohlich, Spencer, and Wittwer~\cite{FSW}. The dynamics (rotation) $T_\omega\theta=\theta+\omega\mod\,1$ takes place on the torus $\tor=\R/\Z$, and all ``randomness" sits in a single parameter, namely $\theta\in\tor$.  The one-dimensional {\em random} model is treated by completely different techniques, starting from F\"urstenberg's classical theorem on positive Lyapunov exponents for random $SL(2,\R)$ cocycles, cf.~\cite{Viana} for a comprehensive exposition of this fundamental result as well as Lyapunov exponents in general. See  the recent papers~\cite{B7}, \cite{GorK}, and~\cite{JitZ} for streamlined elegant treatments of the $1$-dimensional random Anderson model, including the Bernoulli case. For quasi-periodic (and other highly correlated) cocycles, F\"urstenberg's global theorem does not apply, and other techniques must be used. The proof of the following result will in fact be perturbative.

\begin{thm}
\label{thm:FSW}
Let $v\in C^2(\tor)$ be even, with exactly two nondegenerate critical points. Define
\EQ{\label{eq:qp op}
H_{\eps}(\theta) &= \eps^2 \Delta_{\Z} + V_\theta, \quad V_\theta(n) = v(T_\omega^n\theta) \;\; \forall\; n\in\Z
}
where $\omega\in\tor$ is Diophantine, viz.~$\| n\omega\|\ge c_0\, n^{-2}$ for all $n\ge1$ with some $c_0>0$. There exists $\eps_0(c_0,v)$ such that for all $0<\eps\le \eps_0$ the operators $H_{\theta,\eps}$ exhibit Anderson localization for a.e.~$\theta\in\tor$. 
\end{thm}

The evenness assumption allows for substantial simplifications as we shall see. Note that it entails that $V$ is symmetric about~$\frac12$. Theorem~\ref{thm:FSW} cannot hold for all $\theta$, see~\cite{JS}. 
As in the previous section, we shall drop the index~$\eps$ and simply write $H(\theta)$ for~\eqref{eq:qp op}, and $H_\Lambda(\theta)$ for its finite volume version. It is important to keep track of~$\theta$ so we include it in the notation (while in the random case we could drop the $\omega$, the variable in the probability space). 
\begin{figure}[ht]
\tikzset{every picture/.style={line width=0.75pt}} 

\begin{tikzpicture}[x=0.55pt,y=0.55pt,yscale=-1,xscale=1]

\draw    (40.91,1.2) -- (40.91,301.2) ;
\draw    (40.91,151.2) -- (468.41,149.66) ;
\draw [line width=2.25]    (41.91,51.2) .. controls (198.91,49.2) and (141.91,280.2) .. (250.91,280.2) ;
\draw [line width=2.25]    (250.91,280.2) .. controls (358.91,280.53) and (310.91,50.53) .. (467.91,50.53) ;
\draw    (467.91,-0.47) -- (468.91,299.78) ;
\draw [color={rgb, 255:red, 208; green, 2; blue, 27 }  ,draw opacity=1 ][line width=1.5]  [dash pattern={on 5.63pt off 4.5pt}]  (39.91,80.87) -- (468.91,79.87) ;
\draw [color={rgb, 255:red, 208; green, 2; blue, 27 }  ,draw opacity=1 ][line width=1.5]  [dash pattern={on 5.63pt off 4.5pt}]  (40.91,98.87) -- (468.91,99.87) ;
\draw [line width=1.5]  [dash pattern={on 1.69pt off 2.76pt}]  (119,81) -- (117.91,149.87) ;
\draw [line width=1.5]  [dash pattern={on 1.69pt off 2.76pt}]  (134.91,98.87) -- (134.91,150.87) ;
\draw [line width=1.5]  [dash pattern={on 1.69pt off 2.76pt}]  (389.77,78.17) -- (389.5,149) ;
\draw [line width=1.5]  [dash pattern={on 1.69pt off 2.76pt}]  (373,99) -- (372.77,149.17) ;
\draw [color={rgb, 255:red, 208; green, 2; blue, 27 }  ,draw opacity=1 ][line width=3]    (117.91,149.87) -- (136.63,149.78) ;
\draw [color={rgb, 255:red, 208; green, 2; blue, 27 }  ,draw opacity=1 ][line width=3]    (372.77,149.17) -- (389.77,149.17) ;

\draw (27,152) node   [align=left] {0};
\draw (481,151) node   [align=left] {1};
\draw (78,37) node  [rotate=-359.32,xslant=0.04]  {$V( \theta )$};
\draw (251,169) node    {$\frac{1}{2}$};
\draw (128,172) node  [font=\large]  {$J_{1}$};
\draw (383,172) node  [font=\large]  {$J_{2}$};
\draw (501,107) node  [rotate=-359.83]  {$E_{\star } -\delta _{0}$};
\draw (502,74) node    {$E_{\star } +\delta _{0}$};

\end{tikzpicture}
\caption{The potential and energy strip at the initial step}\label{fig:V}
\end{figure}
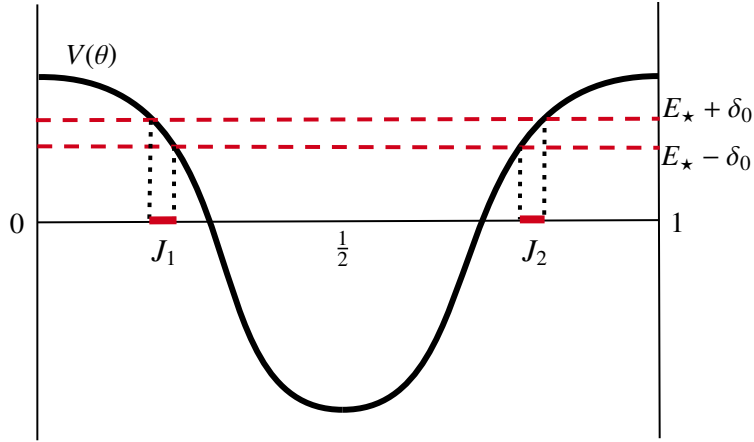
Fix $\theta_*\in\tor$ and $E_*\in\R$. 
The singular sites relative to $\theta_*, E_*$ are defined as 
\EQ{\label{eq:S0}
\calS_0=\calS_0(\theta_*,E_*) &:= \{ n\in\Z\:|\: |V(\theta_*+n\omega)-E_*|\le \delta_0\}  \\
& = \{ n\in\Z\:|\: T^n\theta_* \in V^{-1}([E_*-\delta_0,E_*+\delta_0]) \}
}
Figure~\ref{fig:V} shows one scenario for which $V^{-1}([E_*-\delta_0,E_*+\delta_0])=J_1\cup J_2$ with disjoint intervals. There might be just one interval or the set could be empty. By our assumption of $V$ Morse, $\max_{i=1,2}|J_i|\le C_0(v)\delta_0^{\frac12}$ for all cases. We choose the {\em resonance width} $\delta_0=A_0\eps$ with a large constant $A_0$. We investigate the structure of~$\calS_0$ by means of the example $V(\theta)=\cos(2\pi \theta)$. If $k,\ell \in\calS_0$ are distinct, then 
\[
|\sin(\pi(k-\ell)\omega) \sin(\pi(2\theta_*+(k+\ell)\omega))|\le \delta_0
\]
which implies for small $\delta_0$ that 
\EQ{\nn 
m(k,\ell):=\min( \| (k-\ell)\omega\|, \| 2\theta_*+(k+\ell)\omega\|)\le 2\sqrt{\delta_0}
}
The first alternative here, viz.~$ \| (k-\ell)\omega\|\le 2\sqrt{\delta_0}$ occurs precisely if both $T^k_\omega\theta_*$ and $T^\ell_\omega\theta_*$ fall into $J_1$, or both fall into $J_2$. The second one occurs if they fall into different intervals.  The Diophantine assumption implies that 
\[
c_0|k-\ell|^{-2}\le \| (k-\ell)\omega\|\le 2\sqrt{\delta_0},\qquad |k-\ell|\gtrsim \delta_0^{-\frac14}
\]
Henceforth $\gtrsim$ and $\lesssim$ will indicate multiplicative constants depending on $c_0,v$. On the other hand, $\| 2\theta_*+(k+\ell)\omega\|\le 2\sqrt{\delta_0}$ might occur for $\ell=k+1$ which is the case if $T_\omega(J_1)\cap J_2\ne\emptyset$. It is clear that the  function $m$ appears not just for cosine, but in fact for any $v$ as in the theorem. 

\begin{lem}
\label{lem:step 0}
Any two distinct $k,\ell\in \calS_0$ satisfy
$
m(k,\ell)^2\lesssim \delta_0, 
$
and any three distinct points $k,\ell,n\in\calS_0$ satisfy 
\[
\max(|k-\ell|, |\ell -n|)\gtrsim \delta_0^{-\frac14}
\]
\end{lem}
\begin{proof}
The argument is essentially the same as for cosine, the trigonometric identities being replaced by the symmetry of~$V$ about~$\frac12$: if $\theta_*+k\omega\in J_1$ and $\theta_*+\ell\omega\in J_2=-J_1\mod 1$, then $2\theta_*+(k+\ell)\omega\in J_1-J_1\mod 1$ whence $\| 2\theta_*+(k+\ell)\omega\|^2 \lesssim \delta_0$. 
\end{proof}
\begin{figure}[ht]
\input{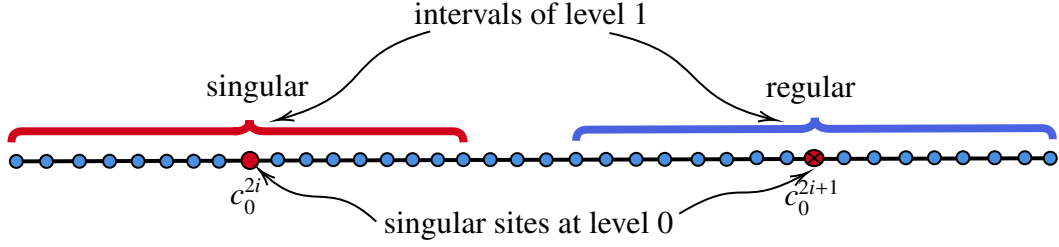}
\caption{Simple resonances at level $1$}
\label{fig:simpleres}
\end{figure} 
\begin{defn}
\label{defi:s0crit}
We label $\calS_0=\{c^i_0\}_{i=-\infty}^\infty$ in increasing order (assuming $\calS_0\ne\emptyset$). These are the singular sites (or  singular ``intervals") at level~$0$. 
Let $s_0:=\min\{ c^i_0- c^{j}_0\: |\: i> j\}$. If $s_0\ge 4|\log \eps|^2$, then we speak of a {\em simple resonance}, otherwise of a {\em double resonance}, both at level~$1$. In the latter case, we replace $\calS_0$ with  $\bar\calS_0=\calS_0\cup(\calS_0+s_0)$, which we again label as $\{c^i_0\}_{i=-\infty}^\infty$. 
In the simple resonant case, we let $I^i_1$ be an interval of length $\ell_1:=\lceil \log(1/\eps)\rceil^2$ centered at $c^{i}_1:=c^i_0$, in the double resonant case $I^i_1$ has length $\ell_1:=\lceil \log(1/\eps)\rceil^4$, centered at $c^{i}_1:=(c^{2i}_0+c^{2i+1}_0)/2\in \frac12\Z$.  By construction, all of the $I^i_1$ are pairwise disjoint, and each $c^i_0$ is contained in a unique interval at level~$1$. We classify those intervals $I^i_1$ as {\em singular} provided 
\EQ{\label{eq:delta1}
\dist(\spec(H_{I^i_1}(\theta_*)), E_*)\le \delta_1:= \eps^{\ell_1^{2/3}}
}
and $\calS_1:=\{c^i_1\:|\: I^i_1\text{\ is singular}\}$. All other intervals $I^i_1$ are called {\em regular}. 
\end{defn}

We shall see later, based on Theorem~\ref{thm:Ber}, that for spectrally a.e.\ energy $E\in\spec(H(\theta))$ the set of singular intervals, which are constructed iteratively at all levels (see below), is not empty. 
Figures~\ref{fig:simpleres}, resp.~\ref{fig:doubleres} illustrate the two cases, with the blue dots being $\Z\setminus\calS_0$. The terminology simple/double resonance is derived from the structure of the eigenfunctions at level~$1$ associated with the operators $H_{I^i_1}(\theta_*)$ and the unique (as we shall see) eigenvalue $E^i_1(\theta_*)$ satisfying~\eqref{eq:delta1}. In the simple resonance case, the eigenfunction has most (say $99\%$) of its $\ell^2$ mass at the center~$c^i_0$, whereas in the double resonant case it may have significant mass at both sites $c^{2i}_0$ and $c^{2i+1}_0$. 

Figure~\ref{fig:simpleres} depicts only one of four possibilities for the intervals at level~$1$, they might both be singular, both regular, or the order could be reversed. The red dot with $\otimes$ is supposed to indicate a return of the trajectory $T^j_\omega\theta_*$ to $J_2$, whereas the red dot on the left a return to~$J_1$, cf.~Figure~\ref{fig:V}.  While the distance between these two red dots is required to be at least $4(\log(1/\eps))^2$, the Diophantine condition forces two red dots of the same kind (associated with $J_1$, resp.~$J_2$) to be separated by at least on the order of~$\delta_0^{-\frac14}$. This is much larger than the length $\ell_1=\lceil\log(1/\eps)\rceil^2$ of the intervals $I^i_1$.

The reason for passing form $\calS_0$ to $\bar\calS_0$ lies with the self-symmetry indicated in Figure~\ref{fig:doubleres} (i.e., $c^{2i+1}_0-c^{2i}_0$ does not depend on~$i$). To see this, note that by definition of $s_0$ there exist $k_1,k_2\in\calS_0$ with $\theta_*+k_i\omega\in J_i$ with $i=1,2$ and $s_0=|k_1-k_2|\le 4(\log(1/\eps))^2$. We are again using the Diophantine condition here to ensure that we do not fall into the same interval (as a standing assumption $\eps$ needs to be small enough depending on $v$ and $c_0$ so as to guarantee this). Next, take any $k\in\calS_0$ with $\theta_*+k\omega\in J_1$ (everything modulo integers which will be henceforth understood tacitly). Then $\theta_*+(k+s_0)\omega\in \tilde J_2$, where $\tilde J_2$ has the same center as $J_2$ and twice the length. On the other hand, it might be that $\theta_*+(k+s_0)\omega\not \in J_2$, but we must still include $k+s_0$ in $\calS_0$ for the construction to work. In fact, Lemma~\ref{lem:step 0} remains valid for $\bar\calS_0$ and the defining inequality~\eqref{eq:S0} is modified only slightly, viz. $|V(\theta_*+n\omega)-E_*|\lesssim \delta_0$ for all $n\in\bar\calS_0$. 

We will establish the following analogue of Lemma~\ref{lem:step 0} at level~$1$. We emphasize again that this  statement only exists for even~$V$. 

\begin{lem}
\label{lem:step 1}
For all $c^i_1, c^j_1\in\calS_1$
one has $
m(c^i_1, c^j_1)^2\lesssim \delta_1$, 
with an absolute implied constant. 
\end{lem}
\begin{figure}
\input{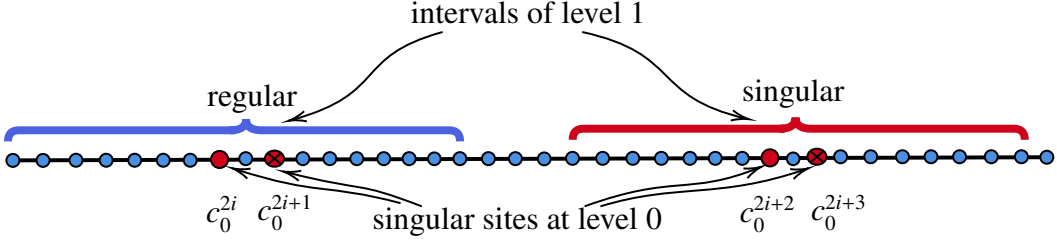}
\caption{Double resonances at level $1$}
\label{fig:doubleres}
\end{figure}
The idea is to carry out a similar argument as for Lemma~\ref{lem:step 0},  with the potential function $V$ replaced by the parametrizations of the eigenvalues of the operator $H_{I^i_1}(\theta)$ {\em localized both near} $\theta_*$ and~$E_*$. This stability hinges crucially on a {\em spectral gap} or on the {\em separation of the eigenvalues}.  The latter can be seen as a quantitative version of the simplicity of the Dirichlet spectrum of Sturm-Liouville operators, such as $H_I(\theta)$. Before discussing the details of Lemma~\ref{lem:step 1}, we exhibit the entire strategy of the proof of Theorem~\ref{thm:FSW}. 

\begin{itemize}
\item In analogy with Definition~\ref{defi:s0crit} define regular and singular intervals at level~$n\ge2$. More specifically, for $n\ge1$ set 
\[
s_n:=\min\{|c^i_n-c^j_n|\:|\: c^i_n, c^j_n\in\calS_n, \; i\ne j\} 
\]
If $s_n> 4\ell_n^2$, then we call this a simple resonance and define $c^i_{n+1}=c^i_n$ for all $i$ and $\ell_{n+1}=\ell_n^2$, otherwise for the double resonance we set $c^i_{n+1}=(c^{2i}_n+c^{2i+1}_n)/2\in \frac12\Z$, $\ell_{n+1}=\ell_n^4$,  and also augment $\calS_n$ to $\bar\calS_n$ by including the mirror image of each $I^i_n$ if it was not already included in~$\calS_n$. By mirror image we mean the reflection about $0$, which is the same as the reflection about~$1/2$ modulo~$\Z$. By construction, the $I^i_{n+1}$ are pairwise disjoint and each $c^i_n$ is contained in a unique interval at level~$n+1$. An interval $I^i_{n+1}$ centered at~$c^i_{n+1}$ is called {\em singular} if 
\EQ{\label{eq:sing n}
\dist(\spec(H_{I^i_{n+1}}(\theta_*)),E_*)\le \delta_{n+1}=\eps^{\ell_{n+1}^{2/3}}
}
and regular otherwise. Define $\calS_{n+1}$ to be the centers of the singular intervals. One can arrange for $\partial I^i_{n+1}$ for all singular not to meet any singular interval of level~$m$ with $m\le n$. 
\item  An arbitrary interval $\Lambda\subset\Z$ is called $n$-regular provided every point in $\Lambda\cap \calS_0$ is contained in a regular interval $I^j_m\subset\Lambda$ for some $m\le n$, cf.\ Figure~\ref{fig:tree}.  Note that every singular point at level~$0$ is either (i)  contained in infinitely many singular intervals $I^{j_n}_n$ for each $n\ge0$ or (ii) contained in a finite number of such intervals at successive levels followed by a regular one.  By induction on scales one proves the following crucial decay and stability property of the Green function associated with $n$-regular intervals~$\Lambda$: $|G_\Lambda(\theta,E)(x,y)|\le \eps^{\frac12|x-y|}$ for all $x,y\in\Lambda$, $|x-y|\ge \ell_n^{5/6}$, $|E-E_*|\les \delta_n$, and $|\theta-\theta_*|\les \delta_n$. 
\item One has 
\EQ{\label{eq:mn}
m(c^i_n, c^j_n)^2\les |E^i_n(\theta_*)-E^j_n(\theta_*)| \les \delta_n\qquad \forall\;c^i_n,c^j_n\in\calS_n
} for all $n\ge0$. Here $E^i_n(\theta_*)$, are the unique eigenvalues of $H_{I^i_n}(\theta_*)$ in the interval  $[E_*-c\delta_n,E_*+c\delta_n]$ with $c$ small. This hinges crucially on the  separation property of the eigenvalues, see Lemma~\ref{lem:sep}. For simple resonances, we will use first order eigenvalue perturbation theory, and  for double resonances, second order perturbation theory.  
\item Based on the estimate on $m$, we prove Theorem~\ref{thm:FSW} by double resonance elimination as in the previous section. In analogy with Theorem~\ref{thm:AL} we start with the polynomially bounded  Fourier basis provided by Theorem~\ref{thm:Ber}, find an increasing nested family of resonant intervals which are resonant at the given energy, and thus due to the  elimination of double resonances obtain exponential decay at all scale. The main departure from the proof of Theorem~\ref{thm:AL} lies with the application of Borel-Cantelli to remove a zero measure set of bad~$\theta\in\tor$. 
\end{itemize}
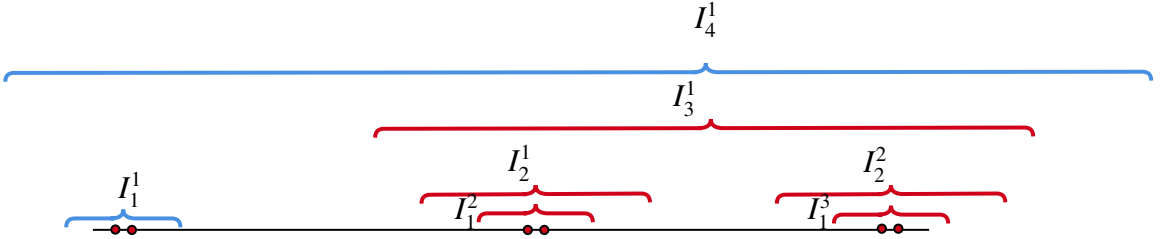
\begin{figure}[ht]
\tikzset{every picture/.style={line width=0.75pt}} 

\begin{tikzpicture}[x=0.48pt,y=0.48pt,yscale=-1,xscale=1]
\draw    (0.51,289.67) -- (659.51,289.67) ;
\draw  [fill={rgb, 255:red, 208; green, 2; blue, 27 }  ,fill opacity=1 ] (15.03,289.43) .. controls (15.03,287.64) and (16.48,286.18) .. (18.27,286.18) .. controls (20.06,286.18) and (21.51,287.64) .. (21.51,289.43) .. controls (21.51,291.22) and (20.06,292.67) .. (18.27,292.67) .. controls (16.48,292.67) and (15.03,291.22) .. (15.03,289.43) -- cycle ;
\draw  [fill={rgb, 255:red, 208; green, 2; blue, 27 }  ,fill opacity=1 ] (28.03,289.76) .. controls (28.03,287.96) and (29.48,286.51) .. (31.27,286.51) .. controls (33.06,286.51) and (34.51,287.96) .. (34.51,289.76) .. controls (34.51,291.55) and (33.06,293) .. (31.27,293) .. controls (29.48,293) and (28.03,291.55) .. (28.03,289.76) -- cycle ;
\draw  [fill={rgb, 255:red, 208; green, 2; blue, 27 }  ,fill opacity=1 ] (340.03,289.76) .. controls (340.03,287.96) and (341.48,286.51) .. (343.27,286.51) .. controls (345.06,286.51) and (346.51,287.96) .. (346.51,289.76) .. controls (346.51,291.55) and (345.06,293) .. (343.27,293) .. controls (341.48,293) and (340.03,291.55) .. (340.03,289.76) -- cycle ;
\draw  [fill={rgb, 255:red, 208; green, 2; blue, 27 }  ,fill opacity=1 ] (353.03,289.76) .. controls (353.03,287.96) and (354.48,286.51) .. (356.27,286.51) .. controls (358.06,286.51) and (359.51,287.96) .. (359.51,289.76) .. controls (359.51,291.55) and (358.06,293) .. (356.27,293) .. controls (354.48,293) and (353.03,291.55) .. (353.03,289.76) -- cycle ;
\draw  [fill={rgb, 255:red, 208; green, 2; blue, 27 }  ,fill opacity=1 ] (632.03,288.76) .. controls (632.03,286.96) and (633.48,285.51) .. (635.27,285.51) .. controls (637.06,285.51) and (638.51,286.96) .. (638.51,288.76) .. controls (638.51,290.55) and (637.06,292) .. (635.27,292) .. controls (633.48,292) and (632.03,290.55) .. (632.03,288.76) -- cycle ;
\draw  [fill={rgb, 255:red, 208; green, 2; blue, 27 }  ,fill opacity=1 ] (619.03,288.76) .. controls (619.03,286.96) and (620.48,285.51) .. (622.27,285.51) .. controls (624.06,285.51) and (625.51,286.96) .. (625.51,288.76) .. controls (625.51,290.55) and (624.06,292) .. (622.27,292) .. controls (620.48,292) and (619.03,290.55) .. (619.03,288.76) -- cycle ;
\draw  [color={rgb, 255:red, 74; green, 144; blue, 226 }  ,draw opacity=1 ][line width=1.5]  (69.51,287.67) .. controls (69.51,283) and (67.18,280.67) .. (62.51,280.67) -- (34.51,280.67) .. controls (27.84,280.67) and (24.51,278.34) .. (24.51,273.67) .. controls (24.51,278.34) and (21.18,280.67) .. (14.51,280.67)(17.51,280.67) -- (-13.49,280.67) .. controls (-18.16,280.67) and (-20.49,283) .. (-20.49,287.67) ;
\draw  [color={rgb, 255:red, 208; green, 2; blue, 27 }  ,draw opacity=1 ][line width=1.5]  (394.51,283.67) .. controls (394.51,279) and (392.18,276.67) .. (387.51,276.67) -- (359.51,276.67) .. controls (352.84,276.67) and (349.51,274.34) .. (349.51,269.67) .. controls (349.51,274.34) and (346.18,276.67) .. (339.51,276.67)(342.51,276.67) -- (311.51,276.67) .. controls (306.84,276.67) and (304.51,279) .. (304.51,283.67) ;
\draw  [color={rgb, 255:red, 208; green, 2; blue, 27 }  ,draw opacity=1 ][line width=1.5]  (674.51,284.67) .. controls (674.51,280) and (672.18,277.67) .. (667.51,277.67) -- (639.51,277.67) .. controls (632.84,277.67) and (629.51,275.34) .. (629.51,270.67) .. controls (629.51,275.34) and (626.18,277.67) .. (619.51,277.67)(622.51,277.67) -- (591.51,277.67) .. controls (586.84,277.67) and (584.51,280) .. (584.51,284.67) ;
\draw  [color={rgb, 255:red, 208; green, 2; blue, 27 }  ,draw opacity=1 ][line width=1.5]  (719.51,267.98) .. controls (719.48,263.31) and (717.14,260.99) .. (712.47,261.02) -- (639.47,261.42) .. controls (632.8,261.46) and (629.46,259.15) .. (629.44,254.48) .. controls (629.46,259.15) and (626.14,261.5) .. (619.47,261.54)(622.47,261.52) -- (546.47,261.94) .. controls (541.8,261.97) and (539.48,264.31) .. (539.51,268.98) ;
\draw  [color={rgb, 255:red, 208; green, 2; blue, 27 }  ,draw opacity=1 ][line width=1.5]  (439.51,268) .. controls (439.48,263.33) and (437.14,261.01) .. (432.47,261.04) -- (359.47,261.45) .. controls (352.8,261.48) and (349.46,259.17) .. (349.44,254.5) .. controls (349.46,259.17) and (346.14,261.52) .. (339.47,261.56)(342.47,261.54) -- (266.47,261.97) .. controls (261.8,262) and (259.48,264.34) .. (259.51,269.01) ;
\draw  [color={rgb, 255:red, 208; green, 2; blue, 27 }  ,draw opacity=1 ][line width=1.5]  (741.51,215.98) .. controls (741.5,211.31) and (739.17,208.98) .. (734.5,208.99) -- (497.52,209.45) .. controls (490.85,209.46) and (487.51,207.14) .. (487.5,202.47) .. controls (487.51,207.14) and (484.19,209.47) .. (477.52,209.48)(480.52,209.48) -- (229.99,209.96) .. controls (225.32,209.97) and (222.99,212.3) .. (223,216.97) ;
\draw  [color={rgb, 255:red, 74; green, 144; blue, 226 }  ,draw opacity=1 ][line width=1.5]  (834.51,171.98) .. controls (834.51,167.31) and (832.18,164.98) .. (827.51,164.98) -- (493.52,165.33) .. controls (486.85,165.34) and (483.52,163.01) .. (483.51,158.34) .. controls (483.52,163.01) and (480.19,165.35) .. (473.52,165.36)(476.52,165.35) -- (-61.49,165.92) .. controls (-66.16,165.92) and (-68.49,168.25) .. (-68.49,172.92) ;

\draw (30,257.98) node  [font=\large]  {$I^{1}_{1}$};
\draw (295,274.98) node  [font=\large]  {$I^{2}_{1}$};
\draw (573,274.98) node  [font=\large]  {$I^{3}_{1}$};
\draw (337,236.98) node  [font=\large]  {$I^{1}_{2}$};
\draw (617,237.98) node  [font=\large]  {$I^{2}_{2}$};
\draw (467,185.98) node  [font=\large]  {$I^{1}_{3}$};
\draw (484,123.98) node  [font=\large]  {$I^{1}_{4}$};
\end{tikzpicture}
\caption{Regular and singular intervals of $4$ levels}\label{fig:tree}
\end{figure}
We begin with the Green function decay on regular intervals (this is the analogue of the regular Green function from Definition~\ref{def:Reg}). We set $\ell_0:=\lceil \log(1/\eps)\rceil$. 

\begin{lem}
\label{lem:G I reg}
For all $n$-regular intervals~$\Lambda$, $n\ge0$,  one has $|G_\Lambda(\theta,E)(x,y)|\le \eps^{\gamma_n |x-y|}$ for all $x,y\in\Lambda$, $|x-y|\ge \ell_n^{5/6}$, $|E-E_*|\les \delta_n$, and $|\theta-\theta_*|\les \delta_n$. The $\gamma_n$ decrease, but $\gamma_n\ge\frac12$ for all $n$. 
\end{lem}
\begin{proof}
At $n=0$ the interval $\Lambda$ contains only regular lattice points, i.e., blue dots in the figures above. Then the Neumann series argument from Lemma~\ref{lem:step0} 
implies that, for $\delta_0=A_0\eps$ with $A_0$ large enough, and for all $|\theta-\theta_*|\ll \delta_0$, $|E-E_*|\ll \delta_0$ (meaning up to a small multiplicative constant), 
\[
|G_\Lambda(\theta,E)(x,y)|\le \delta_0^{-1} (2\eps^2\delta_0^{-1})^{|x-y|} \le \eps^{|x-y|-1}\quad\forall\; x,y\in\Lambda
\]
We are using that $|V(\theta_*+k\omega)-E_*|\gtrsim\delta_0$ implies that $|V(\theta+k\omega)-E|\gtrsim\delta_0$ in the specified range of parameters. 
If $\Lambda$ is $1$-regular, then let $\{ I^i_1\}_{i=i_0}^{i_1}$ be a complete list of all level $1$ intervals, in increasing order,  which cover all points in~$\Lambda\cap \calS_0$.  By construction, $I^i_1\subset\Lambda$ for all $i_0\le i\le i_1$. 
The intervals $I^i_1$ (which are all regular) do not really come off the axis in Figure~\ref{fig:1reg}, they are only depicted in  this way to indicate that they are level~$1$ intervals. The line segment is supposed to depict $\Lambda$ and it consists entirely of regular lattice points at level~$0$ apart from the red singular sites. For the double resonance case shown in Figure~\ref{fig:1reg},  one  red pairs is separated from another by   $\gtrsim\delta_0^{-\frac14}\simeq\eps^{-\frac14}$, which is  much larger than the  $I^i_1$ which are of length~$(\log\eps)^4$.  On the other hand, in the single resonant case recall that the $I^i_1$ are of length $\lceil\log(1/\eps)\rceil^2$, and the separation between the singular sites in $\calS_0$ at least $4(\log\eps)^2$ (but possibly much longer). 

These long sections consisting entirely of regular lattice points between singular pairs in the double resonance case, resp.\ singular sites in the simple resonance case,  allow us to iterate the resolvent identity similar to Lemma~\ref{lem:ind step}.   For general $n$, it is essential to use~\eqref{eq:mn} up to level $n-1$ in order to achieve this separation.  See Appendix~A in~\cite{FSW} for the details. 
\end{proof}

\begin{proof}[Proof of Theorem~\ref{thm:FSW}]
For any $\theta\in\tor$, by Theorem~\ref{thm:Ber} for spectrally a.e.~$E\in\R$ there is a generalized eigenfunction $H(\theta)\psi=E\psi$ with at most linear growth. For any such $E,\psi$ we claim that there exists $N=N(\theta,\psi)\ge1$ so that all intervals $\Lambda_n=[-2\ell_n,2\ell_n]$ are $n$-singular for $(\theta,E)$ provided $n\ge N$. If $\Lambda_n$ is $n$-regular for infinitely many $n$, then by the Poisson formula~\eqref{eq:Poisson} for any $j$ and large $n$, 
\[
 |\psi(j)| \le \sum_{(k',k)\in \partial\Lambda_n} |G_{\Lambda_{n}}(E)(j,k')| |\psi(k)| \le C\eps^{\frac12(2\ell_n-|j|)} \ell_n
\]
Taking the limit $n\to\infty$ yields $\psi\equiv0$, whence our claim. Next, we claim that 
\EQ{\label{eq:sing claim}
\Lambda_n\cap I^i_n\ne\emptyset\text{\ for some singular\ }I^i_n
} for large~$n$. 
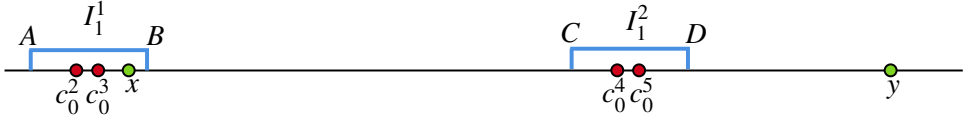
\begin{figure}[ht]
\tikzset{every picture/.style={line width=0.75pt}} 

\begin{tikzpicture}[x=0.55pt,y=0.55pt,yscale=-1,xscale=1]

\draw [line width=0.75]    (0.51,160.2) -- (659.51,160.2) ;
\draw  [fill={rgb, 255:red, 208; green, 2; blue, 27 }  ,fill opacity=1 ] (45.8,160.1) .. controls (45.8,157.84) and (47.64,156) .. (49.9,156) .. controls (52.16,156) and (54,157.84) .. (54,160.1) .. controls (54,162.36) and (52.16,164.2) .. (49.9,164.2) .. controls (47.64,164.2) and (45.8,162.36) .. (45.8,160.1) -- cycle ;
\draw  [fill={rgb, 255:red, 208; green, 2; blue, 27 }  ,fill opacity=1 ] (60.8,160.1) .. controls (60.8,157.84) and (62.64,156) .. (64.9,156) .. controls (67.16,156) and (69,157.84) .. (69,160.1) .. controls (69,162.36) and (67.16,164.2) .. (64.9,164.2) .. controls (62.64,164.2) and (60.8,162.36) .. (60.8,160.1) -- cycle ;
\draw  [fill={rgb, 255:red, 208; green, 2; blue, 27 }  ,fill opacity=1 ] (432.8,160.1) .. controls (432.8,157.84) and (434.64,156) .. (436.9,156) .. controls (439.16,156) and (441,157.84) .. (441,160.1) .. controls (441,162.36) and (439.16,164.2) .. (436.9,164.2) .. controls (434.64,164.2) and (432.8,162.36) .. (432.8,160.1) -- cycle ;
\draw  [fill={rgb, 255:red, 208; green, 2; blue, 27 }  ,fill opacity=1 ] (417.8,160.1) .. controls (417.8,157.84) and (419.64,156) .. (421.9,156) .. controls (424.16,156) and (426,157.84) .. (426,160.1) .. controls (426,162.36) and (424.16,164.2) .. (421.9,164.2) .. controls (419.64,164.2) and (417.8,162.36) .. (417.8,160.1) -- cycle ;
\draw  [fill={rgb, 255:red, 126; green, 211; blue, 33 }  ,fill opacity=1 ] (81.8,160.1) .. controls (81.8,157.84) and (83.64,156) .. (85.9,156) .. controls (88.16,156) and (90,157.84) .. (90,160.1) .. controls (90,162.36) and (88.16,164.2) .. (85.9,164.2) .. controls (83.64,164.2) and (81.8,162.36) .. (81.8,160.1) -- cycle ;
\draw  [fill={rgb, 255:red, 126; green, 211; blue, 33 }  ,fill opacity=1 ] (605.8,160.1) .. controls (605.8,157.84) and (607.64,156) .. (609.9,156) .. controls (612.16,156) and (614,157.84) .. (614,160.1) .. controls (614,162.36) and (612.16,164.2) .. (609.9,164.2) .. controls (607.64,164.2) and (605.8,162.36) .. (605.8,160.1) -- cycle ;
\draw  [color={rgb, 255:red, 74; green, 144; blue, 226 }  ,draw opacity=1 ][line width=1.5]  (17.51,146.2) -- (98.51,146.2) -- (98.51,160.2) ;
\draw [color={rgb, 255:red, 74; green, 144; blue, 226 }  ,draw opacity=1 ][line width=1.5]    (18.51,146.2) -- (18.51,160.2) ;
\draw  [color={rgb, 255:red, 74; green, 144; blue, 226 }  ,draw opacity=1 ][line width=1.5]  (389.51,145.2) -- (470.51,145.2) -- (470.51,159.2) ;
\draw [color={rgb, 255:red, 74; green, 144; blue, 226 }  ,draw opacity=1 ][line width=1.5]    (390.51,145.2) -- (390.51,159.2) ;

\draw (88,171) node    {$x$};
\draw (612,173) node    {$y$};
\draw (63,124) node    {$I^{1}_{1}$};
\draw (436,127) node    {$I^{2}_{1}$};
\draw (44,178) node    {$c^{2}_{0}$};
\draw (65,178) node    {$c^{3}_{0}$};
\draw (420,177) node    {$c^{4}_{0}$};
\draw (439,177) node    {$c^{5}_{0}$};
\draw (17,136) node    {$A$};
\draw (104,136) node    {$B$};
\draw (390,134) node    {$C$};
\draw (476,135) node    {$D$};

\end{tikzpicture}
\caption{A $1$-regular interval $\Lambda$}\label{fig:1reg}
\end{figure}
Since $\Lambda_n$ is not $n$-regular, pick some $c^j_0\in \calS_0\cap \Lambda_n$ (this cannot be empty by Lemma~\ref{lem:G I reg}). By the recursive construction of  singular intervals, see~\eqref{eq:sing n} there is the following dichotomy: either, 
there exists $1\le m\le  n$ with 
\[
c^j_0\in I^{j_1}_1 \subset \ldots \subset I^{j_m}_m 
\]
and $ I^{j_m}_m$ is regular, or $m=n$ with $I^{j_n}_n$ singular. If the first alternative occurs for every $c^j_0\in \calS_0\cap \Lambda_n$, then we may slightly enlarge $\Lambda_n$ to some $\tilde\Lambda_n \subset[-3\ell_n,3\ell_n]$ which is $n$-regular. This is again impossible for large $n$ and so~\eqref{eq:sing claim} holds.  If $\Lambda_n':=[-\ell_{n+1},\ell_{n+1}]$ contains another singular interval at level~$n$, say $I^j_n$, then by \eqref{eq:mn} one has 
$
m(c^i_n,c^j_n)^2\les \delta_n$. But $\|(c^i_n-c^j_n)\omega\|\les \delta_n^{\frac12}$ is impossible by the Diophantine condition whence
\EQ{\label{eq:reflec}
\|2\theta+(c^i_n+c^j_n)\omega\|\les \delta_n^{\frac12}
}
Given that there are at most $\les\ell_{n+1}^2$ many choices of $c^i_n,c^j_n\in \Lambda_n'$, it follows that the measure of $\theta$ as in~\eqref{eq:reflec} is $\les  \ell_{n+1}^2\delta_n^{\frac12}$. This can be summed, whence  by Borel-Cantelli there is a set $\calB$ of measure~$0$ off of which for large $n$, $\Lambda_n'$ contains a unique singular interval at level~$n$.  Furthermore, this singular interval has distance $<3 \ell_n$ from~$0$ and thus $[3\ell_n,\ell_{n+1}]$ and $[-\ell_{n+1},-3\ell_n]$ are $n$-regular, for parameters $(\theta,E)$ with $\theta\in\tor\setminus\calB$. Lemma~\ref{lem:G I reg} and~\eqref{eq:Poisson} conclude the proof. It is essential here that $\calB$ does not depend on $E$, as evidenced by~\eqref{eq:reflec}. 
\end{proof}
\begin{figure}[ht]
\hspace{-1.6cm}
\centering
\begin{subfigure}{.5\textwidth}
  \centering
  \includegraphics[width=1.2\linewidth]{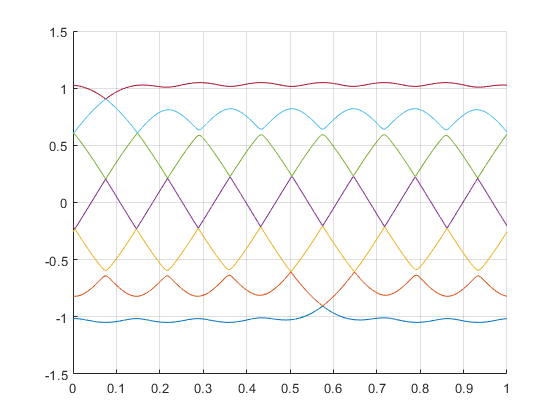}
  \caption{$\eps=0.1$}
  \label{fig:sub1}
\end{subfigure}%
\begin{subfigure}{.5\textwidth}
  \centering
  \includegraphics[width=1.2\linewidth]{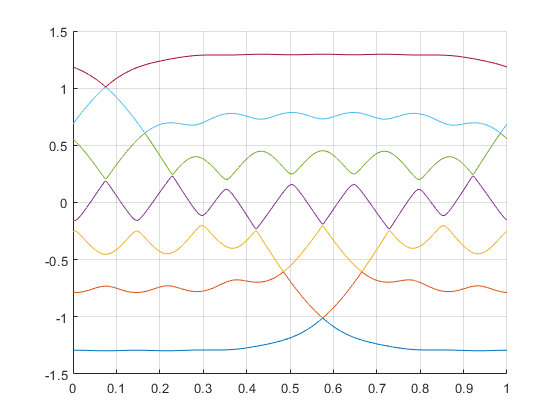}
  \caption{$\eps=0.3$}
  \label{fig:sub2}
\end{subfigure}
\caption{Eigenvalue parameterizations for the cosine potential}
\label{fig:rellich}
\end{figure}
\begin{proof} The remainder of this section is devoted to the proof of Lemma~\ref{lem:step 1}. We begin with  the easier case of  a simple resonance, i.e., $s_0\ge 4(\log\eps)^2$. Fix any $E_*\in \R$ with $[E_*-\delta_0,E_*+\delta_0]\cap [\min V,\max V]\ne\emptyset$, and some $\theta_*\in\tor$. Then $\calS_0=\{c^i_0\}_{i=-\infty}^\infty$, and every $c^i_0$ is contained in a unique level~$1$ interval  $I^i_1$, with $|I^i_1|=\ell_1=\lceil\log(1/\eps)\rceil^2$. These are pairwise disjoint by construction, and they may be regular or singular. We discard the regular ones and only consider those $c^i_1=c^i_0$ for which $I^i_1$ is singular. By the definitions, 
\EQ{\label{eq:VE*}
|V(\theta_*+k\omega)-E_*|& > \delta_0 \;\forall\; k\in I^i_1\setminus\{c^i_0\} \\
|V(\theta_*+c^i_0\omega)-E_*|& \le \delta_0
}
Let $\{E^{i,1}_j(\theta)\}_{j=1}^{\ell_1}$ be the eigenvalue parameterizations (Rellich functions) of $H_{I^i_1}(\theta)$. By min-max,  there exists $k(j,\theta)\in I^i_1$ so that
\EQ{\label{eq:VE close}
|V(\theta+k(j,\theta)\omega)-E^{i,1}_j(\theta) |& \le  2\eps^2 \quad \forall\; 1\le j\le \ell_1
}
Figure~\ref{fig:rellich} exhibits\footnote{The graphs were produced by Yakir Forman at Yale.}  numerically computed Rellich functions for $\ell_1=7$ and the cosine potential. The graphs do not cross, but some of the gaps are  too small to be visible. Subfigures (A) and (B) show how the gaps become wider with increasing $\eps$. The figure demonstrates how we need to jump between different translates $V(\theta+k\omega)$ to approximate any given Rellich graph, hence $k(j,\theta)$, which are not unique near crossing points of $V$ with its own translates by~$\omega$. 
Since $\delta_1\ll\eps^2\ll\eps$,  \eqref{eq:delta1} and \eqref{eq:VE*}, \eqref{eq:VE close} imply that for all $|\theta-\theta_*|\ll \delta_0$
\EQ{\label{eq:1close}
|V(\theta+c^i_0\omega)-E_*| &\ll \delta_0  \\
|V(\theta+k\omega) -E_*| &\gtrsim \delta_0 \quad\forall\; k \in I^i_1\setminus\{c^i_0\}
}
with implied absolute constants (depending only on $v,\omega$).  We now claim that a normalized eigenfunction $\psi(\theta)$ associated with $H_{I^i_1}(\theta)\psi(\theta)=E^{i,1}_j(\theta)\psi(\theta)$ and $k(j,\theta)=c^i_1=c^i_0$ satisfies
\EQ{\label{eq:psi loc}
\| P^\perp \psi(\theta)  \|_{\ell^2(I^i_1)} \les \eps^2 \delta_0^{-1}\ll \eps
}
where $P^\perp$ denotes the orthogonal projection onto all vectors perpendicular to~$\delta_{c^i_0}$ in $\ell^2(I^i_1)$.  Then 
\EQ{\nn
0 &= \psi_{c^i_0}(\theta)P^\perp(H_{I^i_1}(\theta)-E^{i,1}_j(\theta)) \delta_{c^i_0} + P^\perp(H_{I^i_1}(\theta)-E^{i,1}_j(\theta)) P^\perp\psi(\theta) \\
& = \eps^2 \psi_{c^i_0}(\theta) ( \delta_{c^i_0-1}+\delta_{c^i_0+1}) + P^\perp(H_{I^i_1}(\theta)-E^{i,1}_j(\theta)) P^\perp\psi(\theta)
} 
Here and below we use $\delta$ both for the resonance width and in the Dirac sense, without any danger of confusion. 
By \eqref{eq:VE close} and \eqref{eq:1close}, 
\EQ{\label{eq:Gperp}
\big\|[P^\perp(H_{I^i_1}(\theta)-E^{i,1}_j(\theta)) P^\perp]^{-1}\big\|_{\ell^2(I^i_1\setminus\{c^i_0\})}\les\delta_0^{-1}
}
 which implies~\eqref{eq:psi loc} and 
\EQ{\label{eq:VEclose}
|E^{i,1}_j(\theta) -V(\theta+c^i_0\omega)| &= | \langle H_{I^i_1}(\theta)\psi(\theta),\psi(\theta)\rangle -V(\theta+c^i_0\omega)| 
 \les \eps^2\delta_0^{-1}
}
By first order eigenvalue perturbation (Feynman formula), writing $E=E^{i,1}_j$ and $V(\theta)$ for the multiplication operator by the potential, 
\EQ{\label{eq:VE'close}
|E'(\theta) -V'(\theta+c^i_0\omega)| &= | \langle V'(\theta)\psi(\theta),\psi(\theta)\rangle -V'(\theta+c^i_0\omega)| 
 \les \eps^2\delta_0^{-1}
}
and by the second order perturbation formula, with $G^\perp$ being the resolvent on the left-hand side of~\eqref{eq:Gperp}, 
\begin{align}
|E''(\theta) -V''(\theta+c^i_0\omega)| &= | \langle V''(\theta)\psi(\theta),\psi(\theta)\rangle  -V''(\theta+c^i_0\omega) -
2 \langle\psi(\theta), V'(\theta)G^\perp(E(\theta))V'(\theta)\psi(\theta)\rangle| \nn \\ 
&   \les \eps^2\delta_0^{-1} + (\eps^2\delta_0)^{2}\delta_0^{-1} \les \eps^2\delta_0^{-1} \ll \eps \label{eq:VE''close}
\end{align}
The estimates \eqref{eq:VEclose}-\eqref{eq:VE''close} hold for $|\theta-\theta_*|\ll \delta_0$.  We conclude from \eqref{eq:VE'close}, \eqref{eq:VE''close} that 
\EQ{\label{eq:Enondeg}
\min_{|\theta-\theta_*|\ll\delta_0} \big( |E'(\theta)|+|E''(\theta)|\big) \gtrsim 1 \qquad\forall\; \theta_*\in\tor,
}
Recall that $E(\theta)=E^{i,1}_j(\theta)$ depends on $\theta_*$ and $E_*$, where the latter is chosen so that $\calS_0\ne\emptyset$. The constant in~\eqref{eq:Enondeg} is uniform in $\theta_*, E_*$. The reader is invited to  compare~\eqref{eq:Enondeg} to Figure~\ref{fig:rellich}. 
\begin{figure}[ht]
\hspace{-1.6cm}
\centering
\begin{subfigure}{.5\textwidth}
  \centering
  \includegraphics[width=.9\linewidth]{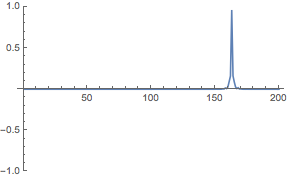}
\end{subfigure}%
\begin{subfigure}{.5\textwidth}
  \centering
  \includegraphics[width=.9\linewidth]{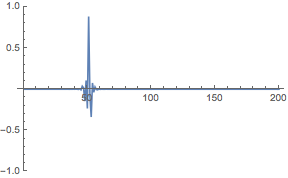}
\end{subfigure}
\begin{subfigure}{.5\textwidth}
  \centering
  \includegraphics[width=.9 \linewidth]{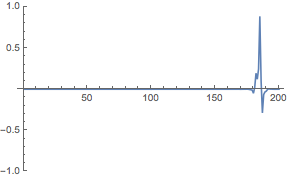}
\end{subfigure}
\caption{Numerically computed eigenfunctions, $\eps=0.3$, $\omega=\sqrt{2}$, $\theta=-17\omega/2$ }
\label{fig:Eigf1}
\end{figure}
Now suppose we have two distinct singular intervals $I^i_1$ and $I^j_1$ relative to $(\theta_*,E_*)$. Then the previous analysis applies to both Rellich functions $E(\theta),\tilde E(\theta)$ defined for $|\theta-\theta_*|\ll\delta_0$ characterized by 
\EQ{\label{eq:eval ch1}
\spec(H_{I^i_1}(\theta))\cap [E_*-\delta_0/2, E_*-\delta_0/2] &= \{E(\theta)\},\\
\spec(H_{I^j_1}(\theta))\cap [E_*-\delta_0/2, E_*-\delta_0/2] &= \{\tilde E(\theta)\}
}
By \eqref{eq:delta1} we have $|E(\theta_*)-E_*|\le \delta_1$, $|\tilde E(\theta_*)-E_*|\le \delta_1$. By \eqref{eq:VE close} we have
\[
|V(\theta_*+c^i_0\omega)-E(\theta_*)|\le 2\eps^2, \quad |V(\theta_*+c^j_0\omega)-\tilde E(\theta_*)|\le 2\eps^2
\]
whence
\[
|V(\theta_*+c^i_0\omega) - V(\theta_*+c^j_0\omega)|\le 4\eps^2+2\delta_1 \le 5\eps^2
\]
We showed in Lemma~\ref{lem:step 0} that this implies 
\EQ{\label{eq:mdelta0}
m(c^i_0,c^j_0)\les \eps\ll\delta_0.
}
 Next, we improve this estimate to to $m(c^i_0,c^j_0)\les\delta_1$.  Suppose \eqref{eq:mdelta0} means $\|(c^i_0-c^j_0)\omega\|\ll \delta_0$.  By~\eqref{eq:conj} one has 
 \EQ{\label{eq:HiHj}
 H_{I^i_1}(\theta)= H_{I^j_1}(\theta+(c^i_0-c^j_0)\omega)
 }
which, combined with \eqref{eq:eval ch1} implies that  $\tilde E(\theta)= E(\theta+(c^j_0-c^i_0)\omega)$ for all $|\theta-\theta_*|\ll\delta_0$. This finally implies that 
\[
|E(\theta_*) - E(\theta_*+(c^j_0-c^i_0)\omega)|\le 2\delta_1
\]
From \eqref{eq:Enondeg} we obtain $\| (c^j_0-c^i_0)\omega\|^2\les \delta_1$. On the other hand, if \eqref{eq:mdelta0} means $\|2\theta_*+(c^i_0+c^j_0)\omega\|\ll \delta_0$, then we have
\EQ{\label{eq:EtilE}
\tilde E(\theta)= E(2\theta_* - \theta - \theta_{**}), \quad \theta_{**} := 2\theta_*+(c^i_0+c^j_0)\omega
}
for all $|\theta-\theta_*|\ll\delta_0$. In terms of Figure~\ref{fig:V} this corresponds to $E(\theta)$ being approximated by~$V$ over $J_1$, whereas $\tilde E(\theta)$ is approximated by~$V$ over $J_2$. Setting $\theta=\theta_*$ in~\eqref{eq:EtilE}, we find that 
\[
|E(\theta_*) - E(\theta_* - \theta_{**})|\le 2\delta_1
\]
which implies $\| \theta_{**}\|^2\les\delta_1$. We have thus proved Lemma~\ref{lem:step 1} for simple resonances. 
\begin{figure}[ht]
\hspace{-1.6cm}
\centering
\begin{subfigure}{.5\textwidth}
  \centering
  \includegraphics[width=.9\linewidth]{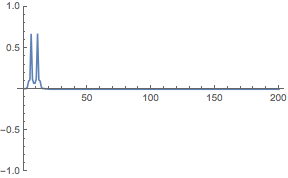}
\end{subfigure}%
\begin{subfigure}{.5\textwidth}
  \centering
  \includegraphics[width=.9\linewidth]{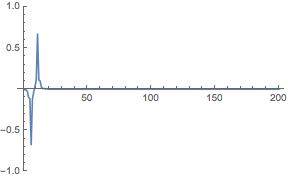}
\end{subfigure}
\begin{subfigure}{.5\textwidth}
  \centering
  \includegraphics[width=.9\linewidth]{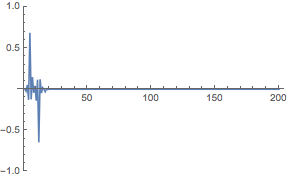}
\end{subfigure}
\caption{Numerically computed eigenfunctions, $\eps=0.3$, $\omega=\sqrt{2}$, $\theta=-17\omega/2$ }
\label{fig:Eigf2}
\end{figure}
Figures~\ref{fig:Eigf1} and~\ref{fig:Eigf2} depict eigenfunctions on finite volume  $200$ computed with \textsc{Mathematica} (the code is included in the appendix for the reader to experiment for themselves, for example by changing $\theta$ or trying rational~$\omega$). The choice of $\theta$ was made as a crossing point of $V$ with one of its translates by a multiple of~$\omega$, since double resonances occur near those points. The first eigenfunction shown in the upper left of Figures~\ref{fig:Eigf1} shows the case of a simple resonance, whereas the second and third are more complicated -- they exhibit a main peak with smaller ones due to resonances at later stages of the induction. On the other hand, Figure~\ref{fig:Eigf2} exhibits double resonances quite clearly (with the bottom eigenfunction exhibiting a more complicated structure). The reader should note the distinct distribution of the $\ell^2$-mass which is quite apparent on the $y$-axes of these figures. 

\medskip

We now prove Lemma~\ref{lem:step 1} for double resonances. Let $I^i_1$ be singular (as the red interval on the right-hand side of Figure~\ref{fig:doubleres}), centered at $c^i_1=\frac12(c^{2i}_0+c^{2i+1}_0)\in\frac12\Z$.  As a side remark, suppose that $c^i_1\in\frac12+\Z$. Then all $c^j_1\in \frac12+\Z$ due to $c^{2j+1}_0-c^{2j}_0=\const.$ for all~$j$ (since we passed to $\bar\calS_0$). At the next levels $n=2,3,\ldots,N$ we will encounter only simple resonances, and so all $c^k_n\in\frac12+\Z$ for all these~$n$. If we then encounter a double resonance at $N+1$, it implies that $c^k_{N+1}\in\frac12\Z$, and the patter repeats itself. Continuing with the main argument, one then has
\EQ{\nn 
& |V(\theta_*+c^{2i}_0\omega)-E_*|\les \delta_0,\quad |V(\theta_*+c^{2i+1}_0\omega)-E_*|\les \delta_0 \\
& |V(\theta_*+ k \omega)-E_*| \gg \delta_0^{\frac12} \quad \forall\; k\in I^i_1\setminus\{ c^{2i}_0, c^{2i+1}_0\} 
}
where the second line follows from the Diophantine condition since $I^i_1=\lceil \log(1/\eps)\rceil^4$ (one can choose a larger lower bound such as $\delta_0^{a}$ for any fixed $0<a\le\frac12$ at the expense of making $\eps$ smaller). By \eqref{eq:VE close}, 
\EQ{\label{eq:two E}
\spec(H_{I^i_1}(\theta))\cap [E_*-\delta_0^{\frac12}, E_*+\delta_0^{\frac12}] =\{ E(\theta), \tilde E(\theta)\} \quad\forall\; |\theta-\theta_*|\les \delta_0^{\frac12}
}
with $E>\tilde E$. By the same type of argument as in the simple resonant case, cf.~\eqref{eq:psi loc}, \eqref{eq:Gperp}, we see that the normalized eigenfunctions of $H_{I^i_1}(\theta)$ associated with $E$, resp.~$\tilde E$, are
\EQ{\label{eq:psitilpsi} 
\psi & = A \delta_{c^{2i}_0} + B \delta_{c^{2i+1}_0} + O(\eps^2\delta_0^{-\frac12} )\\
\tilde \psi &= -B \delta_{c^{2i}_0} + A \delta_{c^{2i+1}_0} + O(\eps^2\delta_0^{-\frac12})
}
uniformly on $ |\theta-\theta_*|\le \delta_0^{\frac12}$ with $A^2+B^2=1$. In place of~\eqref{eq:Gperp} we have 
\EQ{\label{eq:Gperp*}
\big\|[P^\perp(H_{I^i_1}(\theta)-E_*) P^\perp]^{-1}\big\|_{\ell^2(I^i_1\setminus\{c^{2i}_0,c^{2i+1}_0\})}\les\delta_0^{-\frac12}
}
where $P^\perp$ is the orthogonal projection perpendicular to $\lspan(\delta_{c^{2i}_0}, \delta_{c^{2i+1}_0})$ in $\ell^2(I^i_1)$.   The eigenvalues at level~$0$ associated with the points $c^{2i}_0$, $c^{2i+1}_0$ are,
resp., $E^{2i}_0(\theta):=V(\theta+c^{2i}_0\omega)$, $E^{2i+1}_0(\theta):=V(\theta+c^{2i+1}_0\omega)$.  This terminology is justified by the relations
\[
\langle H(\theta)\delta_{c^{2i}_0},\delta_{c^{2i}_0}\rangle = E^{2i}_0(\theta), \quad \langle H(\theta)\delta_{c^{2i+1}_0},\delta_{c^{2i+1}_0}\rangle = E^{2i+1}_0(\theta)
\]
By Lemma~\ref{lem:step 0}, $\|\theta_{**}\|\les \delta_0^{\frac12}$ with $\theta_{**}=2\theta_* + (c^{2i}_0+c^{2i+1}_0)\omega=2(\theta_*+c^i_1\omega)$.  It follows that either (a) $\|\theta_*+c^i_1\omega\|\les\delta_0^{\frac12}$ or  (b) $\|\theta_*+\frac12+c^i_1\omega\|\les\delta_0^{\frac12}$. These relations show that the unique solution  of $E^{2i}_0(\theta)=E^{2i+1}_0(\theta)$ on 
$|\theta-\theta_*|\les \delta_0^{\frac12}$ is either (a)  $\theta_s=-c^i_1\omega$  or (b) $\theta_s=\frac12-c^i_1\omega$ (henceforth, $\theta_s$ will mean either of these whichever applies). These identities are a restatement of $V$ being symmetric both (a) around $0$ and (b) around~$\frac12$.   Furthermore, one has 
\[
E_0^{2i}(\theta+\theta_s) = E_0^{2i+1}(-\theta+\theta_s)
\]
whence $\partial_\theta E_0^{2i}(\theta_s)= - \partial_\theta E_0^{2i+1}(\theta_s)$. 
\begin{figure}[ht]
\tikzset{every picture/.style={line width=0.75pt}} 

\begin{tikzpicture}[x=0.58pt,y=0.48pt,yscale=-1,xscale=1]

\draw    (2.21,439.06) -- (659.21,441.06) ;
\draw    (21.21,-1.94) -- (19.21,460.06) ;
\draw [color={rgb, 255:red, 208; green, 2; blue, 27 }  ,draw opacity=1 ][line width=2.25]    (160.29,440.14) -- (479,441.14) ;
\draw [line width=1.5]    (161.29,434.37) -- (161.29,446.37) ;
\draw [line width=1.5]    (325.29,435.37) -- (325.29,447.37) ;
\draw [line width=1.5]    (359.29,435.37) -- (359.29,447.37) ;
\draw [line width=1.5]    (479.29,435.37) -- (479.29,447.37) ;
\draw  [dash pattern={on 0.84pt off 2.51pt}]  (161.29,0.37) -- (160.29,440.14) ;
\draw  [dash pattern={on 0.84pt off 2.51pt}]  (480,1.36) -- (479,441.14) ;
\draw [color={rgb, 255:red, 65; green, 117; blue, 5 }  ,draw opacity=1 ][line width=1.5]    (160,9.36) -- (479.29,330.98) ;
\draw [color={rgb, 255:red, 65; green, 117; blue, 5 }  ,draw opacity=1 ][line width=1.5]    (479.29,88.37) -- (161.29,417.37) ;
\draw [color={rgb, 255:red, 208; green, 2; blue, 27 }  ,draw opacity=1 ][line width=1.5]    (160,2.36) -- (352.29,194.37) ;
\draw [color={rgb, 255:red, 208; green, 2; blue, 27 }  ,draw opacity=1 ][line width=1.5]    (479.29,79.98) -- (369.28,194.33) ;
\draw [color={rgb, 255:red, 208; green, 2; blue, 27 }  ,draw opacity=1 ][line width=1.5]    (349.29,232.37) -- (159,429.36) ;
\draw [color={rgb, 255:red, 208; green, 2; blue, 27 }  ,draw opacity=1 ][line width=1.5]    (366.28,232.25) -- (477.57,345.25) ;
\draw    (135,335.36) .. controls (182.81,317.55) and (179.37,353.63) .. (233.35,335.92) ;
\draw [shift={(235,335.36)}, rotate = 521.1700000000001] [color={rgb, 255:red, 0; green, 0; blue, 0 }  ][line width=0.75]    (10.93,-3.29) .. controls (6.95,-1.4) and (3.31,-0.3) .. (0,0) .. controls (3.31,0.3) and (6.95,1.4) .. (10.93,3.29)   ;
\draw    (79,33.36) .. controls (118.6,3.66) and (138.6,62.17) .. (177.81,34.24) ;
\draw [shift={(179,33.36)}, rotate = 503.13] [color={rgb, 255:red, 0; green, 0; blue, 0 }  ][line width=0.75]    (10.93,-3.29) .. controls (6.95,-1.4) and (3.31,-0.3) .. (0,0) .. controls (3.31,0.3) and (6.95,1.4) .. (10.93,3.29)   ;
\draw  [color={rgb, 255:red, 208; green, 2; blue, 27 }  ,draw opacity=1 ][line width=1.5]  (369.28,194.33) .. controls (363.63,201.87) and (357.97,201.89) .. (352.29,194.37) ;
\draw  [color={rgb, 255:red, 208; green, 2; blue, 27 }  ,draw opacity=1 ][line width=1.5]  (349.29,232.37) .. controls (354.9,224.8) and (360.57,224.76) .. (366.28,232.25) ;
\draw    (349.29,77.98) .. controls (306.94,82.9) and (315.03,102.38) .. (277.06,106.79) ;
\draw [shift={(275.29,106.98)}, rotate = 354.28999999999996] [color={rgb, 255:red, 0; green, 0; blue, 0 }  ][line width=0.75]    (10.93,-3.29) .. controls (6.95,-1.4) and (3.31,-0.3) .. (0,0) .. controls (3.31,0.3) and (6.95,1.4) .. (10.93,3.29)   ;
\draw    (293.29,381.98) .. controls (250.94,386.9) and (254.18,372.42) .. (216.06,375.81) ;
\draw [shift={(214.29,375.98)}, rotate = 354.28999999999996] [color={rgb, 255:red, 0; green, 0; blue, 0 }  ][line width=0.75]    (10.93,-3.29) .. controls (6.95,-1.4) and (3.31,-0.3) .. (0,0) .. controls (3.31,0.3) and (6.95,1.4) .. (10.93,3.29)   ;
\draw [color={rgb, 255:red, 144; green, 19; blue, 254 }  ,draw opacity=1 ][line width=1.5]  [dash pattern={on 5.63pt off 4.5pt}]  (160.29,239.98) -- (481.29,239.98) ;
\draw [color={rgb, 255:red, 144; green, 19; blue, 254 }  ,draw opacity=1 ][line width=1.5]  [dash pattern={on 5.63pt off 4.5pt}]  (160.29,233.98) -- (481.29,233.98) ;

\draw (648,456.06) node    {$\mathbb{T}$};
\draw (326,457.36) node    {$\theta _{*}$};
\draw (362,457.36) node    {$\theta _{s}$};
\draw (142,462.36) node    {$\theta _{*} -\delta _{0}^{\frac12}$};
\draw (503,461.36) node    {$\theta _{*} +\delta _{0}^{\frac12}$};
\draw (100,346.36) node    {$V( \theta+c^{2i}_0\omega )$};
\draw (75,48.36) node    {$V( \theta +c^{2i+1}_0\omega )$};
\draw (372,77.98) node    {$E( \theta )$};
\draw (312,380.98) node    {$\tilde E( \theta )$};
\draw (132,248.98) node    {$E_{*} -\delta _{1}$};
\draw (132,222.98) node    {$E_{*} +\delta _{1}$};

\end{tikzpicture}
\caption{Crossing graphs and double resonance}\label{fig:crossing}
\end{figure}
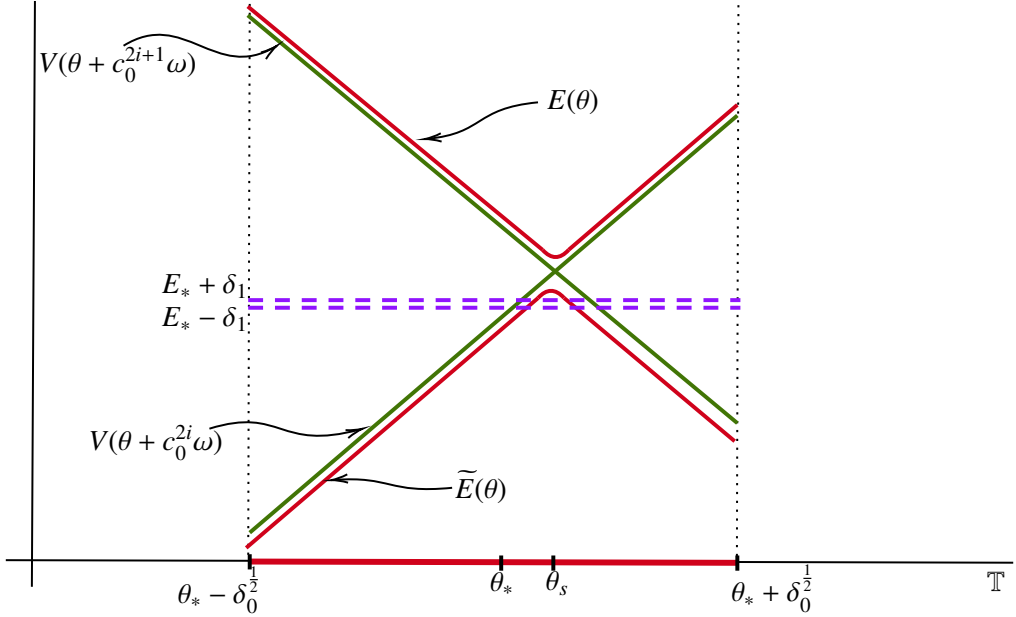
The configuration associated with a double resonance is shown in Figure~\ref{fig:crossing}. Not only do the segments of the $V$-graphs (i.e., $E^{2i}_0$ and $E^{2i+1}_0$) intersect at $\theta_s$, but $E,\tilde E$ have their critical point at $\theta_s$ within the interval $|\theta-\theta_*|\les \delta_0^{\frac12}$. Indeed, 
\[
H_{I^i_1}(\theta+\theta_s) = U H_{I^i_1}(-\theta+\theta_s) U
\]
where $U$ is the reflection on $\Z$ about $c^i_1\in\frac12\Z$. In particular, the eigenvalues are the same. In fact, using~\eqref{eq:two E} one concludes that
\EQ{\label{eq:EtilE symm}
E(\theta+\theta_s) = E(-\theta+\theta_s), \quad \tilde E(\theta+\theta_s) = \tilde E(-\theta+\theta_s)\quad \forall\; |\theta-\theta_*|\les\delta_0^{\frac12}
} 
whence $E'(\theta_s)=\tilde E(\theta_s)=0$. 

Next,  we establish the lower bound 
\EQ{\label{eq:EtilE sep}
E(\theta) - \tilde E(\theta) > (c\eps)^{5\ell_1^{1/2}} \gg \delta_1=\eps^{\ell_1^{2/3}} = \eps^{|\log \eps|^{8/3}}
}
which follows immediately from this separation lemma, see \cite[Lemma~4.1]{FSW}. This spectral gap is much larger than the resonance width~$\delta_1$.

\begin{lem}
\label{lem:sep}
Let $H_\Lambda\psi_j=E_j\psi_j$, $j=1,2$ with nontrivial $\psi_j$. If $\|\psi_j\|_{\ell^2(\Lambda_0)}\ge \frac12\|\psi_j\|_{\ell^2(\Lambda)}$ for $j=1,2$ with $\Lambda_0\subset\Lambda$ and $|\Lambda_0|\ge2$, then 
\EQ{\label{eq:E sep}
|E_1 - E_2|\ge    (c_1\eps)^{|\Lambda_0|}(\eps^{-2} + |\Lambda_0|)^{-1}
}
 with a constant $c_1=c_1(v)>0$. 
\end{lem}
\begin{proof}
Let $\Lambda_0=[n_0-\ell_0,n_0+\ell_0]$ or $\Lambda_0=[n_0-\ell_0,n_0+\ell_0-1]$. By assumption, $\ell_0\ge1$. Normalize $\psi_j(n_0-1)^2+\psi_j(n_0)^2=1$ for $j=1,2$. Setting $\tilde \psi_1(n)=\psi_1(n)$ if $n\in\Lambda$, $n\ge n_0$ and $\tilde \psi_1(n)=-\psi_1(n)$ if $n\in\Lambda$, $n<n_0$ one obtains from considering $\langle H_\Lambda\tilde \psi_1 ,\psi_2\rangle=
 \langle \tilde \psi_1 , H_\Lambda\psi_2\rangle$ that
\EQ{
2\eps^2 |\vec v_1\wedge \vec v_2|&= 2\eps^2|\psi_1(n_0)\psi_2(n_0-1) - \psi_2(n_0)\psi_1(n_0-1)| = 	 \\
&\le 2|E_1-E_2|(\| \psi_1\|_2\| \psi_2\|_2)^{\frac12}  (\| \psi_1\|_{\ell^2(\Lambda_0)} \| \psi_2\|_{\ell^2(\Lambda_0)})^{\frac12}  \\
&\le 2 (C\eps^{-2})^{\ell_0}\, |E_1-E_2|(\| \psi_1\|_2\| \psi_2\|_2)^{\frac12}
}
where $\vec v_j=\binom{\psi_j(n_0)}{\psi_j(n_0-1)}$ and $C=C(v)$. The final estimate is obtained from the transfer matrix representation of the eigenfunctions, viz.\ for $n\ge n_0+1$
\[
\binom{\psi_j(n)}{\psi_j(n-1)} = \prod_{k=n_0}^{n-1} \left[ \begin{matrix} \eps^{-2}(v_k-E_j) & -1 \\ 1 & 0\end{matrix} \right] \vec v_j
\]
and for $n\le n_0-1$
\[
\binom{\psi_j(n-1)}{\psi_j(n)} = \prod_{k=n_0-1}^{n} \left[ \begin{matrix} \eps^{-2}(v_k-E_j) & -1 \\ 1 & 0\end{matrix} \right]^{-1} \vec v_j
\]
On the one hand,  with $B=C\eps^{-2}$,  and using that $\|\vec v_1-\vec v_2\|_2=2|\vec v_1\wedge \vec v_2|$, 
\EQ{\nn
\| \psi_1 - \psi_2 \|_{\ell^2(\Lambda_0)} & \le 2B^{\ell_0}|\vec v_1\wedge \vec v_2|   + \ell_0 B^{\ell_0} |E_1-E_2| \\
&\le 2\eps^{-2} B^{\ell_0} \, |E_1-E_2|(\| \psi_1\|_2\| \psi_2\|_2)^{\frac12}  + \ell_0 B^{\ell_0} |E_1-E_2| 
}
On the other hand,  with $a_j:=\|\psi_j\|_{\ell^2(\Lambda)}$, 
\EQ{\nn
\sqrt{a_1^2+a_2^2} & \le \frac12(a_1+a_2) + \| \psi_1 - \psi_2 \|_{\ell^2(\Lambda_0)} \\
&\le  \frac12(a_1+a_2) + B^{\ell_0}(B\sqrt{a_1a_2} + \ell_0)|E_1-E_2| \\
&\le \frac12(a_1+a_2) + B^{\ell_0}(B + \ell_0)|E_1-E_2|\, \sqrt{a_1a_2}
}
If $|E_1-E_2|< \frac14 B^{-\ell_0}(B + \ell_0)^{-1}$, then 
$
\sqrt{a_1^2+a_2^2} < \frac{5}{8} (a_1+a_2)
$
which is impossible. Adjusting the constants one obtains~\eqref{eq:E sep}.  
\end{proof}

The critical points of $V$ are $\theta=0$ and $\theta=\frac12$. We claim that $\min(\|\theta_*\|, \| \theta_*-\frac12\|)\ge K\delta_0^{\frac12}$ where $K$ is any  large constant, to
be fixed below (as always, provided $\eps$ is small enough). This is immediate from the Diophantine condition due to $s_0\le 4(\log\eps)^2$, cf.~Figure~\ref{fig:V}. In particular,
$|V'(\theta)|\gg \delta_0^{\frac12}$ on the interval $|\theta-\theta_*|\les \delta_0^{\frac12}$. By first order eigenvalue perturbation and~\eqref{eq:psitilpsi}, uniformly on this interval
\[
\partial_\theta E(\theta) = \langle \psi(\theta),V'(\theta)\psi(\theta)\rangle = A^2(\theta) \partial_\theta E^{2i}_0(\theta) +   B^2(\theta) \partial_\theta E^{2i+1}_0(\theta)  + O(\eps^2\delta_0^{-\frac12})
\]
where by the preceding $|\partial_\theta E^{2i}_0(\theta)|\gg \delta_0^{\frac12}$ and $|\partial_\theta E^{2i+1}_0(\theta)|\gg\delta_0^{\frac12}$. 
Setting $\theta=\theta_s$ it follows that $A^2(\theta_s)-B^2(\theta_s)=O(\eps^2\delta_0^{-1})=O(\eps)$. Due to $A^2+B^2=1$, $|A(\theta_s)|^2 = 1/2 - O(\eps)$ and  $|B(\theta_s)|^2=1/2-O(\eps)$.  In fact, the same argument shows that $|A(\theta)|\simeq |B(\theta)|\simeq 1$ for all  $|\theta-\theta_*|\les \delta_0^{\frac12}$ with $ |\partial_\theta E(\theta)|\les\delta_0^{\frac12}$. 

Using this property we can now establish closeness of all eigenvalues. In fact, 
$H_\Lambda(\theta) \psi(\theta)  = E(\theta) \psi(\theta) $ and $H_\Lambda(\theta) \tilde\psi(\theta)  = \tilde E(\theta) \tilde \psi(\theta)$ in combination with~\eqref{eq:psitilpsi} imply that
\EQ{\label{eq:all E close}
E^{2i}_0(\theta)-E(\theta) = O(\eps^2\delta_0^{-\frac12}), \quad E^{2i+1}_0(\theta)-E(\theta) = O(\eps^2\delta_0^{-\frac12})
}
and the same for $\tilde E$.  In particular, 
\EQ{\label{eq:EtilE close}
|E(\theta)-\tilde E(\theta)|\les \eps^2\delta_0^{-\frac12}
}
for all $|\theta-\theta_*|\les \delta_0^{\frac12}$ with $ |\partial_\theta E(\theta)|\les\delta_0^{\frac12}$. 
The final step in our analysis is to establish a lower bound on $|\partial_\theta^2 E(\theta)|$ and $|\partial_\theta^2 \tilde E(\theta)|$ for those~$\theta$. This hinges on the second order perturbation formulas (suppressing $\theta$ as argument)
\EQ{\nn
\partial_\theta^2 E &= \langle \psi, V''\psi\rangle - 2 \langle V'\psi, G(E)^\perp V'\psi\rangle \\
\partial_\theta^2 \tilde E &= \langle \tilde \psi, V''\tilde \psi\rangle - 2 \langle V'\tilde \psi, G(\tilde E)^\perp V'\tilde \psi\rangle
}
on $\ell^2(I^i_1)$ with $G(E)^\perp= [P^\perp_\psi (H_{I^i_1}-E) P^\perp_\psi]^{-1}$ in $P^\perp_\psi\ell^2(I^i_1)$ and $P^\perp_\psi$ being the orthogonal projection onto the complement of $\psi$ in $\ell^2(I^i_1)$. Analogous comments apply $G(\tilde E)^\perp$ which is the resolvent orthogonal to $\tilde\psi$. We now write 
\[
\langle V'\psi, G(E)^\perp V'\psi\rangle = \frac{\langle \psi, V'\tilde\psi\rangle^2}{\tilde E-E} + \langle V'\psi, G(E)^{\perp\perp} V'\psi\rangle
\]
where $G(E)^{\perp\perp} = P^\perp_{\tilde\psi}[P^\perp_\psi (H_{I^i_1}-E) P^\perp_\psi]^{-1}P^\perp_{\tilde\psi}$. By \eqref{eq:two E}, $\|G(E)^{\perp\perp}\|\les \delta_0^{-\frac12}$. 
On the other hand, by \eqref{eq:psitilpsi} 
\[
\langle \psi, V'\tilde\psi\rangle = -2AB\partial_\theta E^{2i}_0(\theta_s) + O(\delta_0^{\frac12})
\]
whence from $|AB|\simeq1$, 
\[
|\langle \psi, V'\tilde\psi\rangle |\gg \delta_0^{\frac12}
\]
Combining this with \eqref{eq:EtilE close} we obtain 
\[
|\langle V'\psi, G(E)^\perp V'\psi\rangle | \gg  \eps^{-2} \delta_0^{\frac32} - O( \delta_0^{-\frac12})  \gg \eps^{-2} \delta_0^{\frac32} \gg 1
\]
Since $| \langle \psi, V''\psi\rangle|\les 1$, it follows that 
\EQ{\label{eq:E'' big}
|\partial_\theta^2 E(\theta)|\gg  \eps^{-2} \delta_0^{\frac32}, \quad |\partial_\theta^2 \tilde E(\theta)|\gg  \eps^{-2} \delta_0^{\frac32}
}
for all $|\theta-\theta_*|\les \delta_0^{\frac12}$ with $ |\partial_\theta E(\theta)|\les\delta_0^{\frac12}$. The exact same analysis applies to $\tilde E$. 
To summarize, these are the main points concerning double resonances at level~$1$. 
\begin{itemize}
\item $s_0\le (\log\eps)^2$ and the level-$0$ singular sites are $\calS_0=\{c^j_0\}_{j=-\infty}^\infty$ with $c^{2i+1}_0-c^{2i}_0=s_0$ for all~$i$. We have $|V(\theta+k\omega)-E_*|\les \delta_0^{\frac12}$ for $k\in \{c^{2i}_0,c^{2i+1}_0\}$, $|V(\theta+k\omega)-E_*|\gg \delta_0^{\frac12}$ for all $k\in I^i_1\setminus\{c^{2i}_0,c^{2i+1}_0\}$, both for all $|\theta-\theta_*|\les \delta_0^{\frac12}$. Here $I^i_1\subset\Z$, $|I^i_1|\simeq (\log\eps)^4$, centered at $c^i_1\in\frac12\Z$ and $\dist(\spec(H_{I^i_1}(\theta_*)),E_*)\les \delta_1=\eps^{\ell_1^{2/3}}$. 
\item $\spec(H_{I^i_1}(\theta))\cap [E_*-\delta_0^{\frac12}, E_*+\delta_0^{\frac12}]=\{E(\theta), \tilde E(\theta)\}$ (with $E>\tilde E$) for these $\theta$, with all other eigenvalues being separated from $E_*$ by~$\gg\delta_0^{\frac12}$. From the level-$0$ estimate $m(c^{2i}_0, c^{2i+1}_0)\les \delta_0$, either $\theta_s=-c^i_1\omega$ or $\theta_s=\frac12-c^i_1\omega$ satisfy $\|\theta_s-\theta_*\|\les \delta_0^{\frac12}$ and the unique critical points of  $E,\tilde E$ in this interval are at $\theta_s$. There is a spectral gap of size $E(\theta)-\tilde E(\theta)>\eps^{5\ell_1^{\frac12}}$. 
\item For every $|\theta-\theta_*|\les \delta_0^{\frac12}$ one has either {\em both} $|\partial_\theta E(\theta)|\gg \delta_0^{\frac12}$ {\em and} $|\partial_\theta \tilde E(\theta)|\gg \delta_0^{\frac12}$ (large slopes), or both $|\partial_\theta E(\theta)|\les \delta_0^{\frac12}$ and $|\partial_\theta \tilde E(\theta)|\les \delta_0^{\frac12}$ (small slopes). This follows from $|\partial_\theta E_0^{2i}(\theta_s) | \gg\delta_0^{\frac12}$,  and the first order eigenvalue perturbation formulas 
\EQ{\nn
\partial E(\theta) &= (A^2(\theta)-B^2(\theta))\partial_\theta E_0^{2i}(\theta_s) + O(\delta_0^{\frac12}) \\
\partial \tilde E(\theta) &= (-A^2(\theta)+B^2(\theta))\partial_\theta E_0^{2i}(\theta_s) + O(\delta_0^{\frac12}) 
}
for all $|\theta-\theta_*|\les\delta_0^{\frac12}$, cf.~\eqref{eq:psitilpsi}. 
\item If the small slope alternative occurs, then   $|A(\theta)|\simeq |B(\theta)|\simeq 1$ and \eqref{eq:all E close} holds for both $E$ and $\tilde E$. In particular, the spectral gap is small as in~\eqref{eq:EtilE close}, and the second derivatives are large and $\gg\delta_0^{-\frac12}$, see~\eqref{eq:E'' big}. This means that the intervals of small slopes around the critical points at $\theta_s$ are of size $\ll\delta_0$. 
\item Figure~\ref{fig:crossing} depicts the situation for a double resonance: $E$ reaches its minimum, resp.\ $\tilde E$ its maximum, at $\theta_s$. The spectral gap is the smallest at this point and the quantitative estimates above hold. In particular, this gap is much larger than $\delta_1$, whence exactly one of $E$ or $\tilde E$ achieve the resonance condition~\eqref{eq:delta1} 
at~$\theta_*$. 
\end{itemize}
To conclude the proof of Lemma~\ref{lem:step 1} we apply this description to two such level~$1$ intervals, say $I^i_1$ and $I^j_1$. 
Because of the double resonance assumption, we have
\[
\| 1/2 + c^i_1\omega\|+ \| 1/2 + c^j_1\omega\| \les\delta_0^{\frac12}
\]
which implies that $m(c^i_1,c^j_1)=\| (c^i_1-c^j_1)\omega\|\les\delta_0^{\frac12}$. 
As in the single resonance case, cf.~\eqref{eq:HiHj}, for all $|\theta-\theta_*|\les\delta_0^{\frac12}$
\[
H_{I^i_1}(\theta) = H_{I^j_1}(\theta+ (c^i_1-c^j_1)\omega ) \text{\ and\ } E^i_1(\theta) = E^j_1(\theta+ (c^i_1-c^j_1)\omega ), \; \tilde E^i_1(\theta) = 
\tilde E^j_1(\theta+ (c^i_1-c^j_1)\omega )
\]
Finally, by \eqref{eq:delta1}, either 
\[
|E^j_1(\theta_*)-E^j_1(\theta_*+ (c^i_1-c^j_1)\omega )|\les\delta_1 \text{\ or \ } |\tilde E^j_1(\theta_*)-\tilde E^j_1(\theta_*+ (c^i_1-c^j_1)\omega )|\les\delta_1
\]
By the bounds derived above on the first and second derivatives on $E^i_1$ etc.\ and elementary  calculus, we finally conclude that $ m(c^i_1,c^j_1)\les\delta_1^{\frac12}$.
Indeed, in the large slopes case, $m(c^i_1,c^j_1)=\|(c^i_1-c^j_1)\omega\|\les \delta_0^{-\frac12}\delta_1$, whereas in the small slopes case, $m(c^i_1,c^j_1)\les \delta_0^{\frac14}\delta_1^{\frac12}$. This is slightly better than 
what  Lemma~\ref{lem:step 1} claims, and we are done.  
\end{proof}

The full induction needed to establish \eqref{eq:mn} follows these exact same lines with no essentially new ideas needed. The reader can either convince themselves of this fact, or consult~\cite{FSW}. Note, however, that Lemma~5.2 in loc.\ cit.\  erroneously sets $\theta_s=-c^i_m\alpha$ forgetting the case (b) above in which $1/2$ has to be added. This is a systematic oversight in Section~5 in that paper which is rooted in a false identity at the conclusion of the proof of Lemma~5.3: $2\|\theta\|=\|2\theta\|$ for the metric on~$\tor$. 

It seems very difficult to approach quasi-periodic localization in more general settings by relying on  eigenvalue parametrization, as we did in this section. 

\subsection{The work of Forman and VandenBoom: dropping evenness of $V$}

We will now discuss the highly challenging task of implementing some version of the Fr\"ohlich-Spencer-Wittwer proof strategy   without the symmetry assumption on the potential.  This has recently been accomplished by Forman and VandenBoom, see~\cite{FV}. 
We now sketch\footnote{The remainder of this subsection was written by Forman and VandenBoom.} the proof of their result.

\begin{thm}
Let $V \in C^2(\mathbb T)$ have exactly two nondegenerate critical points. Define 
\begin{equation}
H_\varepsilon(\theta) = \varepsilon^2\Delta_{\mathbb Z} + V_\theta, \quad V_\theta(n) = V(T_\omega^n\theta)\; \forall\, n\in\mathbb Z
\end{equation}
where $\omega \in \mathbb T$ is Diophantine, viz.\ $\|n\omega\| \geq c_0n^{-2}$ for all $n\geq1$ with some $c_0>0$. There exists $\varepsilon_0(c_0,V)$ such that for all $0<\varepsilon\leq\varepsilon_0$ the operators $H_{\theta,\varepsilon}$ exhibit Anderson localization for a.e.\ $\theta\in\mathbb T$.
\end{thm}
This is precisely the result of Fr\"ohlich, Spencer, and Wittwer without the evenness assumption, and we will make frequent references to the proof of that result. See the previous section. 

As in the symmetric case, we can define singular sites $\mathcal S_0$ relative to $\theta_*,E_*$. However, the $m(k,\ell)$ function is no longer useful, as $\|2\theta_* + (k+\ell)\omega\|$ is no longer small if $T^k_\omega\theta_*$ and $T^\ell_\omega\theta_*$ fall into different connected components of $V^{-1}([E_*-\delta_0,E_*+\delta_0])$. Without symmetry, no such function $m$ can be defined to be  independent of $E_*$.

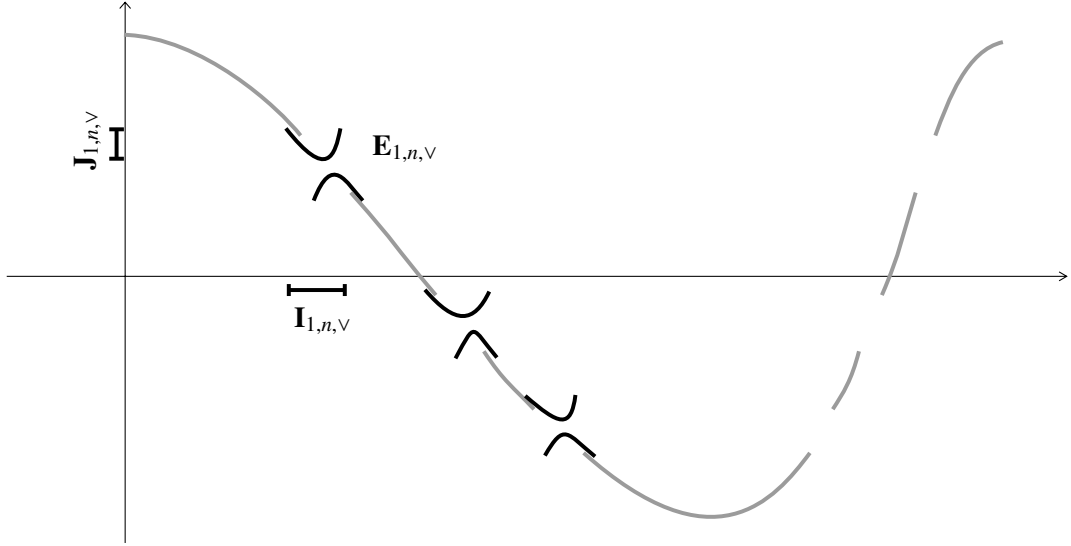
\begin{figure}[h]
	\caption{A cartoon output of the first inductive step: a collection $\mathcal{E}_1$ of Rellich functions of various Dirichlet restrictions of $H$, whose domains (and their relevant translates) cover the circle $\mathbb{T}$.  The curves in black come from double resonances, and the curves in gray are simple resonant.}
\begin{center}
	
	\begin{tikzpicture}[x=0.75pt,y=0.75pt,yscale=-0.45,xscale=0.55]
		
		\draw  (-21,307.94) -- (951,307.94)(87.6,0.5) -- (87.6,607) (944,302.94) -- (951,307.94) -- (944,312.94) (82.6,7.5) -- (87.6,0.5) -- (92.6,7.5)  ;
		\draw [line width=1.5]    (79.74,143.04) -- (79.74,176) ;
		\draw [shift={(79.74,176)}, rotate = 270] [color={rgb, 255:red, 0; green, 0; blue, 0 }  ][line width=1.5]    (0,6.71) -- (0,-6.71)   ;
		\draw [shift={(79.74,143.04)}, rotate = 270] [color={rgb, 255:red, 0; green, 0; blue, 0 }  ][line width=1.5]    (0,6.71) -- (0,-6.71)   ;
		\draw [line width=1.5]    (237.66,323.31) -- (289,323.31) ;
		\draw [shift={(289,323.31)}, rotate = 180] [color={rgb, 255:red, 0; green, 0; blue, 0 }  ][line width=1.5]    (0,6.71) -- (0,-6.71)   ;
		\draw [shift={(237.66,323.31)}, rotate = 180] [color={rgb, 255:red, 0; green, 0; blue, 0 }  ][line width=1.5]    (0,6.71) -- (0,-6.71)   ;
		\draw [color={rgb, 255:red, 0; green, 0; blue, 0 }  ,draw opacity=1 ][line width=1.5]    (235,142) .. controls (273,201) and (280,173) .. (285,142) ;
		\draw [color={rgb, 255:red, 0; green, 0; blue, 0 }  ,draw opacity=1 ][line width=1.5]    (260.5,223) .. controls (277.5,169) and (290.5,203) .. (305.5,223) ;
		\draw [color={rgb, 255:red, 155; green, 155; blue, 155 }  ,draw opacity=1 ][line width=1.5]    (294.5,214) .. controls (343,281) and (327.5,264) .. (372.5,329) ;
		\draw [color={rgb, 255:red, 155; green, 155; blue, 155 }  ,draw opacity=1 ][line width=1.5]    (86.5,37) .. controls (158.45,39.58) and (228.5,119) .. (248.5,150) ;
		\draw [color={rgb, 255:red, 0; green, 0; blue, 0 }  ,draw opacity=1 ][line width=1.5]    (362.5,323) .. controls (390.5,363) and (407.5,361) .. (421.5,325) ;
		\draw [color={rgb, 255:red, 0; green, 0; blue, 0 }  ,draw opacity=1 ][line width=1.5]    (390.5,400) .. controls (409.5,357) and (404.5,364) .. (428.5,399) ;
		\draw [color={rgb, 255:red, 155; green, 155; blue, 155 }  ,draw opacity=1 ][line width=1.5]    (416.5,392) .. controls (435.5,427) and (442.5,431) .. (462.5,457) ;
		\draw [color={rgb, 255:red, 0; green, 0; blue, 0 }  ,draw opacity=1 ][line width=1.5]    (454.5,442) .. controls (485.5,475) and (496.5,480) .. (500.5,441) ;
		\draw [color={rgb, 255:red, 0; green, 0; blue, 0 }  ,draw opacity=1 ][line width=1.5]    (472.5,509) .. controls (490.5,471) and (492.5,484) .. (518.5,510) ;
		\draw [color={rgb, 255:red, 155; green, 155; blue, 155 }  ,draw opacity=1 ][line width=1.5]    (508,506) .. controls (627.5,636) and (682.5,560) .. (715.5,506) ;
		\draw [color={rgb, 255:red, 155; green, 155; blue, 155 }  ,draw opacity=1 ][line width=1.5]    (736.5,457) .. controls (746.5,438) and (753.5,425) .. (760.5,392) ;
		\draw [color={rgb, 255:red, 155; green, 155; blue, 155 }  ,draw opacity=1 ][line width=1.5]    (781.5,329) .. controls (798.5,281) and (800.5,262) .. (812.5,214) ;
		\draw [color={rgb, 255:red, 155; green, 155; blue, 155 }  ,draw opacity=1 ][line width=1.5]    (830.5,150) .. controls (842.5,110) and (859.5,54) .. (892.46,45.22) ;
		
		\draw (38.4,191.6) node [anchor=north west][inner sep=0.75pt]  [rotate=-270.02]  {$\mathbf{J}_{1,n,\lor }$};
		\draw (240.32,340.36) node [anchor=north west][inner sep=0.75pt]    {$\mathbf{I}_{1,n,\lor }$};
		\draw (312,147.4) node [anchor=north west][inner sep=0.75pt]    {$\mathbf{E}_{1,n,\lor }$};

	\end{tikzpicture}
\end{center}

\label{f:RelCollection}
\end{figure}

Instead, we divide the energy axis into several overlapping intervals, and we construct a collection $\mathcal{E}_1$ of well-separated Rellich functions of certain Dirichlet restrictions of $H$ whose domains cover the circle $\mathbb{T}$ with the same structural properties as $\mathbf{E}_0$, cf. Figure \ref{f:RelCollection}. We choose an initial interval length $\ell_1^{(1)}$ and consider energy regions of size $\mathcal O((\ell_1^{(1)})^{-16})$.
Each energy region can be characterized as double-resonant, if it contains some $E_n$ which satisfies $E_n = V(\theta_n) = V(\theta_n+n\omega)$ for some $\theta\in\mathbb T$ and $|n| \leq \ell_1^{(1)}$, or simple-resonant if it does not. Each function $\mathbf E_1 \in \mathcal E_1$ is a Rellich function of $H_{\Lambda_1}$, where $\Lambda_1 \subset \mathbb Z$ is an interval of length $\ell_1^{(1)}$ if the energy region is simple-resonant, or $\ell_1^{(2)} \approx \left(\ell_1^{(1)}\right)^2$ if the energy region is double-resonant. 
The singular intervals are then characterized by 
\begin{equation*}
\mathcal S_1 = \{\Lambda_1 + m\,|\,m\in\mathbb Z,\,|\mathbf E_1(\theta_*+m\omega)-E_*| < \delta_1\}
\end{equation*} 
where $\mathbf E_1 \in \mathcal E_1$ is the Rellich function defined in the energy region containing $E_*$, and $\delta_1$ is defined as Fr\"ohlich, Spencer, and Wittwer define it.

Assuming the constructed Rellich functions satisfy a Morse condition, maintain two monotonicity intervals, and are well-separated from other Rellich functions on the same domain (i.e., we have an upper bound on $\|[P^\perp(H_{\Lambda_1}-\mathbf E_1)P^\perp]^{-1}\|$, as considered above), we can iterate this procedure inductively and conclude the proof as Fr\"ohlich, Spencer, and Wittwer do. 
While we cannot control the bad set of $\theta \in \mathbb T$ by the $m$ function as they do, we can bound it by controlling the number of Rellich functions we construct in $\mathcal E_n$ at each scale. 
Since the energy regions at scale $s$ are of size at least $\mathcal O(\delta_{s-2}^{3})$, 
each energy region at scale $s-1$ gives rise to at most $\mathcal O(\delta_{s-2}^{-3})$ Rellich functions at scale $s$; 
thus, we inductively bound $|\mathcal E_n| \leq \mathcal O(\delta_{n-2}^{-4})$.  The bad set of $\theta$ at scale $n$ for a specific $\mathbf E_n \in \mathcal E_n$ is bounded in measure by $\ell_{n+1}^2\delta_{n-1}^{1/4}$ by a calculus argument. Since $\delta_{n-2}^{-4}\ell_{n+1}^2\delta_{n-1}^{1/4}$ is still summable, we can apply Borel-Cantelli.

It remains to show that the Rellich functions in $\mathcal E_{n+1}$ inherit the structural properties of those in $\mathcal E_n$; 
namely, a Morse condition and a uniform separation estimate. 
By construction, simple resonant Rellich functions are well-separated from others, so they 
satisfy $\|\mathbf E_{n+1} - \mathbf E_n\|_{C^2} \ll \delta_{n}$ by the same arguments used above. 
In the double-resonant case, the Morse lower bound on the second derivative follows by a slight modification of the above argument to allow for $V$'s asymmetry. 
A new argument is required to separate the pair of double-resonant Rellich functions uniformly by a stable, quantifiable gap. 
We thus show 
\begin{lem} \label{l:DRRelFn}
In our setting, double resonances of a Rellich function $\mathbf{E}_n$ of $H_{\Lambda_n}$ resolve as a pair of uniformly locally separated Morse Rellich functions $\mathbf{E}_{n+1,\vee} > \mathbf{E}_{n+1,\wedge}$ of $H_{\Lambda_{n+1}}$ with at most one critical point, cf. Figure \ref{f:DRRelFn}. The size of the gap is larger than the next resonance scale:
\begin{equation*}
\inf \mathbf E_{n+1,\vee} - \sup \mathbf E_{n+1,\wedge} \gg \delta_{n+1}
\end{equation*}
\end{lem}
This gap ensures that any Rellich function $\mathbf{E}_n$ can resonate only with itself at future scales, which ultimately enables our induction.
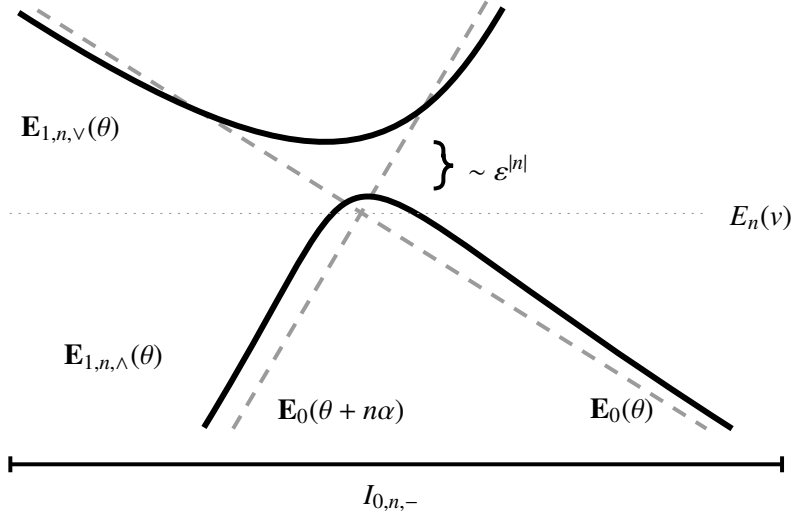
\begin{figure}[h]
	\caption{The resolution of a double resonance of $\mathbf{E}_0 = V$ into a pair of uniformly locally well-separated Rellich curves of a Dirichlet restriction $H_{\Lambda_1}$ of $H$. The curves $\mathbf E_0(\theta)$ and $\mathbf E_0(\theta+n\alpha)$ need not interlace the Rellich curves $\mathbf E_1$, but the auxiliary curves $\lambda,\tilde\lambda$ (not pictured) must.}
	\begin{center}
	\begin{tikzpicture}[x=0.75pt,y=0.75pt,yscale=-0.6,xscale=0.7]
		
		\draw [color={rgb, 255:red, 155; green, 155; blue, 155 }  ,draw opacity=1 ][line width=1.5]  [dash pattern={on 5.63pt off 4.5pt}]  (78.5,21) -- (560.5,373) ;
		\draw [color={rgb, 255:red, 155; green, 155; blue, 155 }  ,draw opacity=1 ][line width=1.5]  [dash pattern={on 5.63pt off 4.5pt}]  (219.5,373) -- (402.5,13) ;
		\draw [color={rgb, 255:red, 0; green, 0; blue, 0 }  ,draw opacity=1 ][line width=2.25]    (65,23) .. controls (283,186) and (348,150) .. (414,19) ;
		\draw [line width=2.25]    (198.5,372) .. controls (342,98) and (251.5,128) .. (579,373) ;
		\draw  [line width=1.5]  (364,171) .. controls (368.67,171) and (371,168.67) .. (371,164) -- (371,161.78) .. controls (371,155.11) and (373.33,151.78) .. (378,151.78) .. controls (373.33,151.78) and (371,148.45) .. (371,141.78)(371,144.78) -- (371,139) .. controls (371,134.33) and (368.67,132) .. (364,132) ;
		\draw [color={rgb, 255:red, 155; green, 155; blue, 155 }  ,draw opacity=1 ] [dash pattern={on 0.84pt off 2.51pt}]  (58.5,192) -- (558.5,192) ;
		\draw [line width=1.5]    (59,403) -- (615,403) ;
		\draw [shift={(615,403)}, rotate = 180] [color={rgb, 255:red, 0; green, 0; blue, 0 }  ][line width=1.5]    (0,6.71) -- (0,-6.71)   ;
		\draw [shift={(59,403)}, rotate = 180] [color={rgb, 255:red, 0; green, 0; blue, 0 }  ][line width=1.5]    (0,6.71) -- (0,-6.71)   ;
		
		\draw (387.08,140.63) node [anchor=north west][inner sep=0.75pt]  [rotate=-359.58,xslant=0.02]  {$\sim \varepsilon ^{|n|}$};
		\draw (250,346.4) node [anchor=north west][inner sep=0.75pt]    {$\mathbf{E}_0( \theta +n\alpha )$};
		\draw (65,108.4) node [anchor=north west][inner sep=0.75pt]    {$\mathbf{E}_{1,n,\lor }( \theta )$};
		\draw (475,345.4) node [anchor=north west][inner sep=0.75pt]    {$\mathbf{E}_0( \theta )$};
		\draw (96,300.4) node [anchor=north west][inner sep=0.75pt]    {$\mathbf{E} _{1,n,\land }( \theta )$};
		\draw (575,183.4) node [anchor=north west][inner sep=0.75pt]    {$E_{n}( v)$};
		\draw (311,418.4) node [anchor=north west][inner sep=0.75pt]    {$I_{0,n,-}$};

	\end{tikzpicture}
	\end{center}

\label{f:DRRelFn}
\end{figure}

To prove Lemma \ref{l:DRRelFn}, we 
interlace two auxiliary curves between the double-resonant Rellich pair. 
Specifically, let $\mathbf E_n(\theta),\mathbf{\tilde E}_n(\theta)$ be the two resonant Rellich functions with corresponding eigenvectors $\psi(\theta),\tilde\psi(\theta)$.
 By the Min-Max Principle, there must be an eigenvalue $\tilde\lambda$ of $P^\perp_\psi H_{\Lambda_{n+1}}P^\perp_\psi$ satisfying 
\begin{equation*} 
\mathbf E_{n+1,\wedge}(\theta) \leq \tilde\lambda(\theta) \leq \mathbf E_{n+1,\vee}(\theta)
\end{equation*}
Moreover, since we have projected away from one resonance, the arguments from the simple-resonance case can be used to show that $\|\tilde\lambda-\mathbf{\tilde E}_n\|_{C^1} \ll \delta_n$. As a consequence of the Morse condition, $|\partial_\theta\mathbf{\tilde E}_n| \gg \delta_n$, so $|\partial_\theta\tilde\lambda|$ is similarly bounded below. 
By repeating this process to construct an eigenvalue $\lambda$ of $P^\perp_{\tilde\psi}H_{\Lambda_{n+1}}P^\perp_{\tilde\psi}$ with $\|\lambda - \mathbf E_n\|_{C^1} \ll \delta_n$, we construct two curves, with large opposite-signed first derivatives, which separate $\mathbf E_{n+1,\vee}$ and $\mathbf E_{n+1,\wedge}$. Combining this with the pointwise separation bound gives a uniform separation bound, proving Lemma \ref{l:DRRelFn} and allowing the inductive argument to proceed.

No version of this proof currently exists for more than two critical points. In higher dimensions, which can mean both a higher-dimensional lattice Laplacian, as well as potentials defined on $\tor^d$ with $d\ge2$, it is even more daunting to implement   this perturbative proof strategy. This is why we will impose a much more rigid assumption on the potential function, namely analyticity, for the remainder of these lectures. Smooth potentials are a largely uncharted territory, especially in higher dimensions. 

\section{Subharmonic functions in the plane}

This section\footnote{Based on notes written and typed by Adam Black during a graduate class by the author at Yale.} establishes some standard facts about harmonic and subharmonic functions in the plane. In the subsequent development of the theory of quasi-periodic localization for analytic potentials, we will make heavy use of such results as Riesz' representation of subharmonic functions, and the Cartan estimate.  A reader familiar with this material can move on to the following section. 

\subsection{Motivation and definition}
Let $ \Omega\subset \bbC $ be a domain (open and connected). Let $\mathcal{H}(\Omega)$ denote the holomorphic functions on $\Omega $. What sort of function is $ \log \abs{f(z)}$ for $ f\in\mathcal{H}(\Omega) $ with $ f\not\equiv 0 $? Recall that for $ f\in \mathcal{H}(\Omega) $ if $ f\neq 0$ in $ \Omega $ simply connected then there exists $g\in\mathcal{H}(\Omega) $, unique up to an additive constant in $2\pi i\Z$,  such that $ f=e^g $. Indeed,  if $ f=e^g $ then $ f'=g'e^f $ so that $ g'=\frac{f'}{f} $. Then for any $ z_0\in \Omega $, set $ g(z)=g(z_0)+\int_{z_0}^z\frac{f'(w)}{f(w)}\,dw $, where this integral is well-defined because the integrand is holomorphic and $ \Omega $ is simply connected. 
The upshot of this is that for non-vanishing $ f $, $ \log\abs{f}=\log e^{\Re g}=\Re g$ so that $ \log \abs{f} $ is harmonic. Notice that this is still true if $ \Omega $ is not simply connected because being harmonic is a local property and we can always find the existence of such a $ g $ in a disc around any point. Now, if $ f(z_0)=0 $, then we may write $ f(z)=(z-z_0)^n\tilde{f}(z) $ where $ \tilde{f}(z) $ does not vanish in some neighborhood of $ z_0 $. In this neighborhood, we have 
\begin{equation}\nn
\log\abs{f(z)}=n\log\abs{z-z_0}+\log\abs{\tilde{f}(z)}
\end{equation}
which we can make sense of in the entire neighborhood by declaring $ \log\abs{z-z_0}=-\infty $ at $ z=z_0 $. Indeed, this function is continuous as map into $ \bbR\cup \{-\infty\} $ relative to   the natural topology. More generally, if $ K\subset\subset \Omega $ (that is, compactly contained) then we let $ \{\zeta_j\}_{j=1}^N $ be the zeroes of $ f $ in $ K $ counted with multiplicity so that $ f(z)=\prod_{j=1}^N(z-\zeta_j)F(z) $ where $ F $ is holomorphic on some $ \Omega' \supset K $ and $ F\neq 0 $ in $ \Omega' $. Then $ \log \abs{f(z)}=\sum_{j=1}^N\log\abs{z-\zeta_j}+\log\abs{F(z)} $. 
From this we infer what type of function $ \log\abs{f} $ is, namely it is harmonic away from the zeroes of $ f $, and $ -\infty $ there, so the value of the function should be \emph{lower} than its average on a small disc. This motivates the following definition, which applies to   all dimensions. However, throughout we limit ourselves to the plane.  
\begin{defn}
	A function $ u $  is {\em subharmonic} on $ \Omega\subset \bbR^2 $, denoted $ u\in\mathcal{SH}(\Omega)$, if 
\begin{itemize}
	\item $ u:\Omega\rightarrow [-\infty,\infty)$ is upper semi-continuous (usc)
	\item $ u $ satisfies the subharmonic mean value property (smvp): 
	\begin{align*}
		u(x_0)\leq\fint_{\partial \bbD(x_0,r)}u(y)\,dy
	\end{align*}
	for any disk $ \bbD(z_0,r)$ such that $ \overline{\bbD(z_0,r)}\subset \Omega $.
\end{itemize}
\end{defn}
One should think of subharmonic functions as lying below harmonic ones, see Corollary~\ref{co:sh} below.  Hence,  in one dimension, subharmonic functions are convex as they lie below lines, which are the one-dimensional harmonic functions. 
The integral in the above definition is well defined (although it may be $ -\infty $) because of the following lemma. 

\begin{lem}
	Let $ f:K\rightarrow [-\infty,\infty) $ be usc with $ K $ compact. Then $f$ attains its maximum.
	\end{lem}
\begin{proof}
	Let $ M:=\sup_{x\in K}f(x) $. Let $ f(x_i)\rightarrow M $ as $ i\rightarrow\infty $. By compactness, pass to a subsequence if necessary so that $ x_i\rightarrow x $. Then $ M=\limsup_{i\rightarrow\infty} f(x_i)\leq f(x)\leq M $.
\end{proof}

\subsection{Basic properties}
In this section we prove some basic properties of subharmonic functions. Readers familiar with the properties of harmonic functions may find these proofs rather familiar.
\begin{prop} 
	If $ u\in\mathcal{SH}(\Omega) $ then $ u(z_0)\leq \fint_{\bbD(z_0,r)}u(z)\,dz $ for all $ \overline{\bbD(z_0,r)}\subset \Omega $.
	\end{prop}
\begin{proof}
For all $ 0<s\leq r $ we have that
\begin{align*}
	u(z_0)\, |\partial \bbD(z_0,s)|\leq\int_{\partial \bbD(z_0,s)}u(z)\,dz
\end{align*}
so that the result follows immediately by integrating both sides from $ 0 $ to $ r $ with respect to $ s $.
\end{proof}

\begin{cor}\label{co:aeeq}
	Let $ u,v\in\mathcal{SH}(\Omega) $ such that $ u(z)=v(z) $ for almost every $ z $. Then $ u\equiv v $.
	\end{cor}
\begin{proof}
	By the smvp and the fact that $ u$ and $ v $  are equal almost everywhere, we see that for every $ z_0 $ for any $ r>0 $ such that $ \overline{\bbD(z_0,r)}\subset \Omega $ 
\begin{align*}
	u(z_0)\leq \fint_{\bbD(z_0,r)}v(z)\,dz
\end{align*}
Let $ r_i\rightarrow 0 $ and let $ v(z) $ attain its maximum on $ \bbD(z_0,r_i) $ at $ z_i $ so that $z_i\rightarrow z_0 $. Thus for all $ i$
\begin{align*}
	u(z_0)\leq\fint_{\bbD(z_0,r_i)}v(z)\,dz\leq v(z_i)
\end{align*}
so taking limsups we see that
\begin{align*}
	u(z_0)\leq v(z_0)
\end{align*}
by usc. By symmetry, we have also that $ v(z_0)\leq u(z_0)$, so we are done.
\end{proof}

\begin{lem}
	Suppose $ u\in C^2(\Omega) $. Then $ u\in \mathcal{SH} $ iff $ \Delta u(z)\geq 0 $ for all $ z\in\Omega $.
	\end{lem}
\begin{proof}
Define
\begin{align*}
	(Mu)_{x_0}(r):=\fint_{\partial \bbD(x_0,r)}u(y)\,\sigma(dy)=\fint_{\abs{w}=1}u(x_0+rw)\,\sigma(dw) 
\end{align*}
where $ \sigma $ is the surface measure on the circle. We compute
\begin{align*}
	\partial_r (Mu)_{x_0}(r)=\fint_{\abs{w}=1}\grad u(x_0+rw)\cdot w\,\sigma(dw)=\frac{1}{\abs{\partial \bbD(0,1)}r}\int_{\partial \bbD(x_0,r)}\grad u(y)\cdot \vec{n}\,\sigma(dw)
\end{align*}
which by the divergence theorem is equal to
\begin{align*}
	\frac{1}{\abs{\partial \bbD(0,1)}r }\int_{\bbD(x_0,r)}\Delta u(y)\,dy
\end{align*}
Thus, we see that if $ \Delta u\geq 0 $ then $ (Mu)_{x_0}(r) $ is non-decreasing with $ r $, and as its limit as $ r\rightarrow 0 $ is $ u(x_0) $, one direction follows. For the other direction, note that if $ \Delta u(x_0)< 0 $ then there exists some disk $\bbD(x_0,r)$ on which $ \Delta u(x)<0 $. The above computation then shows that $ (Mu)_{x_0} $ is decreasing for small enough $ r $, which contradicts the smvp.
\end{proof}

\begin{prop} 
	The function $ f(z)=\log\abs{z} $ is subharmonic on $ \R^2 $.
	\end{prop}
\begin{proof}
	Let $ f_n=\frac{1}{2}\log(\abs{z}^2+1/n)$. Then it is easy to compute in polar coordinates that $ \Delta f_n=\frac{2(1/n)}{(r^2+1/n)^2}\geq 0 $ so that because $ f_n$ is $ C^2(\R^2) $, it is subharmonic. On $ \partial \bbD(z_0,r) $, the sequence $ \{f_n\} $ is bounded above by some $ M $ so that $M-f_n  $ is a positive monotone sequence of integrable functions. By applying the monotone convergence theorem to this sequence we see that
	\begin{align*}
		\lim_{n\rightarrow\infty}\fint_{\partial \bbD
		(z_0,r)}f_n(z)\,dz=\fint_{\partial \bbD(z_0,r)}f(z)\,dz
	\end{align*}
from which the result follows.
\end{proof}

\begin{lem}
The maximum or sum of finitely many subharmonic functions is subharmonic.
\end{lem}
\begin{proof}
Follows directly from the definition. 
\end{proof}

\begin{lem}[Maximum principle]
	Let $ u\in \mathcal{SH}(\Omega) $ with $ \Omega $ connected and suppose there exists $ z_0\in\Omega $ such that $ u(z)\leq M:=u(z_0) $ for all $ z\in\Omega $. Then $ u $ is constant.
	\end{lem}
\begin{proof}
	Consider $ S=\{z\in\Omega\mid u(z)=M\} $. This set is closed because $ u $ is usc. Furthermore it is open because if $ f(z)=M $,  then  $ f(z)\leq \fint_{\partial \bbD(z,r)}f(w)\,dw $ implies that $ f(w)=M $ for all $ w\in\partial \bbD(z,r) $.
\end{proof}

The following result explains the terminology {\em subharmonic}. 

\begin{cor}\label{co:sh}
	Let $ u\in \mathcal{SH}(\Omega) $. If $ v $ is harmonic on $ \overline{\Omega'}\subset \Omega $ for $ \Omega' $ bounded and $ v\geq u $ on $ \partial\Omega' $  
	then $ v\geq u $ in $ \Omega' $.
	\end{cor}
\begin{proof}
The function $ u-v $ is subharmonic so that if $ u-v>0 $  in $ \Omega' $ then it would have a maximum in this region, violating the above.
\end{proof}

\subsection{Review of harmonic functions}
In the next section we will need some basic facts about harmonic functions, which we now briefly recall. They can be found in many places, such as~\cite{John}.  
For $ \Omega\subset\R^2$ a bounded region with smooth boundary, say, we would like to solve the boundary value problem
\begin{align*}
	\begin{cases}
		\Delta u(z)=f&z\in \Omega\\
		u(z)=g&z\in\partial \Omega
\end{cases}
\end{align*}
Recall Green's identity for $ u,v\in C^2(\Omega) $:
\begin{align}
\label{eq:Green}
	\int_\Omega u(\zeta)\Delta v(\zeta)\,d\zeta=\int_\Omega \Delta u(\zeta)v(\zeta)\,dz\zeta +\int_{\partial\Omega}\big( u\frac{\partial v}{\partial n} -\frac{\partial u}{\partial n}v\big) \,d\sigma
\end{align}
If $ v=G(z,\zeta) $ is such that (in the sense of distributions) $ \Delta_z G(z,\zeta)=\delta_\zeta(z) $ and $ G(z,\zeta)=0 $ for $ z\in \partial \Omega $ then 
\begin{align}\label{eq:Grep}
	u(z)=\int_{\Omega }G(z,\zeta)f(\zeta)\,m(d\zeta)+\int_{\partial \Omega}\frac{\partial G}{\partial n}(z,\zeta)g(\zeta)\,d\sigma
\end{align}
with $m$ Lebesgue measure in the plane and $\sigma$ surface measure on the boundary. 
Such a {\em Green  function}  $ G(z,\zeta) $ exists for any bounded domain $\Omega$ for which $\partial\Omega$ satisfies an exterior cone condition. 
This is a standard application of Perron's method, see~\cite{John} (this method applies to any dimension). 
For the case of a disk $\bbD(0,R)\subset\bbC $, there is the explicit formula given by the logarithm of the absolute value of the conformal automorphism of the disk: 
\begin{align}\label{eq:Gdisk}
G(z,\zeta)=\frac{1}{2\pi}\log\abs{z-\zeta}+\frac{1}{2\pi}\log\abs{\frac{R}{R^2-z\overline{\zeta}}}
\end{align}
In particular, by \eqref{eq:Green} a harmonic function on $ \Omega $ which is $C^2(\bar{\Omega})$ with boundary values   $g$ is given by
\begin{align}\label{eq:Pform}
	u(z)=\int_{\partial \Omega}\frac{\partial G}{\partial n}(z,\zeta)g(\zeta)\,\sigma(d\zeta)
\end{align}
This is  {\em Poisson's formula} and $ P_\zeta(z)=\frac{\partial G}{\partial n}(z,\zeta) $ is the Poisson kernel of $\Omega$. 
If $g\in C(\partial\Omega)$, then~\eqref{eq:Pform} defines a harmonic function in $\Omega$ which is the unique solution of the boundary value problem (uniqueness by the maximum principle). 
For the disc of radius $ r $ in the plane we have 
\begin{align*}
	 P_\zeta(z) =\frac{1}{2\pi}\frac{r^2-\abs{z}^2}{r\abs{z-\zeta}^2}
\end{align*}
and there is an analogous expression in higher dimensions. 
This implies Harnack's  inequality, which controls the value of a positive harmonic function on a disc by its value at the center. 

\begin{prop} 
	Let $ u $ be positive harmonic function on the disk $ \overline{\bbD(z_0,R)} \subset \C$. Then for $ \abs{z-z_0}<r $ 
\begin{align*}
	\frac{R-r}{R+r}u(z_0)\leq u(z)\leq \frac{R+r}{R-r}u(z_0)
\end{align*}
\end{prop}
\begin{proof}
Simply bound the Poisson kernel and then apply the mean value property.
\end{proof}

Finally, we recall the following compactness property of families of harmonic  functions  (the analogue of normal families in complex analysis). It is valid in all dimensions but we state it only in the plane. 
\begin{thm} 
\label{thm:normalF}
	A sequence of harmonic functions on $\Omega\subset\C$ that is uniformly bounded on each compact subset of $\Omega$ has a subsequence which converges to some harmonic $ u $ uniformly on each compact subset.
\end{thm}
\begin{proof}
If $ u(z) $ is harmonic on $ \bbD(a,r) $ then taking derivatives of~\eqref{eq:Pform} shows that  $ \abs{D^\alpha u(a)}\leq \frac{C_\alpha \oP{u}_{L^\infty}}{r^\alpha} $ for some universal constant $ C_\alpha $. Thus, any uniformly bounded sequence of harmonic functions is in fact equicontinuous. We can then take a convergent subsequence on any compact subset by Arzela-Ascoli at which point a diagonal argument with increasing compact sets finds the desired $ u $. By the mean value property  $ u $ is harmonic.
\end{proof}

\subsection{Riesz representation of subharmonic functions in $\bbC$}
As noted earlier,  any subharmonic function of the form $ \log\abs{f} $ for $ f\in\mathcal{H}(\Omega) $ admits the representation for any $\Omega'\Subset \Omega$ (compact containment): 
$$ \log\abs{f}=\sum_{j=1}^N\log\abs{z-\zeta_j}+h(z) $$ with $ h $ harmonic in $\Omega'$ and $\zeta_j\in\Omega'$. We can think of this expression as $ \int_\Omega\log\abs{z-\zeta}\,\mu(d\zeta)+h(z) $ where $ \mu = \sum_{j=1}^N\delta_{\zeta_j} $. Note that $h$ is bounded on any $\Omega''\Subset\Omega'$ but not necessarily on $\Omega'$. 
This section develops an analogous representation for all subharmonic functions, known as Riesz representation. The difference is that we can allow any positive finite measure~$\mu$. We begin with some basic properties of logarithmic potentials of such measures. 
\begin{prop} 
	Let $ \Omega $ be a bounded domain and $ \mu\in\mathcal{M}^+(\Omega) $, that is, a positive finite Borel measure on $ \Omega $. Then with $ u(z):=\int_{\Omega}\log\abs{z-\zeta}\,\mu(d\zeta)$
\begin{itemize}
\item $ u\in\mathcal{SH}(\Omega)$
\item $ u> -\infty $ (Lebesgue) almost everywhere
\item $ u $ is bounded above 
\end{itemize}
\end{prop}
\begin{proof}
	Note that for $ z,\zeta \in\Omega $, $ \log\abs{z-\zeta}\leq \log(\text{diam} \Omega) $ so that $ u(z)\leq \log(\text{diam} \Omega)\mu(\Omega) $, which shows that $ u $ is bounded above.\par
Consider $ \bbD\Subset  \Omega$ a disc of radius $ R $. Then with $ m  $  the Lebesgue measure in $\R^2$, 
\begin{align*}
	\int_{\bbD} u(z)\,m(dz)=\int_\Omega\int_{\bbD}\log\abs{z-\zeta}\,m(dz)\,\mu(d\zeta)
\end{align*}
by Fubini-Tonelli because the integrands are bounded from above. In fact, $ \log\abs{z-\zeta} $ is Lebesgue integrable on $ \bbD$: $$ \int_{\bbD} \log\abs{z-\zeta}\,m(dz)\geq \int_{\bbD(0,R)}\log\abs{z}\,m(dz)=2\pi \int_0^Rr\log r\cdot \,dr>-\infty  $$ 
which also shows that the total integral is $ >-\infty $. Since this holds for any disc, we have shown that $ u>-\infty $ a.e.\ in $\Omega$. 
To see that $ u $  is usc, observe that if $ z_j\rightarrow z $  then by (the reverse) Fatou's lemma 
\begin{align*}
	\limsup_{j\rightarrow\infty}u(z_j)=\limsup_{j\rightarrow\infty}\int_\Omega\log\abs{z_j-\zeta}\,\mu(d\zeta)\leq\int_\Omega\limsup_{j\rightarrow\infty}\log\abs{z_j-\zeta}\,\mu(d\zeta)=u(z)
\end{align*}
where the use of Fatou's lemma is justified due to the uniform upper bound on $\log|z-\zeta_j|$. 
Finally, note that for $ \bbD $ a disc centered at $ z_0$
\begin{align*}
	\fint_{\partial \bbD}u(z)\,dz =\int_\Omega \fint_{\partial \bbD}\log\abs{z-\zeta}\,dz\,\mu(d\zeta)
	\geq \int_\Omega \log\abs{z_0-\zeta}\,d\mu(\zeta)=u(z_0)
\end{align*}
which shows that $ u(z) $ satisfies the smvp because $ \log\abs{z} $ does.
\end{proof}

\begin{remark} 
We cannot hope for any better than usc from this construction. For instance, consider $ \mu=\sum_{n=1}^\infty 2^{-n}\delta_{2^{-n}} $ so that $ u(z)=\sum_{i=1}^\infty 2^{-n}\log\abs{z-2^{-n}} $. Then   $ u(0)=2\log 2 $  but $ u(2^{-n})=-\infty $ for all $ n $.
\end{remark}
We will also require the following smooth approximation result. 
\begin{lem}\label{lem:moll sh}
	Let $ u\in\mathcal{SH}(\Omega) $ where $ \Omega $ is a bounded domain. Then there exists a sequence $ u_n\in\mathcal{SH}(\Omega_{1/n})\cap C^\infty(\Omega_{1/n}) $  where $ \Omega_{1/n}:=\{z\in \Omega\mid \text{dist}(z,\partial\Omega)>1/n\}  $ such that $ u_n\rightarrow u $ pointwise and monotone decreasing (in $ \Omega_{1/n_0} $ for $ n>n_0 $).
	\end{lem} 
\begin{proof}
	We accomplish this via mollification, so let $ \varphi\in C^\infty(\bbR^2) $ be a radial function satisfying $ \varphi(x)\geq 0 $, $ \varphi(x)=0 $ for $ \abs{x}\geq 1 $ and $ \int_{\bbR^2}\varphi(x)\,dx=1 $. Define also $ \varphi_n=n^{2} \varphi(nx) $. We claim that $ u_n(z)=(u*\varphi_n)(z) $ satisfies the desired properties. It is clearly smooth and well-defined on $ \Omega_{1/n} $. The smvp for $ u_n $ follows from  Fubini's theorem and $ \varphi_n\ge0 $. 
	To see that $ u_n $ is decreasing, write
\begin{align*}
	u_n(z)=n^{2}\int_{\bbR^2} u(z-w)\varphi(nw)\,dw=2\pi \int_0^\infty\int_0^{1}u(z-\frac{r}{n}e(\theta))\,d\theta\, r\varphi(r)\,dr \ge u(z)
\end{align*}
with $ e(\theta)=e^{2\pi i \theta} $, the final inequality implied by the smvp. First,  $ v(\zeta):=\int_0^1u(z-\zeta e(\theta))\,d\theta $ is subharmonic since it is easily seen to be usc,  and the smvp follows by Fubini (note that $u$ remains subharmonic after a rotation and translation). Second, it is radial and thus an increasing (but not necessarily in the strict sense) function of~$|\zeta|$ by the maximum principle. 
Finally,  $u_n(z)\leq \max_{\abs{z-w}\leq 1/n} u(w) $ for $ z\in\Omega_{1/n} $ so that by usc $ u_n(z)\rightarrow u(z) $ as $ n\rightarrow \infty $. 
\end{proof}

We are now ready to prove Riesz's  representation theorem for subharmonic functions.

\begin{thm}\label{thm:Riesz}
	Let $ u\in\mathcal{SH}(\Omega) $ where $ \Omega $ is some neighborhood of $ \overline{\bbD(0,4)} $. Suppose that $ u\leq M $ on $ \overline{\bbD(0,4)} $ and $ u(0)\geq m>-\infty $. Then there exists $ \mu\in\mathcal{M}^+(\bbD(0,3)) $ and $ h $ harmonic in $ \bbD(0,3) $ such that for all $z\in{\bbD(0,3)} $
	\begin{align*}
		u(z)=\int_{\bbD(0,3)}\log\abs{z-\zeta}\,\mu(d\zeta)+h(z)
	\end{align*}
	Furthermore, there exists $ C_0>0 $ universal such that  $ \|h-M\|_{L^\infty(\bbD(0,2))} \le C_0(M-m)$ and $ \mu(\bbD(0,3))\leq C_0(M-m) $. In fact, for any $\delta\in (0,1)$ there exists $C_0(\delta) $ so that $ \|h-M\|_{L^\infty(\bbD(0,3-\delta))} \le C_0(\delta)(M-m)$.  
	\end{thm}
\begin{proof}
	We first reduce to the smooth case. To this end, suppose that the claim holds for all $ v\in \mathcal{SH}(\overline{\bbD(0,4)}\cap C^\infty(\overline{\bbD(0,4)})$. Choose any $ u\in\mathcal{SH}(\overline{\bbD(0,4)}) $ and let $ u_n\rightarrow u$ in $  \bbD(0,4) $ be as in Lemma~\ref{lem:moll sh}. We then have with some decreasing $M_n\to M$
	\[
	u_n\le M_n \text{\ on\ } \overline{\bbD(0,4)},\quad u_n(0)\ge m
	\]
	By validity of the theorem in the smooth case we may write
	\EQ{\label{eq:RRS}
		u_n(z)=\int_{\bbD(0,3)}\log\abs{z-\zeta}\,\mu_n(d\zeta) +h_n(z)
	}
	and because $ u_n$ is monotone decreasing and uniformly bounded above on any compact set, for any $ \varphi\in C(\overline{\bbD(0,3)}) $ we have that $$ \Span{u_n,\varphi}\rightarrow \Span{u,\varphi}=\int_{\overline{\bbD(0,3)}} u(x+iy) \varphi(x+iy) \, dxdy $$ by the monotone convergence theorem. By assumption, the above measures are uniformly bounded, so by Banach-Alaoglu we may take a weak-* limit in $C(\overline{\bbD(0,3)})^*$, thus $ \mu_n\rightarrow\mu $ in the weak-* sense where $\mu$ is a finite Borel measure on $\overline{\bbD(0,3)}$ which satisfies 
	\[
	\mu(\overline{\bbD(0,3)})\le C_0(M-m)
	\]
	 Since,  with $m(dz)$ being Lebesgue measure in the plane, 
	\[
	 \psi(\zeta):=\int_{\overline{\bbD(0,3)}}\log\abs{z-\zeta}\varphi(z)\,m(dz) 
	 \]  
	 is a continuous function of $ \zeta\in\R^2 $, we conclude that
	 \[
	\lim_{n\to\infty} \int_{\overline{\bbD(0,3)}} \psi(\zeta)\, \mu_n(d\zeta)  = \int_{\overline{\bbD(0,3)}} \psi(\zeta)\, \mu(d\zeta)
	 \]
	 which implies that
	 \EQ{\label{eq:logpot conv}
	  \Span{\int_{\overline{\bbD(0,3)}}\log\abs{z-\zeta}\,\mu_n(d\zeta),\varphi}\rightarrow \Span{\int_{\overline{\bbD(0,3)}}\log\abs{z-\zeta}\,\mu(d\zeta),\varphi}.
	  }
	  By the theorem in the smooth case,  $$ \limsup_{n\to\infty}\;\oP{h_n-M}_{L^\infty(\bbD(0,3-\delta))}\le C_0(\delta)(M-m) $$ so by Theorem~\ref{thm:normalF}  there exists some $ h $ harmonic in $ \bbD(0,3) $ such that a subsequence of $ \{h_n\} $ converges to $ h $ uniformly on all compact subsets of $\bbD(0,3)$. 
	 Thus, for any $\varphi \in C(\bbD(0,3))$ of compact support, $\langle h_n,\varphi\rangle \to \langle h,\varphi\rangle$ along this sequence. In combination with~\eqref{eq:RRS}, \eqref{eq:logpot conv} we conclude that 
	 \[
\langle u,\varphi \rangle = \Span{\int_{\overline{\bbD(0,3)}}\log\abs{z-\zeta}\,\mu(d\zeta)+h(z),\varphi}
	 \]
	 Thus 
	  $$ u(z)=\int_{\overline{\bbD(0,3)}}\log\abs{z-\zeta}\,\mu(d\zeta)+h(z)  \text{\ \ almost everywhere in\ }\bbD(0,3),$$ 
	  which in turn implies equality everywhere by Corollary~\ref{co:aeeq}. Finally, to obtain the desired form we write 
	  $$ \int_{\overline{\bbD(0,3)}}\log\abs{z-\zeta}\,\mu(d\zeta)=\int_{\partial \bbD(0,3)}\log\abs{z-\zeta}\,\mu(d\zeta)+\int_{\bbD(0,3)}\log\abs{z-\zeta}\,\mu(d\zeta) $$ and notice that $ \int_{\partial \bbD{(0,3)}}\log{|z-\zeta|}\,\mu(d\zeta) $ is harmonic in $ \bbD(0,3) $. Thus,
\begin{align*}
	u(z)=\int_{\bbD(0,3)}\log\abs{z-\zeta}\,\mu(d\zeta)+h_0(z)
\end{align*}
where $ h_0(z)=\int_{\partial\bbD(0,3)}\log{|z-\zeta|}\,\mu(d\zeta)+h(z) $ is harmonic in $ \bbD(0,3) $. This harmonic function $h_0$ satisfies similar $L^\infty$ bounds as before, albeit with different constants. 

It remains to prove the theorem for smooth subharmonic functions on $\overline{\bbD(0,4)}$. 
	 In view of~\eqref{eq:Grep} 
\EQ{\label{eq:uG}
	u(z)=\int_{\bbD(0,4) }G(z,\zeta)\Delta u(\zeta)\,m(d\zeta)+\int_{\partial \bbD(0,4)}\frac{\partial G}{\partial n}(z,\zeta)u(\zeta)\,\sigma(d\zeta)
}
so that by using the particular form $ G(z,\zeta) $ in \eqref{eq:Gdisk}, and defining $ \mu(dz):=\frac{1}{2\pi}\Delta u(z)\,dz$ we rewrite the above as
\begin{align*}
	u(z)= \int_{\bbD(0,3)}\log\abs{z-\zeta}\,\mu(d\zeta)+\int_{\bbD(0,3)}\log\frac{4}{\abs{16-z\overline{\zeta}}}\,\mu(d\zeta)+\int_{\bbD(0,4)\setminus \bbD(0,3) }G(z,\zeta)\Delta u(\zeta)\,d\zeta+h_0(z)
\end{align*}
where $ h_0(z):=\int_{\partial \bbD(0,4)}\frac{\partial G}{\partial n}(z,\zeta)u(\zeta)\,\sigma(d\zeta) $ is the harmonic extension of $ u $ to $ \bbD(0,4) $, see~\eqref{eq:Pform}. The second term is harmonic for $ z\in \bbD(0,3) $ because $ 16 -z\overline{\zeta} \ne 0$ and the third term because $ \zeta\in \bbD(0,4)\setminus \bbD(0,3) $ and thus
\EQ{\label{eq:h form}
	h(z)=\int_{\bbD(0,3)}\log\frac{4}{\abs{16-z\overline{\zeta}}}\,\mu(d\zeta)+\int_{\bbD(0,4)\setminus \bbD(0,3) }G(z,\zeta)\Delta u(\zeta)\,d\zeta+h_0(z)
}
is harmonic in $\bbD(0,3)$. We have therefore obtained the desired form for $ u $, we only have left to show the stated bounds. 
To bound $ \mu(\bbD(0,3)) $, use \eqref{eq:uG} to see that
\EQ{\nn
	u(0) & =\int_{\bbD(0,4)}G(0,\zeta)\,\mu(d\zeta)+h_0(0)=\int_{\bbD(0,4)}\log\frac{\abs{\zeta}}{4}\,\mu(d\zeta)+h_0(0) \\
	\log\frac{4}{r}\, \mu(\bbD(0,r)) &\leq \int_{\bbD(0,r)}\log\frac{4}{\abs{\zeta}}\,\mu(d\zeta)=h_0(0)-u(0)\leq M-m
}
where we have used that $ u(0)=m $ and the fact that $ u\leq M $ on $ \partial \bbD(0,4) $ implies that $h_0\le M$. Setting $ r=3 $, we see that $ \mu(\bbD(0,3))\leq C(M-m) $ as desired.
For $ z\in \bbD(0,3) $, the first term in \eqref{eq:h form} is negative by inspection, the second negative since $G<0$,  and the third is bounded above by $ M $ as before. Therefore,  $ h(z)\leq M $. For the reverse bound, Harnack's inequality on $ \abs{z}\leq 3-\delta $ yields
\begin{align*}
	M-h(z)\leq \frac{3+r}{3-r}(M-h(0))\leq \frac{6-\delta}{\delta}(M-h(0))
\end{align*}
and
\begin{align*}
	h(0)=u(0)-\int_{\bbD(0,3)}\log\abs{\zeta}\,\mu(d\zeta)\geq m-\int_{\bbD(0,3)\setminus \bbD(0,1)}\log\abs{\zeta}\,\mu(d\zeta)\geq m-C(M-m)
\end{align*}
so putting these together implies that 
\begin{align*}
M-C_\delta(M-m)\leq h(z)
\end{align*}
for all $ \abs{z}\leq 3-\delta $.
\end{proof}

In fact, by essentially the same proof one can obtain the following more general Riesz representation. Note that one can move the point $z_0$ to $0$ by an automorphism of the disk, which retains the property of being subharmonic. 

\begin{thm}\label{thm:Riesz*}
	Let $ u\in\mathcal{SH}(\overline{\bbD(0,R_1)}) $ and suppose that $ u\leq M $ on $ \overline{\bbD(0,R_1)} $ and $ u(z_0)\geq m>-\infty $ where $|z_0|<R_1$. Let $R_1>R_2>R_3>0$. There exists $ \mu\in\mathcal{M}^+(\bbD(0,R_2)) $ and $ h $ harmonic in $ \bbD(0,R_2) $ such that for all $z\in{\bbD(0,R_2)} $
	\begin{align*}
		u(z)=\int_{\bbD(0,R_2)}\log\abs{z-\zeta}\,\mu(d\zeta)+h(z)
	\end{align*}
	Furthermore, there exist $ C_0=C_0(z_0,R_1,R_2)>0 $ and $ C_1=C_1(z_0,R_1,R_2,R_3)>0 $ universal such that $ \mu(\bbD(0,R_2))\leq C_0(M-m) $ and 
 $ \|h-M\|_{L^\infty(\bbD(0,R_3))} \le C_1(M-m)$. 
  		\end{thm}

See  Theorem~2.2 in~\cite{HLS} for  explicit constants. 

\subsection{Cartan's lower bound}
 Next, we prove Cartan's theorem which controls large negative values of logarithmic potentials. Levin's book~\cite{Levin} has much more on this topic, see page~76. 

\begin{thm}\label{thm:Cartan}
Let $\mu$ be a finite positive measure in $\bbC$ and consider the logarithmic potential
\[
u(z) = \int_{\R^2} \log|z-\zeta|\, \mu(d\zeta)
\]
For any $H\in(0,1)$ there exist disks $\{\bbD(z_j,r_j)\}_{j=1}^J$, for $1\le J\le\infty$ with $\sum_{j=1}^J	r_j\le 5H$ and 
\EQ{\label{eq:lower}
u(z) \ge - \|\mu\| \log(e/H)\qquad \forall\; z\in\C\setminus \bigcup_{j=1}^J \bbD(z_j,r_r)
}
\end{thm}
\begin{proof}
Let $z\in\C$ be a good point if $n(z,r):=\mu(\bbD(z,r))\le pr$ for all $r>0$. Here $p$ depends on $H$ and will be determined. For every bad $z$ there exists $r(z)>0$ with $n(z,r(z))>r(z)p$. Note that $r(z)\le \|\mu\|/p$. 
By Vitali's covering lemma there exist bad points $z_j$  so that $\{ \bbD(z_j,r(z_j))\}_j$ are pairwise disjoint and 
\[
\calB:=\{z\in\C\:|\: z\text{\ \ is a bad point\ \ } \}\subset \bigcup \bbD(z_j,r_j) \text{\ \ with\ }r_j:=5 r(z_j).
\]
In particular, $\sum_j r_j\le 5\|\mu\|/p$ whence we need to set $p=\|\mu\|/H$. If $z\in\C\setminus \bigcup \bbD(z_j,r_j)$, then $z$ is good and we obtain by integrating by parts
\EQ{\nn 
u(z) &\ge \int_0^1 \log r\, d(n(z,r))  = -\int_0^1 \frac{n(z,r)}{r} \, dr \ge -\int_0^H p\, dr + \|\mu\| \log H \\
&=  - pH+\|\mu\|\log H = \|\mu\|\log(H/e)
}
as claimed. 
\end{proof}

We call $\|\mu\|$ the Riesz mass of $u$. 
We leave it to the reader to check that Theorem~\ref{thm:Cartan} with the same proof generalizes as follows.  

\begin{thm}\label{thm:Cartan*}
Under the same assumptions as in the previous theorem, suppose $0<\delta\le1$. Then 
for any $H\in(0,1)$ there exist disks $\{\bbD(z_j,r_j)\}_{j=1}^J$, for $1\le J\le\infty$ with $\sum_{j=1}^J	r^\delta_j\le 5^\delta H$ and 
\EQ{\label{eq:lower*}
u(z) \ge - \frac{1}{\delta}\|\mu\| \log(e/H)\qquad \forall\; z\in\C\setminus \bigcup_{j=1}^J \bbD(z_j,r_r)
}
\end{thm}

We chose $0<\delta\le1$ here instead of $0<\delta\le 2$ since the range $1<\delta\le2$ is weaker than Theorem~\ref{thm:Cartan}. 
As an immediate corollary we conclude that   $\dim(\{z\in\bbC\:|\: u(z)=-\infty\})=0$ in the sense of Hausdorff dimension, for any logarithmic potential of a finite positive measure. By Theorem~\ref{thm:Riesz}, 
this same property  therefore holds locally on $\Omega$ for any subharmonic function on $\Omega$ which is not constant $-\infty$.  For our applications, Cartan's theorem, i.e., Theorem~\ref{thm:Cartan},  will suffice. The following serves to illustrate this result.
\begin{itemize}
\item Consider the logarithm of a polynomial of degree $N$ with roots $\zeta_j\in\C$. Thus, $P(z)=\prod_{j=1}^N (z-\zeta_j)$ and 
\[
u(z) = \log\big|\prod_{j=1}^N (z-\zeta_j)\big| =\int\log|z-\zeta|\,\mu(d\zeta),\qquad \mu = \sum_{j=1}^N \delta_{\zeta_j}
\]
Given $0<H<1$, there exist disks $\bbD(z_j,r_j)$, $1\le j\le J$,  with $\sum_j r_j\le 5H$ and $|P(z)|\ge (H/e)^N$ for all $z\in\bbC\setminus \bigcup \bbD(z_j,r_j)$. By the maximum principle, each disk contains a zero of $P$. Thus, $J\le N$.  The bound on the Riesz mass in Theorem~\ref{thm:Riesz} is nothing other than Jensen's formula counting the roots of analytic functions, see~\cite[page 10]{Levin}. 
\item
 If $\zeta_j=0$ for all $j$, then $|P(z)|=|z|^N\ge H^N$ if $|z|\ge H$. This shows that Cartan's theorem is optimal up to multiplicative constants on~$H$.
 \item On the other hand, suppose $\zeta_j=e(j/N)$ for $1\le j\le N$ where $e(\theta)=e^{2\pi i\theta}$. Then 
$
P(z) = z^N-1
$ 
and we can take the Cartan disks centered at $\zeta_j$ of radius $\rho=1/N$. Then for any $z$ with $z=\zeta_j+\rho e(\theta)$ we have 
\EQ{\label{eq:kreis teil}
|P(z)|=|z^N-1| &= |(\zeta_j+\rho e(\theta))^N -1| \ge  \rho N - \sum_{\ell=2}^N \binom{N}{\ell} \rho^\ell \ge 1 - \sum_{\ell=2}^N \frac{(N\rho)^\ell}{\ell!} = 3-e
}
It follows from the maximum (minimum) principle for analytic functions that $|P(z)|\ge 3-e$ for all $z\in \bbC\setminus \bigcup_{j=1}^N\bbD(\zeta_j,1/N)$. Therefore Cartan's estimate is woefully imprecise in this example. Indeed, for the polynomial $P$ with roots at the $N^{th}$ roots of unity, $u(z)=\log|P(z)|$  behaves in Theorem~\ref{thm:Cartan}  like a subharmonic function with Riesz mass~$1$, at least for $H=1/N$. 
\end{itemize}

In applications of Cartan's theorem to quasi-periodic localization, the distribution of the zeros plays  a decisive role and it is therefore essential to improve on the Cartan bound. In other words, we are in a situation much closer to the roots-of-unity example where Cartan falls far short from the true estimate. Nevertheless, combining Cartan's bound with the dynamics, one can still obtain a nontrivial statement as we shall see in  the following section.  

To conclude this section, we prove Riesz's representation theorem on general from the one for discs which we proved above. We will do this by connection points by chains of disks, which uses Cartan. 

\begin{cor}
\label{cor:Riesz}
Let $\Omega\subset\bbC$ be a bounded domain, $u$ subharmonic on $\Omega$ with $\sup_\Omega u\le M$. Let $K\subset\Omega$ be compact and suppose $\sup_K u\ge m>-\infty$. 
For any $\Omega_2\Subset\Omega_1\Subset\Omega$,  there exist a positive measure $\mu$ on $\Omega_1$ and a harmonic function $h$ on $\Omega_1$ such that
\EQ{\label{eq:RR gen}
u(z) &= \int_{\Omega_1}  \log |z-\zeta|\, \mu(d\zeta) + h(z) \qquad \forall\; z\in \Omega_1\\
\mu(\Omega_1) &\le C_1(\Omega,K,\Omega_1) (M-m) \\
\|h-M\|_{L^\infty(\Omega_2)} &\le C_2(\Omega,K,\Omega_1,\Omega_2) (M-m)
}
\end{cor}
\begin{proof}
By Lemma~\ref{lem:moll sh} we can assume that $u$ is smooth, although this is strictly speaking not necessary. The measure $\mu(dz) =\frac{1}{2\pi}\Delta u\, dxdy$ is unique and therefore $h$ harmonic on $\Omega_1$ if it satisfies~\eqref{eq:RR gen}.  
Let $\sup_K u = u(z_0)$, $z_0\in K$. 
By compactness, there exists $\delta>0$ and $N$ finite so that for any $z\in \Omega_1$ we can find  disks $\bbD(z_j,\delta)\subset\Omega$, $0\le j\le N$, with $z_N=z$, and $z_j\in \bbD(z_{j-1}, \delta/2)$ for all $j\ge1$.  Moreover, we may assume that $\Omega_2\subset \bigcup_{z\in\Omega_1} \bbD(z,\delta/2)$ and by compactness this can be chosen as a finite union. By Riesz's representation as in Theorem~\ref{thm:Riesz*} we have
\EQ{\label{eq:Dz0}
u(z) &= \int_{\bbD(z_0,\delta/2)} \log|z-\zeta|\,\mu(d\zeta)  + h_0(z) \qquad \forall\; z\in \bbD(z_0,\delta/2) \\
\mu(\bbD(z_0,\delta/2)) &\le C_0(\delta) (M-m), \quad \|h_0-M\|_{L^\infty(\bbD(z_0,\delta/4))}\le C_0(\delta) (M-m)
}
Next, apply Theorem~\ref{thm:Cartan} to the logarithmic potential in~\eqref{eq:Dz0} with $H=\delta/100$. Hence, there exists $w_1\in \bbD(z_0,\delta/4)\subset \bbD(z_1,3\delta/4)$ with 
\EQ{\label{eq:w1 good}
u(w_1)\ge m - C_1(\delta)(M-m)
}
while $u\le M$ on $\bbD(z_1,\delta)$. We now apply Riesz's representation as in Theorem~\ref{thm:Riesz*} on this disk, followed by Cartan to find a good point $w_2\in \bbD(z_2,3\delta/4)$ for which and analogue of~\eqref{eq:w1 good} holds. We may repeat this procedure to finitely many times to cover all of $\Omega_1$ by such disks leading to the stated upper bound on the measure $\mu(\Omega_1)$. For the estimate on the harmonic function $h$ defined by~\eqref{eq:RR gen}, pick any $z_*\in\Omega_2$. Then with $\eps_0:=\dist(\partial\Omega_1, \Omega_2)$ we have $\bbD(z_*,\eps_0)\subset \Omega_1$. On the one hand, for all $z\in\Omega_1$, 
\[
h(z) \ge u(z) -\log(\diam(\Omega_1))\mu(\Omega_1)\ge u(z) - C(M-m) 
\]
with the same type of constant as before. By the previous Cartan estimate and chaining argument, we can find $z_{**}\in \bbD(z_*,\eps_0/4)$ which satisfies a bound~\eqref{eq:w1 good} with a purely geometric constant. Hence 
\EQ{\label{eq:hlow}
h(z_{**}) \ge  m - C(M-m)
}
On the other hand, again by Theorem~\ref{thm:Cartan} we may find $\eps_1\in (3\eps_0/4, \eps_0)$ so that for all $|z-z_*|=\eps_1$ one has 
\[
\int_{\Omega_1}  \log |z-\zeta|\, \mu(d\zeta) \ge -C(M-m) 
\]
whence 
\EQ{\label{eq:high}
h(z) \le M + C(M-m) \qquad\forall\; |z-z_*|=\eps_1
}
By Harnack's inequality, \eqref{eq:hlow} and \eqref{eq:high} imply that $h$ satisfies the desired bound on $\bbD(z_*,\eps_0/2)$ and hence everywhere on~$\Omega_2$. 
\end{proof}

Alternatively, one can rely the proof strategy of Theorem~\ref{thm:Riesz}, and use the Green function on general subdomains of $\Omega$ with sufficiently regular boundary. But this seems technically more  involved, at least to the author.

\section{The Bourgain-Goldstein theorem}
\label{sec:BG} 

In this section we will sketch a proof of the main theorem in~\cite{BouG}. Similar to Theorem~\ref{thm:FSW} it addresses Anderson localization for the operators
\EQ{
(H_{x,\omega}\psi)_n &= \psi_{n-1} + \psi_{n+1} +  V(T_\omega^n x)\psi_n
\label{eq:1}
} 
on
$\ell^2(\bbZ)$, where $T_\omega:\tor\to\tor$ is the  rotation $x\mapsto x+\omega\mod 1$ and $V:\tor\to\R$ is analytic. 

\begin{thm}\label{thm:BG}
Suppose the Lyapunov exponents $L(E,\omega)$ associated with \eqref{eq:1} satisfy $\inf_{E,\omega} L(E,\omega)>0$. Then for almost every $\omega\in\tor$, the operator $H_{0,\omega}$ exhibits pure point spectrum with exponentially decaying eigenfunctions. Moreover, for almost every $\omega\in\tor$, the operator $H_{x,\omega}$ exhibits Anderson localization 
for almost every~$x\in\tor$.
\end{thm}

The final statement of the theorem follows simply by Fubini and the fact that one may replace $0$ in $H_{0,\omega}$ with any other $x\in\tor$. See~\cite{BouG,Bou} for versions of this theorem with $V$ analytic on higher-dimensional tori. This section is only meant to serve as a motivation for higher-dimensional techniques involving $\Delta_{\bbZ^d}$ with $d\ge2$, and less as a review of~\cite{BouG} itself. We will often drop $\omega$ from the notation and write $H_x$ or $H(x)$. 

No explicit Diophantine condition arises here in contrast to Theorem~\ref{thm:FSW}. In fact, it is not known if Theorem~\ref{thm:BG} holds for all Diophantine~$\omega$. For $V(x)=\cos(2\pi x)$, Jitomirskaya proved~\cite{Jit} that this is indeed the case. 
Although Diophantine conditions play a decisive role in the proof of Theorem~\ref{thm:BG}, one does remove a measure $0$ set of  ``bad"  $\omega$   in addition to a measure $0$ set of non-Diophantine $\omega$.  The smallness condition on $\eps$ in Section~\ref{sec:FSW} is replaced by positive Lyapunov exponents, a non-perturbative condition. No assumption on the number of monotonicity intervals of $V$ is made, nor do we impose an explicit nondegeneracy condition. Note, however, that the most degenerate case $V=\const$ cannot arise by positive Lyapunov exponents. By analyticity, $V$ therefore cannot be infinitely degenerate anywhere. No analogue of Theorem~\ref{thm:BG} is known if $V$ is merely smooth, nor is it clear what the results might be for smooth~$V$. 

We quickly review some elementary background on Lyapunov exponents. 
Consider~\eqref{eq:1} 
 $T:X\to X$ with an ergodic transformation on a probability space
$(X,\nu)$, and $V$ is a real-valued measurable function.  Define 
\EQ{
\label{eq:Ldef}
 L(E) &= \lim_{n\to\infty} \frac{1}{n} \int_{X}
\log\|M_n(x,E)\|\,\nu(dx) = \inf_{n\ge1} \frac{1}{n} \int_{X}
\log\|M_n(x,E)\|\,\nu(dx)
 }
 where $M_n$ are the transfer matrices
\EQ{\label{eq:Mn}
M_n(x,E) = \prod_{k=n}^1 \left[ \begin{matrix}  E-V(T^k x) & -1 \\
1 & 0 \end{matrix}\right]
}
of \eqref{eq:1}, i.e., the column vectors of $M_n$ are a
fundamental system of the equation $H_x \psi = E\psi$.
The limit in~\eqref{eq:Ldef} exists as stated due to fact that $a_n:=\int_{X}
\log\|M_n(x,E)\|\,\nu(dx)
$ is a subadditive sequence, and it is known that $\lim_{n\to\infty} \frac{1}{n} a_n = \inf_{n\ge1} \frac{1}{n} a_n$ exists for such sequences. Since $M_n\in SL(2,\bbR)$ we have $\| M_n\|\ge1$ and thus $L(E)\ge 0$. It is an important  and often difficult question to decide whether $L(E)>0$ for~\eqref{eq:1}, see~\cite{herman,HLS}  for an example of this. But this circle of problems will not concern us here.  
It was shown by F\"urstenberg
and Kesten~\cite{FurKes}, later generalized in Kingman's subadditive ergodic theorem,  that
\EQ{\label{eq:FK}
\lim_{n\to\infty} \frac{1}{n} \log\|M_n(x,E)\| = L(E)
}
for a.e.~$x\in X$. This does use ergodicity of $T$, whereas~\eqref{eq:Ldef} does not. See~Viana's book~\cite{Viana} for all this. 

The Thouless formula, see~\cite{CraigS}, 
\begin{equation}
\label{eq:thou}
 L(E) = \int_\R \log|E-E'|\, N(dE') \qquad \forall\; E\in\C
\end{equation}
relates the Lyapunov exponent to the density of states. Here $N$ is the integrated density of states (IDS), i.e., 
  the limiting 
distribution of the eigenvalues of~\eqref{eq:1} restricted to 
intervals $\Lambda=[-N,N]$  in the limit $N\to\infty$. In other words, there exists a deterministic nondecreasing function $N$ so that
for a.e.~$x\in X$ one has 
\[
|\Lambda|^{-1} |\{ j\in [1,|\Lambda|] \:|\: E^{(\Lambda)}_j(x)< t\}| \to N(t),
\]
where $E^{(\Lambda)}_j(x)$ are the eigenvalues of $H^\Lambda_{x}$, the restriction of \eqref{eq:1} to $\Lambda$ with Dirichlet boundary conditions. The existence of this limit holds in great generality, see~\cite{FigP}. The Lyapunov exponent is a subharmonic function on~$\C$, and harmonic on $\C\setminus\R$. The Thouless formula identifies the IDS $N$ as the Riesz measure of $L(E)$, and also shows that $L$ and $\frac{dN}{dE}$ are related to each other by the Hilbert transform. For far-reaching considerations involving these concepts see for example Avila's global work on phase transitions~\cite{Avila}.

\subsection{Large deviation theorems}

We now present a key ingredient in the proof of Theorem~\ref{thm:BG}, namely the large deviation estimates (LDTs), see also~\cite{GolS} where they are essential in the study of the regularity of the IDS.  
For the operators \eqref{eq:1} defined in terms of rotations of $\tor$, define 
\[ L_n(E) = \frac{1}{n} \int_{\tor}
\log\|M_n(x,E)\|\,dx.\]
The following LDT can be viewed as a quantitative form of~\eqref{eq:FK}. 

\begin{defn}
By Diophantine, we will now mean any irrational $\omega$ so that $\|n\omega\|\ge b\, n^{-a}$ for all $n\ge1$.   
\end{defn}

It is easy to see that for every $a>1$ a.e.~$\omega$ satisfies such a condition for some $b=b(\omega)$. 

\begin{prop} \label{prop:LDT}
For Diophantine $\omega$ there exist 
$0<\sigma,\tau<1$ depending on $V, a$ so that for all $E\in [-E_0,E_0]$, 
\EQ{
\label{eq:LDT1}
|\{x\in \tor \:| \:    | \log\|M_n(x,E)\| - nL_n(E)|
> n^{1-\sigma}\} | \le \exp(-n^\tau).
}
for all sufficiently large $n\ge n_0(V,a,b,E_0)$. 
\end{prop}

To motivate~\eqref{eq:LDT1}, consider the following scalar, or commutative,  model: 
\EQ{\label{eq:model}
u(x) = \sum_{k=1}^q \log|e(x)-e(k\omega)| 
}
where $\omega=\frac{p}{q}$ and $e(x) = e^{2\pi ix}$. Then
$u(x)=\log|e(xq)-1|$ and $\int_\tor u(x)\, dx=0$  so that for
$\lambda<0$
\EQ{
\label{eq:model ldt} 
|\{x\in\tor\::\: u(e(x))< \lambda\}|
= |\{x\in\tor\::\: |e(x)-1|< e^\lambda\}| 
}
which is of size $e^\lambda$. In this model case, $u(x+1/q)=u(x)$.
Returning to $u(x)=\log\|M_n(x,E)\|$, this exact invariance needs
to be replaced by the almost invariance
\begin{equation}\label{eq:ainv}
\sup_{x\in\tor} |u(x)-u(x+k\omega)|\le Ck \text{\ \ for any\ \ }k\ge1.
\end{equation}
The logarithm in our model case~\eqref{eq:model} is a reasonable choice because of
Riesz's representation theorem for subharmonic functions  applied to the function $u(z)=
\log\|M_n(z,E)\|$ which is subharmonic on a neighborhood of $[0,1]$ in~$\C$ by analyticity of~$V$. 
The subharmonicity can be seen by writing 
\[
u(z) = \sup_{\|\vec v\|=\|\vec w\|=1}\log |\langle M_n(z,E)\vec v,\vec w\rangle|
\]
First, $\log |\langle M_n(z,E)\vec v,\vec w\rangle|$ is subharmonic by analyticity of $\langle M_n(z,E)\vec v,\vec w\rangle$. Second, the sub-mean value property (smvp) survives under suprema, and so $u$ satisfies the smvp. Finally, the function $u(z)$ is clearly continuous. 

\begin{proof}[Proof of Proposition~\ref{prop:LDT} by Riesz and Cartan]
Fix a  rectangle $R$ which compactly contains $[0,1]$. 
By Riesz representation as stated in  Theorem~\ref{thm:Riesz}, there exists a positive measure $\mu$ on $R$ and a harmonic function on~$R$ such that 
\EQ{\label{eq:riesz u}
 u(z) &=\log\|M_n(z,E)\| = \int_R \log|z-\zeta|\, \mu(d\zeta) + h(z) 
 }
Since $\|M_n(z)\|\le e^{Cn}$, $0\le u(z)\les n$ on $R$ (with a constant that depends on $V$, $R$ and $E_0$)  and thus
$\|\mu\|\les n$  as well as $\|h\|_{L^\infty(R')}\les n$, where $[0,1]\Subset R'\Subset R$ is a slightly smaller
rectangle. Fix a small
$\delta>0$ and take $n$ large. Then there is a disk
$\bbD_0=\bbD(x_0,n^{-2\delta})$, $x_0\in [0,1]$ with the property that $\mu(\bbD_0)\les
n^{1-2\delta}$. Write
\[ 
\int_R \log|z-\zeta|\, \mu(d\zeta)  = u_1(z)+u_2(z) = \int_{\bbD_0} \log|z-\zeta|\, \mu(d\zeta) +
\int_{\bbC\setminus \bbD_0} \log|z-\zeta|\, \mu(d\zeta)
\]
Set $\bbD_1=\bbD(x_0,n^{-3\delta})$. Then
\[ |u_2(z)-u_2(z')|\les n^{1-\delta} \quad \forall z,z'\in \bbD_1\]
since
\[
|u_2(z)-u_2(z')| =\left| \int_{\bbC\setminus \bbD_0} \log\Big| 1 + \frac{z'-z}{z-\zeta}\Big |\, \mu(d\zeta)\right| \les \frac{n^{-3\delta}}{n^{-2\delta}}\mu(\bbC)\les n^{1-\delta}
\] 
Cartan's theorem applied to $u_1(z)$ yields disks
$\{\bbD(z_j,r_j)\}_j$ so that $\sum_j r_j \les \exp(-2n^\delta)$ and with the property that 
\[ u_1(z) \gtrsim -n^{1-\delta} \qquad \forall z\in\bbC\setminus
\bigcup_j \bbD(z_j,r_j) \] 
From  $u_1\le0$ on $\bbD_1$ and 
$|h(z)-h(z')|\les n|z-z'|$ on $R'$, it follows that
\EQ{
\label{eq:smalldevia} 
|u(z)-u(z')| &\les n^{1-\delta} \quad \forall z,z'\in \bbD_1\setminus
\bigcup_j \bbD(z_j,r_j)
}
From the Diophantine property with $1<a<\frac43$, say, for any $x,x'\in\tor$
there are positive integers $k,k'\les n^{4\delta}$ such that
\[ x+k\omega, x'+k'\omega \in \bbD_1 \quad \mod \bbZ\]
An elementary way of seeing this is to use Dirichlet's approximation principle, viz.~for any $Q>1$ there exists a reduced fraction $\frac{p}{q}$ so that $|\omega-p/q|\le (qQ)^{-1}$ and $1\le q<Q$. Then use the Diophantine property to bound $q$ from below in terms of~$Q$. 
In order to avoid the Cartan disks $\bigcup_j \bbD(z_j,r_j)$ we need
to remove a set $\cB\subset \tor$ of measure $\les
\exp(-n^{\delta})$. For this step is is important that Cartan controls the sum of the radii, i.e.,  $\sum_j r_j \les \exp(-2n^\delta)$
since then the disks remove at most measure $\les \exp(-2n^\delta)$ from the real line. 
Then from the almost
invariance~\eqref{eq:ainv}, for any $x,x'\in\tor\setminus\cB$,
\[ |u(x)-u(x')| \les n^{4\delta}+n^{1-\delta} \les n^{1-\delta}\]
This implies \eqref{eq:LDT1} with $\sigma=\tau=\delta$. 
\end{proof}

This proof generalizes to other types of dynamics such as higher-dimensional shifts 
$Tx=x+\omega \mod\bbZ^d$, on $\tor^d$ with $d\ge2$. 

\begin{defn}\label{def:cardef} 
Let $0<H<1$. For any subset $\bad\subset\Compl$ we define $\bad\in\Car_1(H)$ if $\bad\subset\bigcup_j \bbD(z_j,r_j)$ with
\EQ{\label{sumradii} \sum_j r_j\le C_0\,H.}
If $d$ is a positive integer greater than one and $\bad\subset\Compl^d$, then we define recursively  $\bad\in\Car_d(H)$ if there exists  $\bad_0\in\Car_{d-1}(H)$ so that
\[ \bad=\{(z_1,z_2,\ldots,z_d)\:|\:(z_2,\ldots,z_d)\in\bad_0\mbox{\ \ or\ }z_1\in\bad(z_2,\ldots,z_d)\mbox{\ \ with\ }\bad(z_2,\ldots,z_d)\in\Car_1(H)\}.\]
We refer to the sets in $\Car_d(H)$ for any $d$ and $H$ summarily as Cartan sets.
\end{defn}

The following theorem from \cite{GolS} furnishes they key property allowing one to extend  the previous proof of~\eqref{eq:LDT1} to higher-dimensional shifts.
We state the case $d=2$, with $d>2$ being similar (see also~\cite{Sch}).
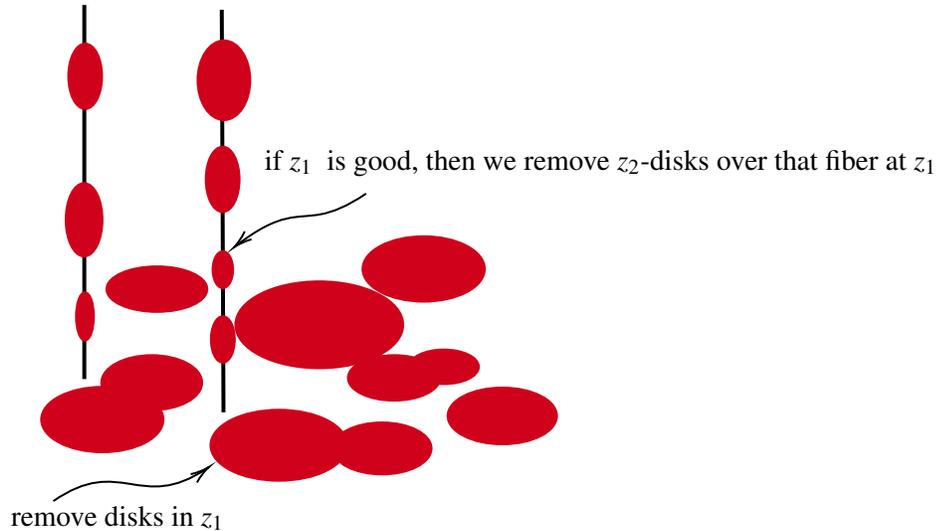
\begin{figure}[ht]
\tikzset{every picture/.style={line width=0.75pt}} 

\begin{tikzpicture}[x=0.65pt,y=0.6pt,yscale=-1,xscale=1]

\draw  [color={rgb, 255:red, 208; green, 2; blue, 27 }  ,draw opacity=1 ][fill={rgb, 255:red, 208; green, 2; blue, 27 }  ,fill opacity=1 ][line width=1.5]  (71,261) .. controls (71,249.95) and (86.67,241) .. (106,241) .. controls (125.33,241) and (141,249.95) .. (141,261) .. controls (141,272.05) and (125.33,281) .. (106,281) .. controls (86.67,281) and (71,272.05) .. (71,261) -- cycle ;
\draw  [color={rgb, 255:red, 208; green, 2; blue, 27 }  ,draw opacity=1 ][fill={rgb, 255:red, 208; green, 2; blue, 27 }  ,fill opacity=1 ][line width=1.5]  (106,237.72) .. controls (106,228.49) and (118.87,221) .. (134.76,221) .. controls (150.64,221) and (163.51,228.49) .. (163.51,237.72) .. controls (163.51,246.96) and (150.64,254.45) .. (134.76,254.45) .. controls (118.87,254.45) and (106,246.96) .. (106,237.72) -- cycle ;
\draw  [color={rgb, 255:red, 208; green, 2; blue, 27 }  ,draw opacity=1 ][fill={rgb, 255:red, 208; green, 2; blue, 27 }  ,fill opacity=1 ][line width=1.5]  (240.51,279.22) .. controls (240.51,270.51) and (253.16,263.45) .. (268.76,263.45) .. controls (284.36,263.45) and (297,270.51) .. (297,279.22) .. controls (297,287.94) and (284.36,295) .. (268.76,295) .. controls (253.16,295) and (240.51,287.94) .. (240.51,279.22) -- cycle ;
\draw  [color={rgb, 255:red, 208; green, 2; blue, 27 }  ,draw opacity=1 ][fill={rgb, 255:red, 208; green, 2; blue, 27 }  ,fill opacity=1 ][line width=1.5]  (169.51,277.22) .. controls (169.51,265.2) and (187.08,255.45) .. (208.76,255.45) .. controls (230.43,255.45) and (248,265.2) .. (248,277.22) .. controls (248,289.25) and (230.43,299) .. (208.76,299) .. controls (187.08,299) and (169.51,289.25) .. (169.51,277.22) -- cycle ;
\draw  [color={rgb, 255:red, 208; green, 2; blue, 27 }  ,draw opacity=1 ][fill={rgb, 255:red, 208; green, 2; blue, 27 }  ,fill opacity=1 ][line width=1.5]  (184,201.22) .. controls (184,186.44) and (205.61,174.45) .. (232.26,174.45) .. controls (258.91,174.45) and (280.51,186.44) .. (280.51,201.22) .. controls (280.51,216.01) and (258.91,228) .. (232.26,228) .. controls (205.61,228) and (184,216.01) .. (184,201.22) -- cycle ;
\draw  [color={rgb, 255:red, 208; green, 2; blue, 27 }  ,draw opacity=1 ][fill={rgb, 255:red, 208; green, 2; blue, 27 }  ,fill opacity=1 ][line width=1.5]  (285,227.72) .. controls (285,222.05) and (293.85,217.45) .. (304.76,217.45) .. controls (315.67,217.45) and (324.51,222.05) .. (324.51,227.72) .. controls (324.51,233.4) and (315.67,238) .. (304.76,238) .. controls (293.85,238) and (285,233.4) .. (285,227.72) -- cycle ;
\draw  [color={rgb, 255:red, 208; green, 2; blue, 27 }  ,draw opacity=1 ][fill={rgb, 255:red, 208; green, 2; blue, 27 }  ,fill opacity=1 ][line width=1.5]  (249.51,234.72) .. controls (249.51,227.14) and (261.26,221) .. (275.76,221) .. controls (290.25,221) and (302,227.14) .. (302,234.72) .. controls (302,242.3) and (290.25,248.45) .. (275.76,248.45) .. controls (261.26,248.45) and (249.51,242.3) .. (249.51,234.72) -- cycle ;
\draw  [color={rgb, 255:red, 208; green, 2; blue, 27 }  ,draw opacity=1 ][fill={rgb, 255:red, 208; green, 2; blue, 27 }  ,fill opacity=1 ][line width=1.5]  (307.51,258.72) .. controls (307.51,249.18) and (321.5,241.45) .. (338.76,241.45) .. controls (356.01,241.45) and (370,249.18) .. (370,258.72) .. controls (370,268.27) and (356.01,276) .. (338.76,276) .. controls (321.5,276) and (307.51,268.27) .. (307.51,258.72) -- cycle ;
\draw  [color={rgb, 255:red, 208; green, 2; blue, 27 }  ,draw opacity=1 ][fill={rgb, 255:red, 208; green, 2; blue, 27 }  ,fill opacity=1 ][line width=1.5]  (258,166) .. controls (258,154.95) and (273.67,146) .. (293,146) .. controls (312.33,146) and (328,154.95) .. (328,166) .. controls (328,177.05) and (312.33,186) .. (293,186) .. controls (273.67,186) and (258,177.05) .. (258,166) -- cycle ;
\draw  [color={rgb, 255:red, 208; green, 2; blue, 27 }  ,draw opacity=1 ][fill={rgb, 255:red, 208; green, 2; blue, 27 }  ,fill opacity=1 ][line width=1.5]  (109,178.72) .. controls (109,171.14) and (121.87,165) .. (137.76,165) .. controls (153.64,165) and (166.51,171.14) .. (166.51,178.72) .. controls (166.51,186.3) and (153.64,192.45) .. (137.76,192.45) .. controls (121.87,192.45) and (109,186.3) .. (109,178.72) -- cycle ;
\draw [line width=1.5]    (175.51,2.45) -- (176.51,256.45) ;
\draw  [color={rgb, 255:red, 208; green, 2; blue, 27 }  ,draw opacity=1 ][fill={rgb, 255:red, 208; green, 2; blue, 27 }  ,fill opacity=1 ][line width=1.5]  (162,46.95) .. controls (162,33.42) and (168.61,22.45) .. (176.76,22.45) .. controls (184.91,22.45) and (191.51,33.42) .. (191.51,46.95) .. controls (191.51,60.48) and (184.91,71.45) .. (176.76,71.45) .. controls (168.61,71.45) and (162,60.48) .. (162,46.95) -- cycle ;
\draw  [color={rgb, 255:red, 208; green, 2; blue, 27 }  ,draw opacity=1 ][fill={rgb, 255:red, 208; green, 2; blue, 27 }  ,fill opacity=1 ][line width=1.5]  (166.76,109.45) .. controls (166.76,98.4) and (170.9,89.45) .. (176.01,89.45) .. controls (181.13,89.45) and (185.27,98.4) .. (185.27,109.45) .. controls (185.27,120.5) and (181.13,129.45) .. (176.01,129.45) .. controls (170.9,129.45) and (166.76,120.5) .. (166.76,109.45) -- cycle ;
\draw    (77.51,312.45) .. controls (117.11,282.75) and (127.79,320.24) .. (166.81,291.88) ;
\draw [shift={(168,291)}, rotate = 503.13] [color={rgb, 255:red, 0; green, 0; blue, 0 }  ][line width=0.75]    (10.93,-3.29) .. controls (6.95,-1.4) and (3.31,-0.3) .. (0,0) .. controls (3.31,0.3) and (6.95,1.4) .. (10.93,3.29)   ;
\draw    (259.51,118.45) .. controls (247.95,126.89) and (239.99,130.05) .. (233.31,131.45) .. controls (218.82,134.48) and (210.27,129.24) .. (183.74,151.42) ;
\draw [shift={(182.51,152.45)}, rotate = 319.64] [color={rgb, 255:red, 0; green, 0; blue, 0 }  ][line width=0.75]    (10.93,-3.29) .. controls (6.95,-1.4) and (3.31,-0.3) .. (0,0) .. controls (3.31,0.3) and (6.95,1.4) .. (10.93,3.29)   ;
\draw  [color={rgb, 255:red, 208; green, 2; blue, 27 }  ,draw opacity=1 ][fill={rgb, 255:red, 208; green, 2; blue, 27 }  ,fill opacity=1 ][line width=1.5]  (169.76,210.45) .. controls (169.76,202.72) and (172.61,196.45) .. (176.14,196.45) .. controls (179.66,196.45) and (182.51,202.72) .. (182.51,210.45) .. controls (182.51,218.18) and (179.66,224.45) .. (176.14,224.45) .. controls (172.61,224.45) and (169.76,218.18) .. (169.76,210.45) -- cycle ;
\draw  [color={rgb, 255:red, 208; green, 2; blue, 27 }  ,draw opacity=1 ][fill={rgb, 255:red, 208; green, 2; blue, 27 }  ,fill opacity=1 ][line width=1.5]  (170.76,166.45) .. controls (170.76,160.37) and (173.16,155.45) .. (176.14,155.45) .. controls (179.11,155.45) and (181.51,160.37) .. (181.51,166.45) .. controls (181.51,172.52) and (179.11,177.45) .. (176.14,177.45) .. controls (173.16,177.45) and (170.76,172.52) .. (170.76,166.45) -- cycle ;
\draw [line width=1.5]    (95.51,-0.55) -- (95.51,235.45) ;
\draw  [color={rgb, 255:red, 208; green, 2; blue, 27 }  ,draw opacity=1 ][fill={rgb, 255:red, 208; green, 2; blue, 27 }  ,fill opacity=1 ][line width=1.5]  (86.76,44.45) .. controls (86.76,33.4) and (90.9,24.45) .. (96.01,24.45) .. controls (101.13,24.45) and (105.27,33.4) .. (105.27,44.45) .. controls (105.27,55.5) and (101.13,64.45) .. (96.01,64.45) .. controls (90.9,64.45) and (86.76,55.5) .. (86.76,44.45) -- cycle ;
\draw  [color={rgb, 255:red, 208; green, 2; blue, 27 }  ,draw opacity=1 ][fill={rgb, 255:red, 208; green, 2; blue, 27 }  ,fill opacity=1 ][line width=1.5]  (85.26,134.95) .. controls (85.26,122.52) and (89.79,112.45) .. (95.39,112.45) .. controls (100.98,112.45) and (105.51,122.52) .. (105.51,134.95) .. controls (105.51,147.38) and (100.98,157.45) .. (95.39,157.45) .. controls (89.79,157.45) and (85.26,147.38) .. (85.26,134.95) -- cycle ;
\draw  [color={rgb, 255:red, 208; green, 2; blue, 27 }  ,draw opacity=1 ][fill={rgb, 255:red, 208; green, 2; blue, 27 }  ,fill opacity=1 ][line width=1.5]  (91.26,195.95) .. controls (91.26,187.94) and (93.33,181.45) .. (95.89,181.45) .. controls (98.44,181.45) and (100.51,187.94) .. (100.51,195.95) .. controls (100.51,203.96) and (98.44,210.45) .. (95.89,210.45) .. controls (93.33,210.45) and (91.26,203.96) .. (91.26,195.95) -- cycle ;

\draw (115,323) node    {remove\ disks\ in\ $z_{1}$};
\draw (396,99) node    {if\ $z_{1}$ \ is\ good,\ then\ we\ remove\ $z_2$-disks\ over\ that\ fiber at $z_1$};

\end{tikzpicture}
\caption{Cartan-$2$ sets in $\bbC^2$}\label{fig:Cartan}
\end{figure}
\begin{thm}
\label{thm:Cartan2} 
Suppose $u$ is  continuous  on $\bbD(0,2)\times \bbD(0,2)\subset\Compl^2$ with $|u|\le 1$. Assume further that
\[\left\{  \begin{array}{ccl} z_1\mapsto u(z_1,z_2) & \text{is subharmonic for each} & z_2\in \bbD(0,2)\\
                              z_2\mapsto u(z_1,z_2) & \text{is subharmonic for each} & z_1\in \bbD(0,2).
           \end{array}
  \right.
\]
Fix some $\gamma\in(0,1/2)$. Given $r\in(0,1)$ there exists a polydisk $\Pi=\bbD(x_1,r^{1-\gamma})\times \bbD(x_2,r)\subset \bbD(0,1)\times \bbD(0,1)$ with $x_1,x_2\in[-1,1]$ and a set $\bad\in \Car_2(H)$ so that
\EQ{  |u(z_1,z_2)-u(z_1',z_2')| &<C_\gamma\,r^{1-2\gamma}\log\frac{1}{r}\mbox{\ \ for all\ \ }(z_1,z_2),(z_1',z_2')\in\Pi\setminus\bad. \\
    H &=\exp\Bigl(-r^{-\gamma}\Bigr).\label{Heta}
}
\end{thm}

This theorem  replaces \eqref{eq:smalldevia} in the
previous proof. For the sake of completeness, we now also sketch a proof by Fourier series as in~\cite{Bou, BouG}. 

\begin{proof}[Proof of Proposition~\ref{prop:LDT} by  Fourier series]
For this technique, it is more convenient to view $u(z)$ as a subharmonic function on an annulus around $|z|=1$. 
This is based on viewing the periodic analytic potential $V(x)$ as an analytic  function of $z=e(x)=e^{2\pi ix}$ instead and then
extending analytically to the annulus $\calA:=\{ z\in\bbC\:|\: 1-\delta<|z|<1+\delta\}$ for some $0<\delta<1$. Thus, write $V(x)=W(e(x))$ with $W$ analytic on that annulus. 
Accordingly, $u(x)=w(e(x))$, and 
 the Riesz representation takes the form
\EQ{\label{eq:riesz u2}
 w(z) &= \int_{\calA} \log|z-\zeta|\, \mu(d\zeta) + h(z) \quad\forall\; z\in\calA
 }
with $\mu$ a positive measure on $\calA$ with $\mu(\calA)\les n$  and $\|h\|_{L^\infty(\calA')}\les n$ for a slightly thinner annulus $\calA'$. Note that $u\ge0$ on $|z|=1$. 
In particular, 
\[
u(x) =  \int_{\calA} \log |e(x)-\zeta|\, \mu(d\zeta) + h(e(x))  \quad\forall \; x\in\tor
\]
Next, we claim  that 
\EQ{\label{eq:FT fzeta}
f_\zeta(x) := \log |e(x)-\zeta|\text{\ \ satisfies\ \ } \sup_{\zeta\in\bbC} |\widehat{f_\zeta}(k)|\le C|k|^{-1}\quad\forall\; k\ne0
}
with an absolute constant. 
First, it suffices to prove this $|\zeta|\le1$ by pulling out $\log|\zeta|$ otherwise. By translation in $x$ we may further assume that $0\le \zeta\le1$.  One checks that
\[
\partial_x \log|e(x)-1| = \pi \cot(\pi x), \; \partial_x \log|e(x)-r| = \pi \frac{2r\sin(2\pi x)}{1+r^2-2r\cos(2\pi x)} 
\] 
the latter for $0\le r<1$. These are, respectively, the kernel of the Hilbert transform on $\tor$ and the conjugate Poisson kernel. Both have uniformly bounded Fourier coefficients, uniformly in~$0\le r\le1$, whence our claim~\eqref{eq:FT fzeta}. 

We conclude that $|\hat{u}(k)|\le Cn|k|^{-1}$ for all $k\ne0$ by integrating over the Riesz mass. For the harmonic function we simply use that $|\partial_x h(e(x))|\les n$ and the decay of the Fourier coefficients follows. 
By the almost invariance property~\eqref{eq:ainv}, 
\[ u(x) -\la u\ra = \frac{1}{k} \sum_{j=1}^k u(x+j\omega)-\la u\ra + O(k) =
 \sum_{\nu\ne0} \hat{u}(\nu) e(x\nu) \frac{1}{k}\sum_{j=1}^k e(j\nu\omega) + O(k)
\]
Then one has that
\[ \Bigl| \frac{1}{k}\sum_{j=1}^k e(j\nu\omega) \Bigr| \les
\min(1,k^{-1}\|\nu\omega\|^{-1})
\]
for all $\nu\ge1$. Also, it follows from \eqref{eq:riesz u} that
$|\hat{u}(\nu)|\les n|\nu|^{-1}$ which in turn implies that
\[
 |u(x)-\la u\ra| \les \frac{1}{k}\sum_{j=1}^k \Bigl|\sum_{|\nu|> K}
 \hat{u}(\nu)e(\nu (x+k\omega))\Bigr| +
\sum_{0<|\nu|\le K} n|\nu|^{-1}
\min(1,k^{-1}\|\nu\omega\|^{-1})
\]
On the one hand, by Plancherel and the decay of the Fourier coefficients, 
\[ \Bigl\| \frac{1}{k}\sum_{j=1}^k \Bigl|\sum_{|\nu|> K}
 \hat{u}(\nu)e(\nu (x+j\omega))\Bigr|\;\Bigr\|_{L^2_x} \le \Big\| \sum_{|\nu|> K}
 \hat{u}(\nu)e(\nu x) \Big\|_2 \les  n\,K^{-1/2}
\]
On the other hand, setting $K=e^{n^\tau}$ it follows from the Diophantine condition (with $a=2$ for simplicity) that
\EQ{\label{eq:sumK}
\sum_{0<|\nu|\le K} n|\nu|^{-1} \min(1,k^{-1}\|\nu\omega\|^{-1})
\les nk^{-\frac12} \log K \les n^{1+\tau} k^{-\frac12}
}
Choosing $\tau>0$ small and $k=n^{\frac12}$, say,
yields~\eqref{eq:LDT1}. To prove~\eqref{eq:sumK}, partition $0<|\nu|\le K$ into sets corresponding to the size of~$\| \nu\omega\|$. First $\|\nu\omega\|\le k^{-1}$ and then $\|\nu\omega\|\in k^{-1}(2^{j-1},2^{j}]$ for $j\ge1$ and $k^{-1} 2^j <1$.  
The Diophantine condition implies that the recurrences into these sets cannot be more frequent than specific arithmetic conditions, which the reader can easily check. The $\log K$ term results from summing the harmonic series over a finite arithmetic progression. 
\end{proof}

\begin{remark}
\label{rem:Dioph}
Write the  Diophantine condition in the form  $\|k\omega\|\ge h(k)$ for all $k\ge1$. Later we will need to exploit the fact that the previous proofs require this 
condition only in the range  $1\le k\le n$. 
\end{remark}

For applications related to the study of fine properties of the IDS
it turns out to be important to obtain sharp versions of~\eqref{eq:LDT1}.
The commutative model example suggests that the optimal relation is $0\le 1-\sigma=\tau\le 1$. Here $\sigma=1-\tau=0$ corresponds to the largest possible deviations and smallest measures. The previous two proofs do not easily yield such a statement, but it was proved in~\cite{GolS} by a more involved argument. The book~\cite{Bou} contains an elegant Fourier series proof, see Theorem~5.1 on page~25. Both these references require stronger Diophantine conditions. 

There is a close connection between the Wegner estimate in Section~\ref{sec:AL FS} and the LDT from above. We refer to reader to~\cite[Lemma 5.5]{GolS3} for the precise formulation of a Wegner estimate derived via LDT for the quasi-periodic model~\eqref{eq:1}. 

\subsection{LDT and regular Green functions} 

As in Section~\ref{sec:AL FS} and~\ref{sec:FSW} the key to proving localization in Theorem~\ref{thm:BG} is to exclude arbitrarily long chains of resonances (absence of infinite tunneling). In fact, one shows that one cannot have double resonances on sufficiently long scales, in exact analogy with the localization results we proved above. The LDT theorems from above enter into this analysis through the Green function associated with~\eqref{eq:1} on finite intervals. In fact, from Cramer's rule for any $\Lambda = [a,b]\in\Z$, and $a\le j\le k\le b$, 
\EQ{\label{eq:GE}
(H_\Lambda(x)-E)^{-1}(j,k) = \frac{\det (H_{[a,j-1]}(x)-E) \det(H_{[k,b]}(x)-E)}{\det(H_\Lambda(x)-E)}
}
for fixed $x,\omega$, the latter Diophantine. We denote $f_n(x,E)=\det(H_{[1,n]}(x)-E)$. In explicit form, the matrix is 
\EQ{ \label{eq:fn}
H_{[1,n]}(x)-E = \left[
\begin{array}{ccccccccc}
                         v_1(x)-E & 1 & 0 & 0 & . & . & . & . &  0    \\
                        1 &  v_{2}(x)-E & 1 & 0 & 0 & . & . & . &  0 \\
                        0 & 1 &  v_{3}(x)-E & 1 & 0 & 0 &  . & . & 0 \\
                        . & . & . & . & . & . & . & . & . \\
                        . & . & . & . & . & . & . & . & . \\
                        . & . & . & . & . & . & . & . & . \\
                        . & . & . & . & . & . & . & . & 1 \\
                        0 & 0 & . & . & . & . & . &  1 &  v_n(x)-E
\end{array} \right]
}
with $v_j(x)=V(T^j  x)$ and $T=T_\omega$. Thus, \eqref{eq:GE} implies that 
\EQ{
\label{eq:GE2}
G_{[1,n]}(x,E)(j,k)=(H_{[1,n]}(x)-E)^{-1}(j,k) = \frac{f_{j-1}(x,E) f_{n-k}(T^k x,E)}{f_n(x,E)},\quad 1\le j\le k\le n
}
with the convention $f_{0}=1$. The transfer matrices defined in \eqref{eq:Mn} satisfy for all $n\ge 1$	
\EQ{\label{eq:Mnfn}
M_n(x,E) &= \left[\begin{array}{cc} (-1)^n f_n(x,E) & (-1)^n f_{n-1}(Tx,E) \\
                                 (-1)^{n-1} f_{n-1}(x,E) &  (-1)^{n-1} f_{n-2}(Tx,E) \\
                                     \end{array} \right]
}
where we   set $f_{-1}=0$.  The following uniform upper bound from~\cite[Proposition 4.3]{GolS2} improves on the LDT. As expected, as a subharmonic function $\log\|M_n(x,E)\|$ can only have large deviations towards values which are much smaller than $nL_n(E)$ but cannot exhibit deviations in the opposite direction. The following inequality requires positive Lyapunov exponents and relies on some machinery which we have not discussed here, such as the avalanche principle from~\cite{GolS}. Moreover, \cite{GolS2} imposes a Diophantine condition of the form
\[
\|n\omega\|\ge \frac{b}{n(\log n)^2}, \quad n\ge 2
\]
which holds for some $b>0$ for a.e.~$\omega$. Of course one needs $V$ analytic since the following lemma heavily relies on subharmonic functions and the LDT from above. 

\begin{lem}
\label{lem:UUB}
 Assume $L(E) \ge\gamma > 0$ for all $E\in I$, some interval. For all  $n \ge 1$ one has 
 \[
\sup_{x\in\tor}\log \|M_n(x,E)\| \le n L_n(E)+C(\log n)^B,
\]
for some absolute constant $B$ and $C=C(V,\gamma,b,I)$. 
\end{lem}

In view of 	\eqref{eq:Mnfn} and the Thouless formula~\eqref{eq:thou}  it is natural to ask if each entry of $M_n$, i.e., the determinants $f_n$ satisfy an LDT individually. This was proven to hold in~\cite[Section2]{GolS2}. 

\begin{prop}
\label{prop:LDT fn}
There exists $\sigma>0$ so that for large $n$
\[
|\{x\in\tor\:|\: \log|f_n(x,E)|< nL_n(E)-n^\sigma\}|\le e^{-n^\sigma}
\]
\end{prop}

A stronger statement is possible if we assume positive Lyapunov exponents. See~\cite[Lemma 5.1]{GolS3}. 

\begin{prop}
\label{prop:bmo}
Assume $L(E) \ge\gamma > 0$ for all $E\in I$. 
For some constants $A$ and $C$ 
depending on $\omega$, $V$,  and $\gamma$, every $n\ge1$ satisfies
\EQ{\label{eq:bmo}
 \Big| \int_0^1 \log |\det (H_{[1,n]} (x) - E) |\, dx
- n \, L_n (E) \Big|  & \le C \\
 \|  \log | \det (H_{[1,n]} (x) - E) |\, \|_\BMO & \le
C(\log n)^A.
}
Thus,  
\EQ{
\label{eq:det LDT}
 | \{  x\in \tor \: |\:
|\log | \det (H_{[1,n]} (x) - E) | - n\, L_n (E) |
>  H \} | &
\le C\exp \Big( - \frac{H}{(\log n )^A} \Big)
}
for any $H>(\log n)^A$. If $V$ is a trigonometric polynomial, then  the set on the left-hand side is covered by  $2\deg(V)n$ many intervals
each not exceeding in length the measure bound of~\eqref{eq:det LDT}. 
\end{prop}

The final statement follows from the fact that $z^{dn}\det (H_{[1,n]} (x) - E)$ with $z=e(x)$,  is a polynomial of degree $2dn$. 
The estimate \eqref{eq:det LDT} follows from the $\BMO$ bound~\eqref{eq:bmo}  by means of the classical John-Nirenberg inequality.  The large deviation estimate for the determinants $f_n(x,E)$ do not appear in the original proof of Theorem~\ref{thm:BG}, and they were established later in~\cite{GolS2}. However, they help to   streamline some of the technical aspects of~\cite{BouG}. For example,
in view of~\eqref{eq:GE2}, \eqref{eq:Mnfn} and Lemma~\ref{lem:UUB}, the Green function satisfies for large $n$ (and of course for positive Lyapunov exponents)
\EQ{\label{eq:GEbd}
|G_{[1,n]}(x,E)(j,k)| &\le \exp\big( (j-1)L_{j-1}(E) + (n-k) L_{n-k}(E) -n L_n(E)  + (\log n)^{2A} \big) \text{\ \ provided}\\
|f_n(x,E)| &\ge  nL_n(E)- (\log n)^{2A}
}
which therefore holds up to a set of measure $\les \exp(-(\log n)^A)$ (assuming $A\ge B$). It was proved in \cite{GolS} that $L_n(E)-L(E)\le Cn^{-1}$ whence it follows from~\eqref{eq:GEbd} that
\EQ{\label{eq:GEbd*}
|G_{[1,n]}(x,E)(j,k)| &\le \exp\big(-|j-k|L(E)  + (\log n)^{2A} \big) 
}
up to a set of measure $\les \exp(-(\log n)^A)$. By the preceding this set can be made $\les \exp(-n^\sigma)$ with $0<\sigma<1$ if we settle for the weaker Green function bound
\EQ{\label{eq:GEbd**}
|G_{[1,n]}(x,E)(j,k)| &\le \exp\big(-|j-k|L(E)  + n^\sigma(\log n)^{A} \big)  \quad\forall\; j,k\in [1,n]
}
and large $n$. This is precisely the notion of regular Green functions from Section~\ref{sec:AL FS}. Since 
\[
\dist(\spec(H_{[1,n]}(x)), E) = \| G_{[1,n]}(x,E)\|^{-1}
\]
the connection with a Wegner-type estimate is also immediately apparent. For example, from \eqref{eq:GEbd**} one concludes the following statement. We assume throughout that 
\EQ{\label{eq:E range}
|E|\le 2+\|V\|_\infty
} since this range contains $\spec(H_x)$. 

\begin{cor}
\label{cor:qp Wegner}
Under the same assumptions as Proposition~\ref{prop:bmo} one has 
\EQ{\label{eq:Weg qp}
 | \{  x\in \tor \: |\: \dist(\spec(H_{[1,n]}(x)), E)<\exp\big(-n^{1/3}\big)
 \} | &
\le e^{-n^{\frac14}}
}
for large $n$. In addition, the set on the left-hand side is contained in $O(n)$ many intervals assuming $V$ is a trigonometric polynomial. 
\end{cor}

\subsection{Eliminating double resonances} 

We will assume for convenience that $V$ is a trigonometric polynomial. 
As in the proof of localization in Sections~\ref{sec:AL FS} and~\ref{sec:FSW} we begin from a generalized (nonzero) eigenfunction $H(x)=E\psi(x)$ which by Theorem~\ref{thm:Ber} grows at
most linearly: $|\psi(n)|\le C(1+|n|)$. We claim that for any $n$ sufficiently large there exists a window $\Lambda_0=[-m,m]$ with $n\le m\le  n^3$ such that $H_{\Lambda_0}(0)$ is resonant with~$E$. Quantitatively, we claim  
\EQ{\label{eq:Lambda0 res}
\dist(\spec(H_{[-m,m]}(0)), E)\le e^{-m^{1/4}}
}
Indeed, denote  the set in Proposition~\ref{prop:LDT fn} by $\calB_n$. It consists of $O(n)$ intervals of length $e^{-n^\sigma}$. Therefore, by the Diophantine condition the set 
\EQ{\nn 
\{ n\le m\le  n^2 \:|\:   m\omega\in \calB_n(E)\cup(-\calB_n(E))\} 
}
has cardinality $O(n)$. Pick an $m\in [n,n^2]$ which is not in this set. Then by \eqref{eq:GEbd**}
\[
|G_{[m-n,m+n]}(0,E)(m+1,m\pm n)|+|G_{[-m-n,-m+n]}(0,E)(-m-1,-m\pm n)|\le \exp(-\gamma n/2),\quad 
\]
whence by \eqref{eq:Poisson} for large $n$
\[
\sqrt{\psi(m+1)^2+\psi(-m-1)^2} \le 2C\exp(-\gamma n/2)(1+n^2) \le \exp(-\gamma n/3)
\]
Combined with $(H_{[-m,m]}(0))-E) \psi = \psi(m+1)\delta_{m+1}+  \psi(-m-1)\delta_{-(m+1)}$ this estimate implies that
\[
\| (H_{[-m,m]}(0)-E) \psi \| \le \exp(-\gamma n/3)\le e^{-m^{1/4}}
\]
which is what we claimed in~\eqref{eq:Lambda0 res}. Let us denote by $\Dio_n(b)$ the Diophantine condition
\EQ{\label{eq:Dionb}
\| k\omega\|\ge \frac{b}{k(1+\log k)^2} \quad\forall \; 1\le k\le n
}
and $\Dio(b)=\bigcap_{n=1}^\infty \Dio_n(b)$. 
Then under this condition we have the following stronger LDT for the determinants, see \cite[Corollary 2.15]{GSV}: 

\begin{lem}
\label{lem:det ldt stable}
Assume $\omega\in \Dio_n(b)$ and positive Lyapunov exponents as above. 
For any $E_0$ in the range~\eqref{eq:E range}, 
\EQ{\label{eq:det stable ldt}
|\{ x\in\tor\:|\: \log|f_n(x,\omega,E)| < nL_n(E,\omega) - n^{\frac12} \text{\ for some\ }|E-E_0|\le e^{-n} \}|\le  e^{-n^{\frac13}}
}
if $n\ge n_0(V,b,\gamma)$ is large. 
\end{lem}
\begin{proof}
For fixed $E_0$ we already stated this LDT for the determinant in Proposition~\ref{prop:bmo}. The stability in $E$ over the exponentially small interval $[E_0-e^{-n}, E_0+e^{-n}]$ is precisely what  \cite[Corollary 2.15]{GSV} provides. The statement in loc.\ cit.\ is slightly weaker, but replacing the upper bound of~\cite[Corollary 2.14]{GSV} with the stronger one of Lemma~\ref{lem:UUB} implies~\eqref{eq:det stable ldt}. 
\end{proof}

In view of this lemma, and \eqref{eq:Lambda0 res} we now introduce the following set which will allow us to eliminate double resonances: for any $b>0$
\EQ{\label{eq:Sn def}
\calS_n(b) &:= \big \{(\omega,x)\in \Dio_n(b)\times\tor \:|\: \exists E\in\R\text{\ with\ } \dist(\spec(H_{[-n,n]}(0,\omega)),E)\le e^{-n^{1/4}}, \text{\ and} \\
& \qquad \log|f_m(x,\omega,E)|\le mL_m(E,\omega) - m^{1/2}\text{\ for some\ } m\in [n^{1/4}/2, n^{1/4}] \big\} 
}
If $\dist(\spec(H_{[-n,n]}(0,\omega)),E)\le e^{-n^{1/4}}$, then $|E-E_{j,n}(\omega)|\le   e^{-n^{1/4}}$ for some eigenvalue $E_{j,n}(\omega)$ of~$H_{[-n,n]}(0,\omega)$. Applying Lemma~\ref{lem:det ldt stable} with $E_0=E_{j,n}(\omega)$ and summing over $1\le j\le 2n+1$ one concludes by Fubini that 
\EQ{\label{eq:Sn bd}
|\calS_n(b)|\le 3 n^{\frac54} e^{-n^{1/12}}.
} 
The set of bad $\omega$, which we will need to exclude in order to prevent double resonances,  is 
\EQ{\label{eq:calBn}
\calB_n(b):=\{ \omega \in \tor \:|\: (\omega, \ell \omega)\in \calS_n(b) \text{\ for some\ }\pm \ell \in [n^{2s},2n^{2s}]\}
}
Here $s\ge 2$ is an absolute constant, which we will specify later.  
The following {\em lemma on steep lines} from~\cite{BouG} guarantees that $\calB_n(b)$ has very small measure. This hinges not only on the small measure estimate of~\eqref{eq:Sn bd}, which by itself is insufficient,  but also on the structure of the set~$\calS_n(b)$. Specifically, the fact that the horizontal slices
\EQ{
\label{eq:cut set}
(\calS_n(b))_x:=\{\omega\in\tor\:|\: (\omega,x)\in\calS_n(b)\}
}
are contained in no more than $O(n^{s})$ many intervals of very small measure.  
\begin{figure}[ht]
\tikzset{every picture/.style={line width=0.75pt}} 

\begin{tikzpicture}[x=0.55pt,y=0.55pt,yscale=-1,xscale=1]

\draw   (100,50.95) -- (509.51,50.95) -- (509.51,460.46) -- (100,460.46) -- cycle ;
\draw [color={rgb, 255:red, 208; green, 2; blue, 27 }  ,draw opacity=1 ][line width=1.5]    (141.51,51.95) -- (100,460.46) ;
\draw [color={rgb, 255:red, 208; green, 2; blue, 27 }  ,draw opacity=1 ][line width=1.5]    (181.51,51.95) -- (140,460.46) ;
\draw [color={rgb, 255:red, 208; green, 2; blue, 27 }  ,draw opacity=1 ][line width=1.5]    (341.51,50.95) -- (300,459.46) ;
\draw [color={rgb, 255:red, 208; green, 2; blue, 27 }  ,draw opacity=1 ][line width=1.5]    (300.51,51.95) -- (259,460.46) ;
\draw [color={rgb, 255:red, 208; green, 2; blue, 27 }  ,draw opacity=1 ][line width=1.5]    (221.51,50.95) -- (180,459.46) ;
\draw [color={rgb, 255:red, 208; green, 2; blue, 27 }  ,draw opacity=1 ][line width=1.5]    (259.51,51.47) -- (218,459.98) ;
\draw [color={rgb, 255:red, 208; green, 2; blue, 27 }  ,draw opacity=1 ][line width=1.5]    (382.51,50.95) -- (341,459.46) ;
\draw [color={rgb, 255:red, 208; green, 2; blue, 27 }  ,draw opacity=1 ][line width=1.5]    (423.51,50.95) -- (382,459.46) ;
\draw [color={rgb, 255:red, 208; green, 2; blue, 27 }  ,draw opacity=1 ][line width=1.5]    (465.51,51.47) -- (424,459.98) ;
\draw [color={rgb, 255:red, 208; green, 2; blue, 27 }  ,draw opacity=1 ][line width=1.5]    (509.51,50.95) -- (468,459.46) ;
\draw    (53,269.95) .. controls (92.6,240.25) and (112.6,298.76) .. (151.81,270.82) ;
\draw [shift={(153,269.95)}, rotate = 503.13] [color={rgb, 255:red, 0; green, 0; blue, 0 }  ][line width=0.75]    (10.93,-3.29) .. controls (6.95,-1.4) and (3.31,-0.3) .. (0,0) .. controls (3.31,0.3) and (6.95,1.4) .. (10.93,3.29)   ;
\draw  [fill={rgb, 255:red, 74; green, 144; blue, 226 }  ,fill opacity=1 ] (180.07,295.21) -- (288.07,57.8) -- (291.21,59.23) -- (183.22,296.64) -- cycle ;
\draw  [fill={rgb, 255:red, 74; green, 144; blue, 226 }  ,fill opacity=1 ] (262.73,321.71) .. controls (245.61,306) and (269.06,252.57) .. (315.12,202.38) .. controls (361.17,152.19) and (412.39,124.24) .. (429.51,139.95) .. controls (409.16,131.2) and (361.78,158.89) .. (318.78,205.75) .. controls (275.79,252.6) and (252.26,302.18) .. (262.73,321.71) -- cycle ;
\draw  [fill={rgb, 255:red, 74; green, 144; blue, 226 }  ,fill opacity=1 ] (187.83,93.47) .. controls (188.25,93.07) and (188.92,93.09) .. (189.32,93.52) -- (409.04,327.24) .. controls (409.44,327.67) and (409.42,328.34) .. (408.99,328.74) -- (406.68,330.91) .. controls (406.25,331.31) and (405.58,331.29) .. (405.18,330.87) -- (185.47,97.14) .. controls (185.07,96.72) and (185.09,96.05) .. (185.52,95.65) -- cycle ;
\draw    (62,166.95) .. controls (107.06,147.15) and (180.54,163.62) .. (222.26,146.48) ;
\draw [shift={(223.51,145.95)}, rotate = 516.54] [color={rgb, 255:red, 0; green, 0; blue, 0 }  ][line width=0.75]    (10.93,-3.29) .. controls (6.95,-1.4) and (3.31,-0.3) .. (0,0) .. controls (3.31,0.3) and (6.95,1.4) .. (10.93,3.29)   ;

\draw (92,466.35) node [anchor=north west][inner sep=0.75pt]    {$0$};
\draw (511.51,463.86) node [anchor=north west][inner sep=0.75pt]    {$1$};
\draw (14,275.35) node [anchor=north west][inner sep=0.75pt]    {$( \omega ,\ell \omega )$};
\draw (41,162.35) node [anchor=north west][inner sep=0.75pt]  [font=\large]  {$\mathcal{S}$};

\end{tikzpicture}
\caption{The lemma on steep lines}\label{fig:steep}
\end{figure}
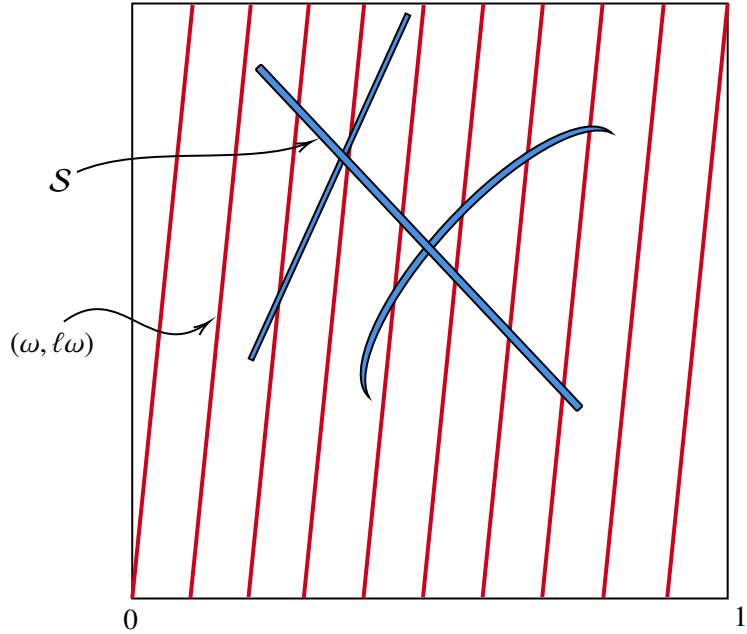
\begin{lem}
\label{lem:steep lines}
Suppose the Borel set $\calS\subset\tor^2$ has the property that for every $x\in\tor$ the horizontal slice $\calS_x$, viewed as a subset of $[0,1]$,  consists of no more than $M$ intervals. Then 
\EQ{
\label{eq:steep lines}
|\{ \omega\in\tor\:|\: (\omega,\ell\omega)\in \calS\! \mod\Z^2\text{\ for some\ }\ell \in [N,2N] \} |\le \frac{M}{N} +8N^{\frac52} |\calS|^{\frac12}
}
\end{lem}
\begin{proof}
By Fubini, for each $\gamma>0$, 
\[
|\{x\in\tor\:|\: |\calS_x|>\gamma\}|\le  |\calS|\gamma^{-1}
\]
We define the set of good $x\in\tor$ as
 \EQ{\label{eq:x good}
 \calG &:=\{ x\in\tor\:|\: |\calS_x|\le \gamma  \text{\ and for all \ }j\in [1,N]\text{\ one has\ }  \| x j\| > 4N^2\gamma \}
 }
 with $\|\cdot\|$ the norm of~$\tor$. Then 
 \[
 |\tor\setminus\calG|\le   |\calS|\gamma^{-1} + 4N^3 \gamma 
  \]
  We optimize here by setting $\gamma = (4N^3)^{-\frac12}|\calS|^{\frac12}$ whence
  \EQ{\label{eq:Gc est}
 |\tor\setminus\calG|\le   4 N^{\frac32}  |\calS|^{\frac12}
  }
Correspondingly, 
\[
\calS = \calS_*\cup \calS_{**} := (\calS\cap \calG) \cup  (\calS\cap (\tor\setminus\calG))
\]
We eliminate $\calS_{**}$ as follows:
\EQ{
\label{eq:steep lines**}
& |\{ \omega\in\tor\:|\: (\omega,\ell\omega)\in \calS_{**}\! \mod\Z^2\text{\ for some\ }\ell \in [N,2N] \} | \\
&\le \sum_{\ell=N}^{2N}  |\{\omega\in\tor\:|\: \ell\omega \in \tor\setminus\calG\}| \le (N+1) |\tor\setminus\calG|\le 8 N^{\frac52}  |\calS|^{\frac12}
}
On the other hand, where $\{x\}=x-\lfloor x\rfloor$ for $x>0$ denotes the fractional part, 
\EQ{
\label{eq:steep lines*}
& |\{ \omega\in\tor\:|\: (\omega,\ell\omega)\in \calS_{*}\! \mod\Z^2\text{\ for some\ }\ell \in [N,2N] \} | \\
&\le \sum_{\ell=N}^{2N} \int_0^1 \one_{\calS_{*}} (\omega,\{ \ell\omega\}) \, d\omega  = 
\sum_{\ell=N}^{2N} \frac{1}{\ell} \sum_{k=0}^{\ell-1} \int_0^1 \one_{\calS_{*}} ((x+k)/\ell,x ) \, dx \\
&\le  \sum_{\alpha=1}^M \frac{1}{N }  \int_\calG  \sum_{\ell=N}^{2N}  \sum_{k=0}^{\ell-1} \one_{I_\alpha(x)} ((x+k)/\ell) \, dx 
}
Here, for $x\in\calG$,  $(\calS_{*})_x= \calS_x=\bigcup_{\alpha=1}^M  I_\alpha(x)$ with $I_\alpha(x)$ intervals of length $|I_\alpha(x)|\le \gamma$, possibly empty. 
We claim that for all  $x\in\calG$ one has
\EQ{\label{eq:1 Claim}
\sum_{\ell=N}^{2N}  \sum_{k=0}^{\ell-1} \one_{I_\alpha(x)} ((x+k)/\ell) \le 1
}
Indeed, suppose $\ell\ne \ell'$ both in $[N,2N]$ and  
$
\frac{x+k}{\ell} ,  \frac{x+k'}{\ell'} \in I_\alpha(x) 
$.  Then 
\[
\Big|\frac{x+k}{\ell}  - \frac{x+k'}{\ell'} \Big| \le |I_\alpha(x)| 
\]
whence $|x(\ell-\ell') + k\ell'-k'\ell|\le \ell\ell'|I_\alpha(x)|$ and thus $\| j x\|\le 4N^2\gamma$ for some $1\le j\le N$. But this is excluded by $x$ being in the good set. 
So it follows that $\ell=\ell'$, which implies that for $k\ne k'$ 
\[
\gamma \ge |I_\alpha(x)| \ge \Big|\frac{x+k}{\ell}  - \frac{x+k'}{\ell} \Big| \ge \frac{1}{\ell}\ge \frac{1}{N}
\]
which contradicts that  $N^2\gamma^2=  (4N)^{-1}|\calS|  <1$. So Claim~\eqref{eq:1 Claim} is correct, and the entire contribution to~\eqref{eq:steep lines*} is at most $M/N$. 
\end{proof}
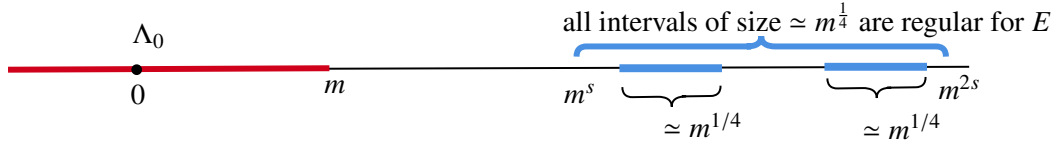
\begin{figure}[ht]
\tikzset{every picture/.style={line width=0.75pt}} 

\begin{tikzpicture}[x=0.55pt,y=0.55pt,yscale=-1,xscale=1]

\draw    (-0.49,251.17) -- (660.51,249.17) ;
\draw [color={rgb, 255:red, 208; green, 2; blue, 27 }  ,draw opacity=1 ][line width=2.25]    (-0.49,251.17) -- (220.51,250.17) ;
\draw [color={rgb, 255:red, 74; green, 144; blue, 226 }  ,draw opacity=1 ][line width=3]    (420,250.17) -- (490.51,250.17) ;
\draw  [color={rgb, 255:red, 74; green, 144; blue, 226 }  ,draw opacity=1 ][line width=2.25]  (644.51,242.17) .. controls (644.51,237.5) and (642.18,235.17) .. (637.51,235.17) -- (527.51,235.17) .. controls (520.84,235.17) and (517.51,232.84) .. (517.51,228.17) .. controls (517.51,232.84) and (514.18,235.17) .. (507.51,235.17)(510.51,235.17) -- (397.51,235.17) .. controls (392.84,235.17) and (390.51,237.5) .. (390.51,242.17) ;
\draw [color={rgb, 255:red, 74; green, 144; blue, 226 }  ,draw opacity=1 ][line width=3]    (561,249.17) -- (631.51,249.17) ;
\draw  [fill={rgb, 255:red, 0; green, 0; blue, 0 }  ,fill opacity=1 ] (85,251.17) .. controls (85,249.51) and (86.34,248.17) .. (88,248.17) .. controls (89.66,248.17) and (91,249.51) .. (91,251.17) .. controls (91,252.83) and (89.66,254.17) .. (88,254.17) .. controls (86.34,254.17) and (85,252.83) .. (85,251.17) -- cycle ;
\draw   (422,263.17) .. controls (422.07,267.84) and (424.43,270.14) .. (429.1,270.07) -- (445.86,269.82) .. controls (452.53,269.72) and (455.89,272) .. (455.96,276.67) .. controls (455.89,272) and (459.19,269.62) .. (465.86,269.52)(462.86,269.56) -- (482.62,269.27) .. controls (487.29,269.2) and (489.58,266.84) .. (489.51,262.17) ;
\draw   (563,257.17) .. controls (563.07,261.84) and (565.43,264.14) .. (570.1,264.07) -- (586.86,263.82) .. controls (593.53,263.72) and (596.89,266) .. (596.96,270.67) .. controls (596.89,266) and (600.19,263.62) .. (606.86,263.52)(603.86,263.56) -- (623.62,263.27) .. controls (628.29,263.2) and (630.58,260.84) .. (630.51,256.17) ;

\draw (83,217.57) node [anchor=north west][inner sep=0.75pt]    {$\Lambda _{0}$};
\draw (379,202.57) node [anchor=north west][inner sep=0.75pt]    {all\ intervals\ of\ size\ $\simeq m^{\frac14}$\ are\ regular\ for\ $E$};
\draw (82,259.57) node [anchor=north west][inner sep=0.75pt]    {$0$};
\draw (380,258.57) node [anchor=north west][inner sep=0.75pt]    {$m^{s}$};
\draw (637.51,252.57) node [anchor=north west][inner sep=0.75pt]    {$m^{2s}$};
\draw (215,256.57) node [anchor=north west][inner sep=0.75pt]    {$m$};
\draw (448,278.57) node [anchor=north west][inner sep=0.75pt]    {$\simeq m^{1/4}$};
\draw (585,279.57) node [anchor=north west][inner sep=0.75pt]    {$\simeq m^{1/4}$};

\end{tikzpicture}
\caption{Absence of double resonances}\label{fig:2res BG}
\end{figure}
To obtain the {\em complexity bound} \eqref{eq:cut set}, we use semi-algebraic methods. 
A closed set $\calS\subset \R^N$ is called {\em semi-algebraic} if there are polynomials $P_j\in \R[X_1,\ldots,X_N]$, $1\le j\le s$  of degrees bounded by~$d$ so that 
\[
 \calS = \bigcup_{k} \bigcap_{j\in \calF_k} \{P_j \, \sigma_{kj}\, 0\}
\]
with $\sigma_{kj}\in\{\leq,\geq,0\}$ and $\calF_k\subset\{1,2,\ldots,s\}$. The degree of $\calS$ is bounded by $sd$ and is in fact the infimum of $sd$ over all such representations. 

One might expect to get away with more elementary arguments based on zero counts alone. Note, however,   that $E$ is projected out of in the set $\calS_n(b)$ which makes it necessary to perform quantifier elimination. In fact, we will need to use a quantitative Seidenberg-Tarski theorem to control the complexity parameter~$M$ in Lemma~\ref{lem:steep lines}. This fundamental result states that  any projection of $\calS$ onto a subspace of $\R^N$ is again semi-algebraic and the degree can only grow at a power rate (depending on $N$). 
See \cite{BPR1} and~\cite{BPR2}. 

These semi-algebraic techniques are available here since $V$ is a trigonometric polynomial although by approximation and truncation, $V$ analytic can also be handled in~\cite{BouG}. Heuristically speaking, the semi-algebraic quantitative complexity bounds  replace the explicitly imposed complexity in Theorem~\ref{thm:FSW} where exactly two monotonicity intervals of $V$ are assumed. 

We claim that $\calS_n(b)$ is contained in  
\begin{align}
\tilde \calS_n(b) &:= \Pi_{\R^2}\big \{(\omega,x,E))\in \Dio_n(b)\times\tor\times\R \:|\: \log|f_{2n+1}(n\omega,\omega,E)|\le (2n+1)L_{2n+1}(E,\omega) - n^{1/4}/2,  \nn \\
& \qquad \text{\ and\ \ }  \log|f_m(x,\omega,E)|\le mL_m(E,\omega) - m^{1/2}\text{\ for some\ } m\in [n^{1/4}/2, n^{1/4}] \big\}  \label{eq:Sn def*}
\end{align}
where $\Pi_{\R^2}$ projects on to $(\omega,x)$ and moreover, that $\tilde \calS_n(b)$ has essentially the same measure bound as $\calS_n(b)$. And conversely, 
\EQ{\nn 
\tilde \calS_n(b)  &\subset \big \{(\omega,x)\in \Dio_n(b)\times\tor \:|\: \exists E\in\R\text{\ with\ } \dist(\spec(H_{[-n,n]}(0,\omega)),E)\le e^{-n^{1/4}/4}, \text{\ and} \\
& \qquad \log|f_m(x,\omega,E)|\le mL_m(E,\omega) - m^{1/2}\text{\ for some\ } m\in [n^{1/4}/2, n^{1/4}] \big\} 
}
These relations follow from noting that 
\[
\dist(\spec(H_{[-n,n]}(0,\omega)),E) = \| (H_{[-n,n]}(0,\omega))-E)^{-1}\| \text{\ \ and\ \ } \|A\|\le \|A\|_{HS}\le \sqrt{d}\, \|A\|
\]
for any $d\times d$ matrix $A$, and using the relation \eqref{eq:GEbd**}. In particular, we obtain essentially the same estimates on their two-dimensional measure.  
The sets $\tilde \calS_n(b) := \Pi_{\R^2} \calD_n(b)$ with 
\EQ{\label{eq:Dn}
 \calD_n(b)  &=  \{(\omega,x,E))\in \Dio_n(b)\times\tor\times\R \:|\:\: |f_{2n+1}(n\omega,\omega,E)|\le \exp\big( (2n+1)L_{2n+1}(E,\omega) - n^{1/4}/2\big)\}   \\
& \cap \bigcup_{ m\in [n^{1/4}/2, n^{1/4}]}  \{ (\omega,x,E))\in \Dio_n(b)\times\tor\times\R \:|\:\: |f_m(x,\omega,E)|\le \exp\big(mL_m(E,\omega) - m^{1/2} \big)  \big\} 
}
are already quite close to our sought after polynomial description. However, a polynomial expression for the Lyapunov exponents in finite volume needs to be found. Note that
while we may pass to their infinite volume versions  due to the~\cite{GolS} rate of convergence estimate $L_m(E,\omega)-L(E,\omega)\le Cm^{-1}$, it would be counter productive to do so at this point. Rather, we will use that uniformly in $x$, 
\[
n L_n(E,\omega) = \frac{1}{n^2}\sum_{j=1}^{n^2} \log \|M_n(x+j\omega,\omega,E)\| + O((\log n)^A)
\]
This follows from $\|M_n\|\ge1$, \eqref{eq:Mnfn}, and the same arguments which we used in the proof of \eqref{eq:Lambda0 res}.  
Therefore, we can replace 
\[
 |f_m(x,\omega,E)|\le \exp(mL_m(E,\omega) - m^{1/2})
\]
with 
\[
 |f_m(x,\omega,E)|^{2m^2}\le e^{-m^{\frac52}} \prod_{j=1}^{m^2} \|M_m(j\omega,\omega,E)\|_{HS}^2
\]
This is a polynomial inequality in all variables of degree $O(m^4)=O(n)$. The set on the first line of $\calD_n$ is described by a polynomial inequality of degree $O(n^4)$. 
Since there are $\les n^{\frac14}$ polynomials involved in the description of the semi-algebraic set $\calD_n(b)$ above, it is of degree $\les n^{5}$.  Projecting out $E$, we conclude that $\tilde \calS_n(b)$ has degree $O(n^s)$ for some finite $s$ as claimed. Finally, each horizontal slice consists of at most 
$O(n^s)$ many connected components, i.e., intervals. 

\begin{proof}[Proof of Theorem~\ref{thm:BG}]
The set of admissible $\omega$ for the theorem is 
\[
\Omega := \Dio\setminus \bigcup_{j=1}^\infty \limsup_{n\to\infty} \calB_n(1/j),\quad \Dio:=\bigcup_{j=1}^\infty \Dio(1/j)
\]
where $\calB_n(b)$ is defined in \eqref{eq:calBn}. By Lemma~\ref{lem:steep lines}, 
\[
 |\calB_n(1/j)|\le C(j) (n^{-s}+n^{C} e^{-n^{1/12}}),\quad \sum_{n=1}^\infty |\calB_n(1/j)|<\infty
\]
whence by Borel-Cantelli 
$
|\limsup_{n\to\infty} \calB_n(1/j)| =0
$. Since $\Dio$ has full measure in $\tor$, so does~$\Omega$. Now freeze some $\omega\in\Omega$. Note in particular that $\omega\in \Dio(b)$ for some $b>0$ whence the LDT results all hold. 
Given a generalized eigenfunction $H(x)\psi=E\psi$ by Theorem~\ref{thm:Ber}, we showed that for all sufficiently large~$n$,
 \eqref{eq:Lambda0 res} holds for some $n\le m\le n^3$.  By definition of $\calB_n(b)$, see Figure~\ref{fig:2res BG}, we conclude that all Green functions $G_{\Lambda}(0,\omega,E)$ with $\Lambda\subset [m^s,2m^s]$ and $|\Lambda|\simeq m^{\frac14}$ satisfy
\[
\| G_{\Lambda}(0,\omega,E)\|\le e^{|\Lambda|^{\frac12}}, \quad |G_{\Lambda}(0,\omega,E)(x,y)|\le e^{-\gamma|x-y|+|\Lambda|^{\frac12}} \qquad\forall\; x,y\in\Lambda
\]
Using the resolvent identity iteratively as in Lemma~\ref{lem:ind step}, albeit with all subintervals being regular for $E$, we conclude that the Green function on the large window is also regular for~$E$:
\[
\| G_{\pm [m^s,2m^{s}]}(0,\omega,E)\|\le e^{m^{\frac18}}, \quad |G_{\pm [m^s,2m^{s}]}(0,\omega,E)(x,y)|\le e^{-\gamma|x-y|+m^{\frac18}} \qquad\forall\; x,y\in \pm [m^s,2m^{s}]
\]
from which the exponential decay of $\psi$ immediately follows. 
\end{proof}


\begin{figure}[ht]
\centering
\includegraphics[width=130mm,scale=1.8]{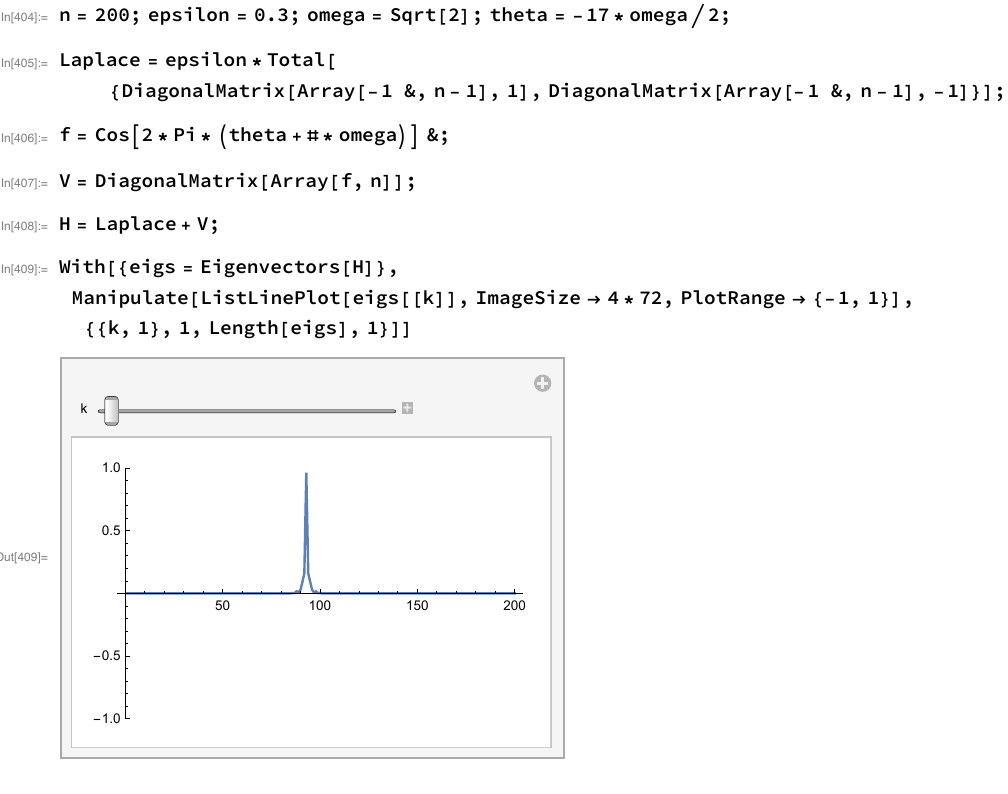}
\caption{\textsc{Mathematica} code for Figures~\eqref{fig:Eigf1}, \eqref{fig:Eigf2}}
\end{figure}

\end{document}